\documentclass{amsart}
\usepackage{amsmath,amssymb,dsfont,amsthm}
\usepackage{graphics}
\usepackage{xypic}
\usepackage{latexsym}
\usepackage{amsmath}
\usepackage{amsfonts}
\usepackage{geometry}
\usepackage{amssymb}
\usepackage{tensor}

\usepackage{lineno}

\date{\today}

\title[Analytic torsion of cones]{The analytic torsion of the cone over an odd dimensional manifold}


\author{L. Hartmann  and M. Spreafico}
\address[Luiz Hartmann]{\tt UFSCar, Universidade Federal de S\~{a}o Carlos, S\~{a}o Carlos, Brazil. \newline Partially supported by FAPESP 2009/15145-9}
\email{hartmann@dm.ufscar.br}
\address[Mauro Spreafico]{\tt ICMC, Universidade S\~{a}o Paulo, S\~{a}o Carlos, Brazil.}
\email{mauros@icmc.usp.br}

\numberwithin{equation}{section}

\newtheorem{theo}{Theorem}[section]
\newtheorem{lem}{Lemma}[section]
\newtheorem{corol}{Corollary}[section]
\newtheorem{defi}{Definition}[section]

\newtheorem{prop}{Proposition}[section]
\newtheorem{rem}{Remark}[section]

\renewcommand{\S}{\mathcal{S}}

\renewcommand{\Re}{{\rm Re}}

\newcommand{\Sp}{{\rm Sp}}
\newcommand{\beq}{\begin{equation}}
\newcommand{\eeq}{\end{equation}}

\newcommand{\Z}{{\mathds{Z}}}
\newcommand{\R}{{\mathds{R}}}
\newcommand{\C}{{\mathds{C}}}
\newcommand{\F}{{\mathds{F}}}
\newcommand{\Q}{{\mathds{Q}}}
\newcommand{\T}{{\mathcal{T}}}

\newcommand\e{{\rm e}}
\newcommand{\A}{{\mathcal{A}}}
\renewcommand{\H}{{\mathcal{H}}}

\renewcommand{\P}{{\mathcal{P}}}
\newcommand{\B}{{\mathcal{B}}}

\newcommand{\End}{{\rm End}}
\newcommand{\Hom}{{\rm Hom}}

\renewcommand{\b}{{\partial}}

\newcommand{\ec}{{\mathsf e}}
\renewcommand{\ge}{{\mathsf g}}

\renewcommand{\det}{{\rm det}}

\newcommand{\E}{{\mathcal{E}}}

\newcommand{\tr}{C_{[l_1,l_2]} W}

\date{}

\DeclareMathOperator*{\Rz}{Res_0}

\DeclareMathOperator*{\Ru}{Res_1}

\begin{document}



\maketitle

\begin{abstract}  We study the analytic torsion of the cone over an orientable odd dimensional compact connected Riemannian manifold $W$. We prove that the logarithm of the analytic torsion of the cone decomposes as the sum of the logarithm of the root of the analytic torsion of the boundary of the cone, plus a topological term, plus a further term that is a rational linear combination of local Riemannian invariants of the boundary. We also prove that this last term coincides with the anomaly boundary term appearing in the Cheeger M\"uller theorem \cite{Che1} \cite{Mul} for a manifold with boundary, according to Br\"uning and Ma \cite{BM}. 
We also prove  Poincar\'e duality for the analytic torsion of a cone.
 
\end{abstract}

\tableofcontents

\section{Introduction and statement of the results}
\label{s0}

Analytic torsion was originally introduced by Ray and Singer in \cite{RS}, as an analytic counter part of the Reidemeister torsion of Reidemeister, Franz and de Rham \cite{Rei} \cite{Fra} \cite{deR}. Since then, analytic torsion became an important invariant of Riemannian manifolds, and has been intensively studied, several generalizations have been introduced and significative  results obtained. Concerning the original invariant, the first important result was achieved by W. M\"uller \cite{Mul} and J. Cheeger \cite{Che1}  who proved that for a compact connected Riemannian manifold without boundary, the analytic torsion and the Reidemeister torsion coincide, result conjectured by Ray and Singer in \cite{RS}, because of the several similar properties shared by the two torsions. This result is nowadays known as the celebrated Cheeger-M\"uller theorem. The next natural question along this line of investigation was to answer the same problem for manifolds with boundary. It was soon realized that the answer to such a question was an highly non trivial one, since the boundary introduces some wild terms. The first case to be analyzed was the case of a product metric near the boundary. W. L\"uck proved in \cite{Luc} that in this case the boundary term is topological, and depends only upon the Euler characteristic of the boundary. The answer to the general case required 20 more years of work, and has eventually been given in a recent  paper of Br\"uning and Ma \cite{BM}. The new contribution of the boundary, beside the topological one given by L\"uck, is called anomaly boundary term and we denote it by $A_{\rm BM,abs}$. The term $A_{\rm BM,abs}$ has  a quite complicate expression, but only depends on some local quantities constructed from the metric tensor near the boundary (see Section \ref{cm} for details). The next natural step is to study the analytic torsion for spaces with singularities. A first, natural type of space with singularity is the cone over a manifold, $CW$. Cones and spaces with conical singularities have been deeply investigated by J. Cheeger in a series of works \cite{Che1} \cite{Che2} \cite{Che3} (see also \cite{Nag}).  Due to this investigation, all information on $L^2$-forms, Hodge theory, and Laplace operator on forms on $CW$ are available. Further information on the class of regular singular operators, that contains the Laplace operator on $CW$, are given in works of Br\"uning and Seely (see in particular \cite{BS2}). As a result it is not difficult to obtain a complete description of the eigenvalues of the Laplace operator on $CW$ in terms of the eigenvalues of the Laplace operator on $W$. With all these tools available, namely on one side the formula for the boundary term, and on the other some representation of the eigenvalues of the Laplace operator on the cone, it is natural to tackle the problem of investigating the analytic torsion of $CW$. What one expects in this case in fact is some relation between the torsion of the cone and the torsion of the section. A possible extension of the Cheeger M\"uller theorem could follow, or not. Indeed, it is general belief that in case of conical singularity such an extension would require intersection R torsion more than classical R torsion (see \cite{Dar}). However, if the section is a rational homology manifold, then the two torsion coincide (see \cite{Che1}, end of Section 2), and the classical Cheeger M\"uller theorem is expected to extend.

If $\mathcal{C}(W)$ is the chain complex associated to some cell decomposition of $W$, then the algebraic mapping cone $Cone(\mathcal{C}(W))$ gives the chain complex for a cell decomposition of  $CW$. It is then easy to see that the R torsion of $CW$ only depends on the choice of a base for the zero dimensional homology. Even if Poincar\'e duality does not hold, it does hold between top and bottom dimension, and therefore we can use the method of Ray and Singer in order to fix the base for the zero homology using the Riemannian structure and harmonic forms (see \cite{RS} Section 3, see also \cite{HMS}). The result for the R torsion is
\[
\tau(CW)=\sqrt{\rm Vol(CW)}.
\]

On the other side, one wants the analytic torsion. The analytic tools necessary to deal with the zeta functions appearing in the definition of the analytic torsion, constructed with the eigenvalues of the Laplace operator on $CW$, are  available by works of M. Spreafico  \cite{Spr4} \cite{Spr5} \cite{Spr9}. In these works, the zeta function associated to a general class of
double sequences is investigated. In particular, a decomposition result is presented and formulas for the zeta invariants
of a decomposable sequence are given. This technique applies to the case of the zeta function on $CW$, and permits to obtain some results on the analytic torsion that we will describe here below. Before, we note that this approach has been also followed in \cite{HMS}, \cite{Ver} and \cite{HS}. The main results of these papers are that in the case of $W$ an odd low dimensional sphere, then the classical Cheeger M\"uller theorem with the anomalous boundary term of Br\"uning and Ma holds for $CW$, while if $W$ is an even dimensional sphere, it does not hold. 
Explicit formulas for $W$ the sphere of dimensions one, two and three are also given, and in \cite{HS} it is conjectured that the Cheeger M\"uller theorem with the anomalous boundary term of Br\"uning and Ma holds for $W$ any odd dimensional sphere. This is proved to be true in Theorem \ref{t02} below. In Theorem \ref{t03} the same result is proved for the cone over a generic orientable compact connected Reimannian manifold of  odd dimension. We split the proof of  Theorem \ref{t03} in two parts. We first prove in Section \ref{s7} that the result is true if the dimension of the manifold is smaller than six. A basic ingredient in this proof are  results of P. Gilkey on local invariants of Riemannan manifolds \cite{Gil}. Next, in Section \ref{ultima}, we prove the general statement. The reason for giving a different explicit proof for the low dimensional cases is due to the fact that the proof of the general case is based on a result that has not been published yet, namely a formula for the anomaly boundary term of Br\"{u}ning and Ma with mixed boundary conditions. This result is contained in a preprint of the same authors \cite{BM2}, and we thanks the author for making available part of their work.

We are now ready to state the main results of this paper, for we fix some notation. Let $(W,g)$ be an orientable  compact connected Riemannian manifold of finite dimension $m$ without boundary  and with Riemannian structure $g$. We denote by $C_lW$ the cone over $W$ with the Riemannian structure
\[
dx\otimes dx+x^2 g,
\]
on $CW-\{pt\}$, where $pt$ denotes the tip of the cone and $0<x\leq l$ (see Section \ref{Lap1.1} for details). The formal Laplace operator on forms on $C_ W-\{pt\}$ has a suitable $L^2$-self adjoint extension $\Delta_{\rm abs/rel}$ on $C_l W$ with absolute or relative boundary conditions on the boundary $\b C_l W$ (see Section \ref{Lap1.3} for details), with pure discrete spectrum $\Sp \Delta_{\rm abs/rel}$. This permit to define the associated zeta function 
\[
\zeta(s, \Delta_{\rm abs/rel})=\sum_{\lambda\in \Sp_+ \Delta_{\rm abs/rel}} \lambda^{-s},
\]
for $\Re(s)>\frac{m+1}{2}$. This zeta function has a meromorphic analytic continuation to the whole complex $s$-plane with at most isolated poles (see Section \ref{s2c} for details). It is then possible to define the analytic torsion of the cone
\[
\log T_{\rm abs/rel}(C_l W)=\frac{1}{2}\sum_{q=0}^{m+1}(-1)^q q\zeta'(0,\Delta^{(q)}_{\rm abs/rel}).
\]



In this setting,  we have the following results (analogous results with relative boundary conditions also follow by Poincar\'e duality on the cone, proved in Theorem \ref{Poinc2} below).

\begin{theo}\label{t01} The analytic torsion on the cone $C_l W$ on an orientable compact connected Riemannian manifold $(W,g)$ of odd dimension $2p-1$ is
\begin{align*}
\log T_{\rm abs}(C_lW)= &\frac{1}{2} \sum_{q=0}^{p-1} (-1)^{q+1} {\rm rk}H_q(W;\Q)\log \frac{2(p-q)}{l}+\frac{1}{2} \log T(W,l^2g)+{\rm S}(\b C_l W),
\end{align*}
where the singular term ${\rm S}(\b C_l W)$ only depends on the boundary of the cone:
\[
{\rm S}(\b C_l W)=\frac{1}{2}\sum_{q=0}^{p-1}  \sum_{j=0}^{p-1}\sum^{j}_{k=0} \Rz_{s=0}\Phi_{2k+1,q}(s)
\binom{-\frac{1}{2}-k}{j-k} \sum^{q}_{h=0}(-1)^{h}\Ru_{s=j+\frac{1}{2}}\zeta\left(s,\tilde \Delta^{(h)}\right)(q-p+1)^{2(j-k)},
\]
where the functions $\Phi_{2k+1,q}(s)$ are some universal functions, explicitly known by some recursive relations, and $\tilde\Delta$ is the Laplace operator on forms on the section of the cone.
\end{theo}

It is important to observe that the singular term ${\rm S}(\b C_l W)$ is a universal linear combination of local Riemannian invariants of the boundary, for the residues of the zeta function of the section are such linear combination (see Section \ref{gene} for details).

\begin{theo}\label{t02} When $W$ is the odd dimensional sphere (of radius $a$), with the standard induced Euclidean metric, then the singular term of the analytic torsion of the cone $C_l W$ appearing in Theorem \ref{t01} coincides with the anomaly boundary term of Br\"uning and Ma, namely ${\rm S}(\b C_l S^{2p-1}_{a})=A_{\rm BM,abs}(\b C_lS^{2p-1}_{a})$. In this case, the formula for the analytic torsion reads
\begin{align*}
\log T_{\rm abs}(C_lS^{2p-1}_{a})= \frac{1}{2} \log {\rm Vol} (C_l S^{2p-1}_{a})+ A_{\rm BM,abs}(\b C_lS^{2p-1}_{a}),
\end{align*}
where
\begin{align*}
 A_{\rm BM,abs}(\b C_l S^{2p-1}_{a})=&\frac{(2p-1)!}{4^p (p-1)!}\sum_{k=0}^{p-1} \frac{1}{(p-1-k)!(2k+1)} \sum^{k}_{j=0}
\frac{(-1)^{k-j}2^{j+1}}{(k-j)!(2j+1)!!}a^{2k+1}.
\end{align*}

\end{theo}

\begin{corol} The natural extension of Cheeger M\"uller theorem for manifold with boundary is valid for  the cone over an odd dimensional sphere, namely
\[
\log T_{\rm abs}(C_l S^{2p-1})=\log \tau(C_l S^{2p-1})+A_{\rm BM,abs} (\b C_l S^{2p-1}).
\]

\end{corol}

The result in the corollary should be understood as a particular case of the still unproved general result that the analytic torsion and the intersection R-torsion of a cone coincides up to the boundary term, for the intersection R-torsion is the classical R-torsion for the cone over a sphere.

\begin{theo}\label{t03} When $(W,g)$ is an orientable compact connected Riemannian manifold of odd dimension, then the singular term of the analytic torsion of the cone $C_l W$ appearing in Theorem \ref{t01} coincides with the anomaly boundary term of Br\"uning and Ma, namely ${\rm S}(\b C_l W)=A_{\rm BM,abs}(\b C_l W)$.
\end{theo}

We conclude with a remark on the even dimensional case, namely when the dimension of the section $W$ is even. It is quit clear that all the arguments used in the odd dimensional case go through also in the even dimensional case. So we obtain formulas for the analytic torsion as in the theorems above. However, in the even dimensional case some further term appears: this was described in some details for $W=S^2$ in \cite{HMS}. Since we do not have a clear understanding of this new term yet, we prefer to omit  the  non particularly illuminating formulas for the even dimensional case here.

\section{Preliminaries and notation}
\label{s1}

In this section we will recall some basic results in  Riemannian geometry, Hodge de Rham theory and global analysis, and the definitions of the main objects we will deal with in this work. All the results are contained either in classical literature or in \cite{RS}, \cite{Mul}, \cite{Che2}, \cite{BM}, \cite{Gil} \cite{BGV}. This section can be skipped at first reading, and it is added exclusively for the reader's benefit. From one side, it permits to the interested reader to find out all the necessary tools to verify the arguments developed in the following sections, with precise reference to definitions and formulas, avoiding the fuss of searching in several different papers and books, from the other, it provides a unified notation, whereas notations of different authors are too often quite different in the available literature.

\subsection{$\Z/2$-graded algebras and Berezin integrals}
\label{ppp}

Let  $V,W$ be  finite dimensional vector spaces over a field $\F$ of characteristic zero, with Euclidean inner products $\langle\_,\_\rangle$. Let  $V^*=\Hom(V,\F)$ denotes the dual of $V$, and fix an isomorphism of $V$ onto $V^*$ by $v^*(u)=\langle v,u\rangle$.  Then, we identify $\F$-homomorphisms with tensor product of tensors by $\Hom(V,W)= V^*\otimes W$. If  $V,W$ have dimensions $m, n$ respectively, and  $\{e_k\}_{k=1}^m$, $\{b_l\}_{l=1}^n$ are orthonormal bases of $V$ and $W$ respectively, then we identify  $t\in \Hom(V,W)$ with the tensor $T\in V^*\otimes W$ with components $T_{kl}=\langle b_l,t(e_k)\rangle=b^*(t(e_k))$, namely
\[
T=\sum_{k=1}^m\sum_{l=1}^n T_{kl} e^*_k\otimes b_l,
\]
where we denote by $\{e_k^*\}$ the dual base. 
We use the {\it Euclidean norm} of a linear homomorphism $|t|=\sqrt{{\rm tr} (t^* t)}$. In the orthonormal base, this gives for the associated tensor
\[
|T|^2=\sum_{k=1}^m\sum_{l=1}^n T^2_{kl}.
\]

We denotes by $\Lambda V$ the exterior algebra of $V$ (the universal algebra with unit generated by $V$ with the relation $v^2=0$).  A {\it $\Z/2$-graded algebra} $A$ is a vector space with an involutions such that $A=A_+\oplus A_-$, and a product $\cdot$ that preserves the involution, i.e. such that $A_j A_k\subset A_{jk}$.  The  exterior algebra is a first example. Let $A$ and $B$ be two $\Z/2$-graded algebras. The vector spaces tensor product $A\otimes B$ has a natural  $\Z/2$-grading, $A\otimes B=(A\otimes B)_+\oplus (A\otimes B)_-$, where $(A\otimes B)_+=A_+\otimes B_+\oplus A_-\otimes B_-$, and $(A\otimes B)_-=A_+\otimes B_-\oplus A_-\otimes B_+$. This becomes a $\Z/2$-graded algebra, that we denote by $A\hat\otimes B$, with the product defined by (we will omit the dot in the following)
\[
(a\hat\otimes b)\cdot (a'\hat\otimes b')=(-1)^{|b|\, |a'|} a\cdot a'\hat\otimes b\cdot b'.
\]

There are two natural immersions of $A$ in $A\hat\otimes A$ as an algebra: we identify  $A$ with $A\otimes 1$ and we denote by $\hat A=1\otimes A$. 
Since $A\otimes 1\otimes 1\otimes B=A\otimes B$, we have that $A\hat\otimes B=A\otimes\hat B$. 
Let $A=B=\Lambda V$, for some vector space $V$, where we denote the product by $\wedge$, as usual. Then,  $V\hat\otimes V\in \Lambda V\hat\otimes\Lambda V$. If  $v\in V$, then $v=v\hat\otimes 1\in V=V\hat\otimes 1$, $\hat v=1\hat\otimes v\in \hat V=1\hat\otimes V$, and $v\wedge \hat v\in V\wedge \hat V$. Note that, $v\wedge \hat v=(v\hat\otimes 1)\wedge (1\hat\otimes v)$, while $(1\hat\otimes v)\wedge (v\hat\otimes 1)=-v\wedge \hat v$, and 
\[
\hat u\wedge \hat v=(1\hat\otimes u)\wedge(1\hat\otimes v)=1\hat\otimes (u\wedge v)=\widehat{u\wedge v}\in \widehat{V\wedge V}.
\]

This permits to identify an antisymmetric  endomorphism $\phi$ of $V$ with the element 
\[
\hat \phi=\frac{1}{2}\sum_{j,k=1}^m \langle\phi(v_j),v_k\rangle \hat v_j\wedge \hat v_k,
\]
of $\widehat{\Lambda^2 V}$. For the elements $\langle\phi(v_j),v_k\rangle$ are the entries of the tensor representing $\phi$ in the base $\{v_k\}$, and this is an antisymmetric matrix. Now assume that $r$ is an antisymmetric endomorphism of $\Lambda^2 V$. Then, $(R_{jk}=\langle r(v_j),v_k\rangle)$ is a tensor of two forms in $\Lambda^2 V$. We extend the above construction identifying $R$ with the element
\begin{equation}\label{proap}
\hat R=\frac{1}{2}\sum_{j,k=1}^m \langle r(v_j),v_k\rangle\wedge \hat v_j\wedge \hat v_k,
\eeq
of $\Lambda^2 V\wedge \widehat{\Lambda^2 V}$. This can be generalized to higher dimensions. In particular,
all the construction can be done taking the dual $V^*$ instead of $V$. We conclude with the definition of the Berezin integral. Assume  $V$ to be oriented. We define the {\it Berezin integral} be
\begin{align*}
\int^B&:\Lambda W\hat\otimes \Lambda V\to \Lambda W,\\
\int^B&:\alpha\otimes \beta\mapsto c_B\beta(e_1,\dots, e_m)\alpha,
\end{align*}
where $c_B=\frac{(-1)^\frac{m(m+1)}{2}}{\pi^\frac{m}{2}}$. In particular, note that 
\[
\int^B \e^{-\frac{\hat A}{2}}=Pf\left(\frac{A}{2\pi}\right), 
\]
and this vanishes if ${\rm dim}E=m$ is odd.


\subsection{Some Riemannian geometry}
\label{rim}

Let $(W,g)$ be an orientable connected Riemannian manifold of dimension $m$ without boundary, where $g$ denotes the Riemannian structure. We denote by $TW$ the (total space) of the tangent bundle over $W$, and by $T^* W$ the dual bundle. We denote by $\Lambda T^* W$ the exterior algebra of $T^* W$. We denote by $\Gamma(W,TW)$ and $\Gamma(W,T^*W)$ the corresponding spaces of smooth sections, and by $\Gamma_0(W,TW)$ and $\Gamma_0(W,T^*W)$ the spaces of smooth sections with compact support. Let $x=(x_1,\dots, x_m)$ be a local coordinate system on $W$. We denote by $\{\b_j\}_{j=1}^m$ the local coordinate base on $TW$, and by $\{d x_j\}_{j=1}^m$ the dual base of one forms on $T^* W$: $dx_j(\b_k)=\delta_{jk}$. We denote by $\{e_j\}_{j=1}^m$ and $\{e^*_j\}_{j=1}^m $ local orthonormal bases of $TW$ and $T^*W$, $e^*_j(e_k)=\delta_{jk}$. Then, 
\begin{align*}
g&=\sum_{j,k=1}^m g_{jk}dx_j\otimes dx_k=\sum_{j,k=1}^m \delta_{jk}e^*_j\otimes e^*_k,
\end{align*}
and the  {\it volume element} on $W$ is  
$dvol_g=e^*_1\wedge \dots \wedge e^*_m=\sqrt{|\det g|} dx_1\wedge\dots \wedge dx_m$.

Let $\nabla:\Gamma(W,TW)\to \Gamma(W,\End(T W))$ denotes the {\it covariant derivative} associated to the Levi Civita connection of the metric $g$. This is completely determined by the tensor $\Gamma=\sum_{\alpha\beta\gamma=1}^m \Gamma_{\alpha\beta\gamma} e_\alpha^*\otimes e_\beta^*\otimes e_\gamma,
\in \Gamma(W,T^*W\otimes T^* W\otimes T W)$, with components the   {\it Christoffel symbols}:
\[
\Gamma_{\alpha\beta\gamma}=\Gamma(e_\alpha, e_\beta, e_\gamma^*)=e_\gamma^*(\Gamma(e_\alpha, e_\beta, \_))
=e_\gamma^*(\nabla_{e_\alpha}e_\beta)=g(e_\gamma, \nabla_{e_\alpha}e_\beta)=-\Gamma_{\alpha\gamma\beta}, 
\]
that can be computed using the formula:
\beq\label{gam1}
\Gamma_{\alpha \beta\gamma}=\frac{1}{2}\left(c_{\alpha\beta\gamma}+c_{\gamma\alpha\beta}+c_{\gamma \beta\alpha}\right),
\eeq
where the {\it Cartan structure constants} $c_{\alpha\beta\gamma}=-c_{\beta\alpha\gamma}$  are defined by
\[
[e_\alpha,e_\beta]=\sum_{\gamma=1}^m  c_{\alpha \beta\gamma}e_\gamma.
\]

The {\it connection one form} associated to $\nabla$ is the matrix valued one form $\omega\in \mathtt{so}(m,\Gamma(W, \Lambda T^*W))$, with components
\[
\omega_{\beta\gamma}=\sum_{\alpha=1}^m  \Gamma ( e_\alpha, e_\gamma,e_\beta^*)e^*_\alpha.
\]

The {\it curvature } associated to the Riemannian connection $\nabla$ (of the metric $g$)  is the linear map $r:\Gamma(W,TW\otimes TW)\to \Gamma(W,\End(T W))$, defined by:
\[
 r(x,y,z)=\nabla_x\nabla_y z -\nabla_y\nabla_x z- \nabla_{[x,y]} z,
\]
where $x,y,z,  r(x,y,z)\in \Gamma(W,TW)$, and corresponds to the tensor    $R=\sum_{\alpha,\beta,\gamma,\delta=1}^m R_{\alpha\beta\gamma\delta}  e_\alpha^*\otimes e_\beta^*\otimes e_\gamma^*\otimes e_\delta,
\in \Gamma(W,T^* W\otimes T^* W\otimes T^*W\otimes TW)$, with components
\begin{align*}
 R_{\alpha\beta\gamma\delta}=-R_{\beta\alpha\gamma\delta}&=R(e_\alpha,e_\beta,e_\gamma,e_\delta^*)=e^*_\delta(r(e_\alpha,e_\beta,e_\gamma))\\
 &=e_\delta^*(\nabla_{e_\alpha}\nabla_{e_\beta} e_\gamma -\nabla_{e_\beta}\nabla_{e_\alpha} e_\gamma -\nabla_{[e_\alpha,e_\beta]} e_\gamma).
\end{align*}

The {\it curvature two}  form associated to $R$ is the  matrix valued  two form $\Omega\in \mathtt{so}(m,\Gamma(W, \Lambda^2 T^*W))$,  with components
\[
\Omega_{\gamma\delta}=\sum_{\alpha,\beta=1,\alpha<\beta}^m  R( e_\alpha, e_\beta,e_\delta,e_\gamma^*)e^*_\alpha\wedge e^*_\beta,
\]
and can be computed by the following formula:
\beq\label{cur2}
\Omega_{\alpha\beta}=d\omega_{\alpha\beta}+\sum_{\gamma=1}^m \omega_{\alpha\gamma}\wedge \omega_{\gamma\beta}.
\eeq

We introduce two more tensor fields. The Ricci tensor $Ric=\sum_{\alpha,\beta=1}^m Ric_{\alpha\beta} e_\alpha^*\otimes e_\beta^* \in \Gamma(W,T^* W\otimes T^* W)$, defined by 
\[
Ric(x,y)=\sum_{k=1}^m e_k^*(\nabla_{e_k}\nabla_x y -\nabla_x\nabla_{e_k} y- \nabla_{[e_k,x]} y),
\]
and the scalar curvature tensor, defined by
\[
\tau=\sum_{k=1}^m Ric(e_k,e_k).
\]

The components of these tensors in terms of the curvature tensor are:
\begin{align*}
Ric_{\alpha\beta}&=\sum_{k=1}^m R(e_k,e_\alpha,e_\beta,e_k^*),\\
\tau&=\sum_{k,h=1}^m R(e_k,e_h,e_h,e_k^*).
\end{align*}

\subsection{Hodge theory and  de Rham complex}
\label{hodg}

We  recall some  results on the de Rham complex (see for example \cite{Mor})  and some results from \cite{RS} and \cite{Mul}. From now one we will assume  that $W$ is compact. In this section we  also assume that $W$ has no boundary.

Let denote by $\Omega^{q}$ the space of sections $\Gamma(W,\Lambda^{(q)} T^*W)$. The exterior differential $d$ defines the {\it de Rham complex}
\[
\xymatrix{
\mathcal{C}_{DR}: \dots \ar[r]&\Omega^{(q)}(W)\ar[r]^d &\Omega^{(q+1)}(W)\ar[r]&\dots,
}
\]
whose homology coincides with the rational homology of $W$. The Hodge star $\star:\Lambda^{(q)} T^* W\to\Lambda^{(m-q)} T^* W$, defines an isometry $\star:\Omega^{q}(W)\to\Omega^{m-q}(W)$, and an inner product on $\Omega^{q}(W)$
\[
(\alpha,\beta)=\int_W \alpha\wedge \star \beta=\int_W \langle\alpha,\beta\rangle dvol_g.
\]

The closure of  $\Omega^{q}$ with respect to this inner product is the  Hilbert space the $L^2$ $q$-forms on $W$. The de Rham complex with this product is an elliptic complex. The dual of the exterior derivative $d^\dagger$, defined by $(\alpha, d\beta)=(d^\dagger \alpha, \beta)$, satisfies $d^\dagger=(-1)^{mq+m+1} \star d\star$. The Laplace operator is $\Delta=(d+d^\dagger)^2$. It satisfies:   1) $\star\Delta=\Delta\star$, 2) $\Delta$ is self adjoint, and 3)
 $\Delta\omega=0$ if and only if $d\omega=d^\dagger \omega=0$. Let $\H^q(W)=\{\omega\in \Omega^{(q)}(W)~|~\Delta\omega=0\}$, be the space of the $q$-harmonic forms. Then, we have the {\it Hodge decomposition}
\[
\Omega^{q}(W)=\H^q(W)\oplus d\Omega^{q-1}(W)\oplus d^\dagger \Omega^{q+1}(W).
\]

All this  generalizes considering a bundle over $W$. In particular, we are interested in the following situation. Let $\rho:\pi_1(W)\to O(k,\F)$ be a representation of the fundamental group of $W$, and let  $E_\rho$ be the associated vector bundle over $W$ with fibre $\F^k$ and group $O(k,\F)$, $E_\rho=\widetilde W\times_\rho \F^k$. Then, we denote by 
$\Omega(W,E_\rho)$ be the graded linear space of smooth forms on $W$ with values in $E_\rho$, namely $\Omega(W,E_\rho)=\Omega(W)\otimes E_\rho$.  The exterior differential on $W$ defines the exterior differential on $\Omega^q(W, E_\rho)$, $d:\Omega^q(W, E_\rho)\to
\Omega^{q+1}(W, E_\rho)$. The metric $g$ defines an Hodge operator on $W$ and hence on $\Omega^q(W, E_\rho)$, $\star:\Omega^q(W, E_\rho)\to
\Omega^{m-q}(W, E_\rho)$, and,  using the inner product $\langle\_,\_\rangle$ in $E_\rho$,
an inner product on $\Omega^q(W, E_\rho)$ is defined by
\beq\label{inner}
(\omega,\eta)=\int_W \langle \omega\wedge\star\eta \rangle.
\eeq

It is clear that the adjoint $d^\dagger$ and the Laplacian $\Delta=(d+d^\dagger)^2$ are also defined on the spaces of sections with values in $E_\rho$. Setting $\H^q(W,E_\rho)=\{\omega\in \Omega^{(q)}(W,E_\rho)~|~\Delta\omega=0\}$, be the space of the $q$-harmonic forms, we have the Hodge decomposition
\beq\label{hodge1}
\Omega^{q}(W,E_\rho)=\H^q(W,E_\rho)\oplus d\Omega^{q-1}(W,E_\rho)\oplus d^\dagger \Omega^{q+1}(W,E_\rho).
\eeq

This  induces a decomposition of the eigenspace of a given eigenvalue $\lambda\not=0$ of $\Delta^{(q)}$ into the spaces of {\it closed forms} and {\it coclosed forms}: $\E^{(q)}_\lambda=\E^{(q)}_{\lambda, {\rm cl}}\oplus \E^{(q)}_{\lambda,{\rm ccl}}$, where 
\begin{align*}
\E^{(q)}_{\lambda,{\rm cl}}&=\{\omega\in\Omega^{q}(W,E_\rho)~|~\Delta\omega=\lambda\omega, \, d\omega=0\},\\
\E^{(q)}_{\lambda,{\rm ccl}}&=\{\omega\in\Omega^{q}(W,E_\rho)~|~\Delta\omega=\lambda\omega, \, d^\dagger\omega=0\}.
\end{align*}

Defining {\it exact forms} and {\it coexact forms} by 
\begin{align*}
\E^{(q)}_{\lambda,{\rm ex}}&=\{\omega\in\Omega^{q}(W,E_\rho)~|~\Delta\omega=\lambda\omega, \, \omega=d\alpha\},\\
\E^{(q)}_{\lambda,{\rm cex}}&=\{\omega\in\Omega^{q}(W,E_\rho)~|~\Delta\omega=\lambda\omega, \, \omega=d^\dagger\alpha\}.
\end{align*}

Note that, if $\lambda\not=0$, then  $\E^{(q)}_{\lambda,{\rm cl}}=\E^{(q)}_{\lambda,{\rm ex}}$, and $\E^{(q)}_{\lambda,{\rm ccl}}=\E^{(q)}_{\lambda,{\rm cex}}$, 
and we have an  isometry
\beq\label{iso1}\begin{aligned}
\phi:&\E^{(q)}_{\lambda,{\rm cl}}\to\E^{(q-1)}_{\lambda,{\rm cex}},\\
\phi:&\omega\mapsto \frac{1}{\sqrt{\lambda}}d^\dagger \omega,
\end{aligned}
\eeq
whose inverse is $\frac{1}{\sqrt{\lambda}}d$. Also, the restriction of the Hodge star defines an isometry
\begin{align*}
\star:& d^\dagger \Omega^{(q+1)}(W)\to d\Omega^{(m-q-1)}(W),
\end{align*}
and that composed with the previous one gives the isometries:
\beq\label{it}
\begin{aligned}
\frac{1}{\sqrt{\lambda}}d\star &:\E^{(q)}_{\lambda,{\rm cl}}\to\E^{(m-q+1)}_{\lambda,{\rm cex}},\\
\frac{1}{\sqrt{\lambda}}d^\dagger\star&:\E^{(q)}_{\lambda,{\rm ccl}}\to\E^{(m-q-1)}_{\lambda,{\rm ex}}.
\end{aligned}
\eeq

\subsection{Manifolds with boundary}
\label{bord}

Next consider a manifold with boundary. If $W$ has a boundary $\b W$, then there is a natural splitting near the boundary, of $\Lambda W$ as direct sum of vector bundles $\Lambda T^*\b W\oplus N^* W$, where $N^*W$ is the dual to the normal bundle to the boundary. 
Locally, this reads as follows. Let $\b_x$ denotes the outward pointing unit normal vector to the boundary, and $dx$ the corresponding one form. Near the boundary we have the collar decomposition
$Coll(\b W)=(-\epsilon,0]\times \b W$, and if $y$ is a system of local coordinates on the boundary, then $(x,y)$ is a local system of coordinates in $Coll(\b W)$. The metric tensor decomposes near the boundary in this local system as
\[
g=dx\otimes dx+  g_\b(x),
\]
where $ g_\b(x)$ is a family of metric structure on $\b W$ such that $ g_\b(0)=i^* g$, where $i:\b W\to W$ denotes the inclusion.  The smooth forms on $W$ near the boundary decompose as $\omega=\omega_{\rm tan}+\omega_{\rm norm}$, where $\omega_{\rm norm}$ is the orthogonal projection on the subspace generated by $dx$, and $\omega_{\rm tan}$ is in $C^\infty(W)\otimes\Lambda(\b W)$. We  write $\omega=\omega_1+ dx \wedge\omega_{2}$, where $\omega_j\in C^\infty( W)\otimes \Lambda(\b W)$, and
\beq\label{dec}
\star\omega_2=-dx \wedge \star\omega.
\eeq

Define absolute boundary conditions by
\[
B_{\rm abs}(\omega)=\omega_{\rm norm}|_{\b W}=\omega_2|_{\b W}=0,
\]
and relative boundary conditions by
\[
B_{\rm rel}(\omega)=\omega_{\rm tan}|_{\b W}=\omega_1|_{\b W}=0.
\]

Note that, if $\omega \in \Omega^{q}(W)$, then $B_{\rm abs}(\omega) = 0$ if and only if $B_{\rm rel} (\star\omega) = 0$,
$B_{\rm rel}(\omega) = 0$ implies $B_{\rm rel} (d\omega) = 0$, and  $B_{\rm abs}(\omega) = 0$ implies $B_{\rm
abs} (d^{\dag}\omega) = 0$. Let $\B(\omega)=B(\omega)\oplus B((d+d^\dagger)(\omega))$. Then the operator $\Delta=(d+d^\dagger)^2$ with boundary conditions $\B(\omega)=0$  is self adjoint, and if $\B(\omega)=0$, then $\Delta\omega=0$ if and only if $(d+d^\dagger)\omega=0$. Note  that $\B$ correspond to
\beq\label{abs}
\B_{\rm abs}(\omega)=0\hspace{20pt}{\rm if~ and~ only~ if}\hspace{20pt}\left\{\begin{array}{l}\omega_{\rm norm}|_{\b W}=0,\\
(d\omega)_{\rm norm}|_{\b W}=0,\\
       \end{array}
\right.
\eeq
\beq\label{rel}
\B_{\rm rel}(\omega)=0\hspace{20pt}{\rm if~ and~ only~ if}\hspace{20pt}\left\{\begin{array}{l}\omega_{\rm tan}|_{\b W}=0,\\
(d^\dagger\omega)_{\rm tan}|_{\b W}=0,\\
       \end{array}
\right.
\eeq


Let
\begin{align*}
\H^q(W,E_\rho)&=\{\omega\in\Omega^q(W,E_\rho)~|~\Delta^{(q)}\omega=0\},\\
\H_{\rm abs}^q(W,E_\rho)&=\{\omega\in\Omega^q(W,E_\rho)~|~\Delta^{(q)}\omega=0, B_{\rm abs}(\omega)=0\},\\
\H_{\rm rel}^q(W,E_\rho)&=\{\omega\in\Omega^q(W,E_\rho)~|~\Delta^{(q)}\omega=0, B_{\rm rel}(\omega)=0\},
\end{align*}
be the spaces of harmonic forms with boundary conditions, then the Hodge decomposition reads
\begin{align*}
\Omega^{q}_{\rm abs}(W,E_\rho) &=  \H_{\rm abs}^q(W,E_\rho) \oplus d \Omega^{q-1}_{\rm abs}(W,E_\rho)\oplus
d^{\dagger}\Omega^{q+1}_{\rm abs}(W,E_\rho),\\
\Omega^{q}_{\rm rel}(W,E_\rho) &=  \H_{\rm rel}^q(W,E_\rho) \oplus d \Omega^{q-1}_{\rm rel}(W,E_\rho)\oplus
d^{\dagger}\Omega^{q+1}_{\rm rel}(W,E_\rho).\\
\end{align*}

\subsection{The form valued zeta functions and the analytic torsion}
\label{forms}

By the results of the previous sections, the  Laplace operator $\Delta^{(q)}$, with boundary conditions $\B_{\rm abs/rel}$  has a pure point spectrum $\Sp \Delta_{\rm abs/rel}^{(q)}$ consisting of real non negative eigenvalues. The sequence $\Sp_+ \Delta_{\rm abs/rel}^{(q)}$ is a totally regular sequence of spectral type accordingly to Section \ref{sb2.1}, and the  {\it forms valued zeta function} is the associated zeta function, defined by
\[
\zeta(s,\Delta_{\rm abs/rel}^{(q)} )=\zeta(s,\Sp_+ \Delta_{\rm abs/rel}^{(q)} )=\sum_{\lambda\in\Sp_+ \Delta_{\rm abs/rel}^{(q)}}\lambda^{-s},
\]
when $\Re(s)>\frac{m}{2}$ (see Propositions \ref{ss.l1} and {\ref{l3.1}). The {\it analytic torsion} $T_{\rm abs/rel}((W,g);\rho)$ of $(W,g)$ with respect to the representation $\rho:\pi_1(W)\to O(k,\R)$ is defined by 
\[
\log T_{\rm abs/rel}((W,g);\rho)=\frac{1}{2}\sum_{q=1}^m (-1)^q q \zeta'(0,\Delta_{\rm abs/rel}^{(q)}).
\]

We will omit the representation in the notation, whenever we mean the trivial representation.

\begin{theo}\label{Poinc1} {\bf Poincar\'e duality for analytic torsion} \cite{Luc}. Let $(W,g)$ be an orientable  compact connected Riemannian manifold of dimension $m$, with possible boundary. Then, for any representation $\rho$, 
\[
\log T_{\rm abs}((W,g);\rho)=(-1)^{m+1}\log T_{\rm rel}((W,g);\rho).
\]
\end{theo}

We now use results of section \ref{hodg} to define {\it closed, coclosed, exact} and {\it coexact zeta functions}. We again restrict ourselves to the case of a manifold without boundary (see \cite{RS} for the case of manifold with boundary). By the very definition, we have 
\[
\zeta(s,\Delta^{(q)})=\sum_{\lambda\in\Sp_+\Delta^{(q)}} \dim \E^{(q)}_\lambda \lambda^{-s}=
\zeta_{\rm cl}(s,\Delta^{(q)})+\zeta_{\rm ccl}(s,\Delta^{(q)}),
\]
where
\begin{align*}
\zeta_{\rm cl}(s,\Delta^{(q)})&=\sum_{\lambda\in\Sp_+\Delta^{(q)}} \dim \E^{(q)}_{\lambda, {\rm cl}}\lambda^{-s},\\
\zeta_{\rm ccl}(s,\Delta^{(q)})&=\sum_{\lambda\in\Sp_+\Delta^{(q)}} \dim \E^{(q)}_{\lambda, {\rm ccl}}\lambda^{-s}.\\
\end{align*}

Since, by (\ref{iso1}), $\zeta_{\rm cl}(s,\Delta^{(q)})=\zeta_{\rm ccl}(s,\Delta^{(q-1)})$, we obtain from the above relations the following formulas for the torsion of a closed $m$ dimensional manifold $W$:
\begin{align*}
\log T((W,g);\rho)=\frac{1}{2}\sum_{q=1}^m (-1)^q q\zeta'(0,\Delta^{(q)})&=\frac{1}{2}\sum_{q=1}^m(-1)^q\zeta'_{\rm cl}(0,\Delta^{(q)})\\
&=-\frac{1}{2}\sum_{q=0}^{m-1}(-1)^q\zeta'_{\rm ccl}(0,\Delta^{(q)}).
\end{align*}

In particular, using again duality, for an odd dimensional manifold $W$ of dimension $m=2p-1$, 
\beq\label{odd}
\begin{aligned}
\log T((W,g);\rho)&=\sum_{q=1}^{p-1} (-1)^q\zeta'_{\rm cl}(0,\Delta^{(q)})+\frac{(-1)^p}{2}\zeta'_{\rm cl}(0,\Delta^{(p)})\\
&=-\sum_{q=0}^{p-2} (-1)^q\zeta'_{\rm ccl}(0,\Delta^{(q)})+\frac{(-1)^{p}}{2}\zeta'_{\rm ccl}(0,\Delta^{(p-1)}).
\end{aligned}
\eeq


\subsection{The Cheeger M\"uller theorem for manifolds with boundary, and the anomaly boundary term of Br\"uning and Ma}  
\label{cm}

In case of a smooth orientable compact connect Riemannian manifold $(W,g)$ with boundary $\b W$, for any representation $\rho$ of the fundamental group (for simplicity assume ${\rm rk}(\rho)=1$), the analytic torsion is given by the Reidemeister torsion plus some further contributions. It was shown by J. Cheeger in \cite{Che1}, that this further contribution only depends on the boundary, namely that
\[
\log T_{\rm abs}((W,g);\rho)=\log\tau(W;\rho)+c(\b W).
\]

In the case of a product metric near the boundary, the following formula for this boundary contribution was given by W. L\"uck \cite{Luc}, where $\chi(X)$ denotes the Euler characteristic of $X$,
\[
\log T_{\rm abs}((W,g);\rho)=\log\tau(W;\rho)+\frac{1}{4}\chi(\b W)\log 2.
\]

In  the general case a further contribution appears, that measures how the metric is far from a product metric.
A formula for this new anomaly boundary contribution is contained in some recent result of Br\"uning and Ma \cite{BM}. More precisely, in \cite{BM} (equation (0.6)) is given a formula for the ratio of the analytic torsion of two metrics, $g_0$ and $g_1$,
\beq\label{bat}
\begin{aligned}
\log \frac{T_{\rm abs}((W,g_1);\rho)}{T_{\rm abs}((W,g_0);\rho)}=
\frac{1}{2}\int_{\b W} \left(B(\nabla_1)-B(\nabla_0)\right),
\end{aligned}
\eeq
where $\nabla_j$ is the covariant derivative of the metric $g_j$, the 
forms $B(\nabla_j)$ are defined as follows, according to Section 1.3 of \cite{BM} (observe, however, that  we take the opposite sign with respect to the definition in \cite{BM}, since we are considering left actions instead of right actions, and also that we use the formulas of \cite{BM} in the particular case of a flat trivial bundle $F$). Using the notation of Section \ref{ppp} (in particular see equation (\ref{proap})), we define the following forms

\beq\label{pippo}\begin{aligned}
\mathcal{S}_j&=\frac{1}{2}\sum_{k=1}^{m-1}(i^*\omega_j-i^*\omega_0)_{0 k}\wedge\hat e^*_{k},\\
\widehat {i^*\Omega_j}&=
\frac{1}{2}\sum_{k,l=1}^{m-1}i^*\Omega_{j,k l}\wedge\hat e^*_{k}\wedge  \hat e^*_{l}\\
\hat{\Theta}&=\frac{1}{2}\sum_{k,l=1}^{m-1}\Theta_{kl} \wedge \hat  e^*_{k}\wedge  \hat e^*_{l}.\\
\end{aligned}
\eeq

Here, $\omega_j$  are the connection one forms, and $\Omega_j$, $j=0,1$, the curvature two forms associated to the metrics $g_0$ and $g_1$, respectively, while $ \Theta$ is the curvature two form of the boundary (with the metric induced by the inclusion), and $\{e_k\}_{k=0}^{m-1}$ is an orthonormal base of $TW$ (with respect to the metric $g$). Then, set  
\beq\label{ebm1} 
B(\nabla_j)=\frac{1}{2}\int_0^1\int^B
\e^{-\frac{1}{2}\hat{\Theta}-u^2 \mathcal{S}_j^2}\sum_{k=1}^\infty \frac{1}{\Gamma\left(\frac{k}{2}+1\right)}u^{k-1}
\mathcal{S}_j^k du. 
\eeq

Taking $g_1=g$, and $g_0$ an opportune deformation of $g$, that is a product metric near the boundary, it is easy to see that (see equation (\ref{sr2}) of Section \ref{Lap1.2})
\[
\log \frac{T_{\rm abs}((W,g_1);\rho)}{T_{\rm abs}((W,g_0);\rho)}=\frac{1}{2}\int_{\b W} B(\nabla_1).
\]

Note that the right end side of this equation is (as expected) a local quantity, and is well defined if there exists a regular collar neighborhood of the boundary. If this is the case,  we define the Br\"{u}ning and Ma {\it anomaly boundary term} by
\beq\label{anom}
A_{\rm BM,abs}(\b W)=\frac{1}{2}\int_{\b W} B(\nabla_1),
\eeq
and we have 
\beq\label{pop1}
\log T_{\rm abs}((W,g);\rho)=\log\tau(W;\rho)+\frac{1}{4}\chi(\b W)\log 2+A_{\rm BM,abs}(\b W).
\eeq

\section{Zeta determinants}
\label{sb2}

In this section we collect some results on the theory of the zeta function associated to a sequence of spectral type  introduced in works of M. Spreafico \cite{Spr3} \cite{Spr4} \cite{Spr5} and \cite{Spr9}. This will give the analytic tools necessary in order to evaluate the zeta determinants appearing in the calculation of the analytic torsion in the following sections. We give the basic definition  in Section \ref{sb2.1}, some results concerning simple sequences in Section \ref{ss2}, the main results for double sequences in Section \ref{sb2.2}, and we specializes to the zeta functions associated to the Laplace operator on Riemannian manifolds in Section \ref{s2}. We present here a simplified version of the theory, that is sufficient for our purpose here; we refer to the mentioned papers for further details  (see in particular the  general formulation in Theorem 3.9 of \cite{Spr9} or the Spectral Decomposition Lemma of \cite{Spr5}).

\subsection{Zeta functions for sequences of spectral type}
\label{sb2.1}

Let $S=\{a_n\}_{n=1}^\infty$ be a sequence of non vanishing complex numbers, ordered by increasing modules, with the unique point of accumulation at infinite. The positive real number (possibly infinite)
\[
s_0=\limsup_{n\to\infty} \frac{\log n}{\log |a_n|},
\]
is called the {\it exponent of convergence} of $S$, and denoted by $\ec(S)$. We are only interested in sequences with  $\ec(S)=s_0<\infty$. If this is the case, then there exists a least integer $p$ such that the series $\sum_{n=1}^\infty a_n^{-p-1}$ converges absolutely. We assume $s_0-1< p\leq s_0$, we call the integer $p$ the {\it genus} of the sequence $S$, and we write $p=\ge(S)$.  We define the {\it zeta function} associated to $S$ by the uniformly convergent series
\[
\zeta(s,S)=\sum_{n=1}^\infty a_n^{-s},
\]
when $\Re(s)> \ec(S)$, and by analytic continuation otherwise.  We call the open subset $\rho(S)=\C-S$ of the complex plane the  {\it resolvent set} of $S$. For all $\lambda\in\rho(S)$, we define the {\it Gamma function} associated to $S$  by the canonical product
\beq\label{gamma}
\frac{1}{\Gamma(-\lambda,S)}=\prod_{n=1}^\infty\left(1+\frac{-\lambda}{a_n}\right)\e^{\sum_{j=1}^{\ge(S)}\frac{(-1)^j}{j}\frac{(-\lambda)^j}{a_n^j}}.
\eeq

When necessary in order to define the meromorphic branch of an analytic function, the domain for $\lambda$ will be the  open subset $\C-[0,\infty)$ of the complex plane.
We use the notation $\Sigma_{\theta,c}=\left\{z\in \C~|~|\arg(z-c)|\leq \frac{\theta}{2}\right\}$,
with $c\geq \delta> 0$, $0< \theta<\pi$. We use
$D_{\theta,c}=\C-\Sigma_{\theta,c}$, for the complementary (open) domain and $\Lambda_{\theta,c}=\partial \Sigma_{\theta,c}=\left\{z\in \C~|~|\arg(z-c)|= \frac{\theta}{2}\right\}$, oriented counter clockwise, for the boundary.
With this notation, we define now a particular subclass of sequences. Let $S$ be as above, and assume that $\ec(S)<\infty$, and that there exist $c>0$ and $0<\theta<\pi$, such that $S$ is contained in the interior of the sector $\Sigma_{\theta,c}$. Furthermore, assume that the logarithm of the associated Gamma function has a uniform asymptotic expansion for large $\lambda\in D_{\theta,c}(S)=\C-\Sigma_{\theta,c}$ of the following form
\[
\log\Gamma(-\lambda,S)\sim\sum_{j=0}^\infty a_{\alpha_j,0}(-\lambda)^{\alpha_j} +\sum_{k=0}^{\ge(S)} a_{k,1}(-\lambda)^k\log(-\lambda),
\]
where $\{\alpha_j\}$ is a decreasing sequence of real numbers. Then, we say that $S$ is a {\it totally regular sequence of spectral type with infinite order}. We call the open set $D_{\theta,c}(S)$ the {\it asymptotic domain} of $S$.

\subsection{The zeta determinant of some simple sequences}
\label{ss2}

The results of this section are known to specialists, and can be found in different places. We will use the formulation of \cite{Spr1}. For positive real numbers $l$ and $q$, define the {\it non homogeneous quadratic Bessel zeta function} by
\[
z(s,\nu,q,l)=\sum_{k=1}^\infty \left(\frac{j_{\nu,k}^2}{l^2}+q^2\right)^{-s},
\]
for $\Re(s)>\frac{1}{2}$. Then, $z(s,\nu,q,l)$ extends analytically to a meromorphic function in the complex plane with simple poles at $s=\frac{1}{2}, -\frac{1}{2}, -\frac{3}{2}, \dots$. The point $s=0$ is a regular point and
\beq\label{p00}
\begin{aligned}
z(0,\nu,q,l)&=-\frac{1}{2}\left(\nu+\frac{1}{2}\right),\\
z'(0,\nu,q,l)&=-\log\sqrt{2\pi l}\frac{I_\nu(lq)}{q^\nu}.
\end{aligned}
\eeq

In particular, taking the limit for $q\to 0$,
\[
z'(0,\nu,0,l)=-\log\frac{\sqrt{\pi}l^{\nu+\frac{1}{2}}}{2^{\nu-\frac{1}{2}}\Gamma(\nu+1)}.
\]

\subsection{Zeta determinant for a class of double sequences}
\label{sb2.2}

Let $S=\{\lambda_{n,k}\}_{n,k=1}^\infty$ be a double sequence of non
vanishing complex numbers with unique accumulation point at the
infinity, finite exponent $s_0=\ec(S)$ and genus $p=\ge(S)$. Assume if necessary that the elements of $S$ are ordered as $0<|\lambda_{1,1}|\leq|\lambda_{1,2}|\leq |\lambda_{2,1}|\leq \dots$. We use the notation $S_n$ ($S_k$) to denote the simple sequence with fixed $n$ ($k$). We call the exponents of $S_n$ and $S_k$ the {\it relative exponents} of $S$, and we use the notation $(s_0=\ec(S),s_1=\ec(S_k),s_2=\ec(S_n))$. We define {\it relative genus} accordingly.

\begin{defi} Let $S=\{\lambda_{n,k}\}_{n,k=1}^\infty$ be a double
sequence with finite exponents $(s_0,s_1,s_2)$, genus
$(p_0,p_1,p_2)$, and positive spectral sector
$\Sigma_{\theta_0,c_0}$. Let $U=\{u_n\}_{n=1}^\infty$ be a totally
regular sequence of spectral type of infinite order with exponent
$r_0$, genus $q$, domain $D_{\phi,d}$. We say that $S$ is
spectrally decomposable over $U$ with power $\kappa$, length $\ell$ and
asymptotic domain $D_{\theta,c}$, with $c={\rm min}(c_0,d,c')$,
$\theta={\rm max}(\theta_0,\phi,\theta')$, if there exist positive
real numbers $\kappa$, $\ell$ (integer), $c'$, and $\theta'$, with
$0< \theta'<\pi$,   such that:
\begin{enumerate}
\item the sequence
$u_n^{-\kappa}S_n=\left\{\frac{\lambda_{n,k}}{u^\kappa_n}\right\}_{k=1}^\infty$ has
spectral sector $\Sigma_{\theta',c'}$, and is a totally regular
sequence of spectral type of infinite order for each $n$;
\item the logarithmic $\Gamma$-function associated to  $S_n/u_n^\kappa$ has an asymptotic expansion  for large
$n$ uniformly in $\lambda$ for $\lambda$ in
$D_{\theta,c}$, of the following form
\beq\label{exp}
\hspace{30pt}\log\Gamma(-\lambda,u_n^{-\kappa} S_n)=\sum_{h=0}^{\ell}
\phi_{\sigma_h}(\lambda) u_n^{-\sigma_h}+\sum_{l=0}^{L}
P_{\rho_l}(\lambda) u_n^{-\rho_l}\log u_n+o(u_n^{-r_0}),
\eeq
where $\sigma_h$ and $\rho_l$ are real numbers with $\sigma_0<\dots <\sigma_\ell$, $\rho_0<\dots <\rho_L$, the
$P_{\rho_l}(\lambda)$ are polynomials in $\lambda$ satisfying the condition $P_{\rho_l}(0)=0$, $\ell$ and $L$ are the larger integers 
such that $\sigma_\ell\leq r_0$ and $\rho_L\leq r_0$.
\end{enumerate}
\label{spdec}
\end{defi}

When a double sequence $S$ is spectrally decomposable over a simple sequence $U$, Theorem 3.9 of \cite{Spr9} gives a formula for the derivative of the associated zeta function at zero. In order to understand such a formula, we need to introduce some other quantities. First, we define the functions
\beq\label{fi1}
\Phi_{\sigma_h}(s)=\int_0^\infty t^{s-1}\frac{1}{2\pi i}\int_{\Lambda_{\theta,c}}\frac{\e^{-\lambda t}}{-\lambda} \phi_{\sigma_h}(\lambda) d\lambda dt.
\eeq

Next, by Lemma 3.3 of \cite{Spr9}, for all $n$, we have the expansions:
\beq\label{form}\begin{aligned}
\log\Gamma(-\lambda,S_n/{u_n^\kappa})&\sim\sum_{j=0}^\infty a_{\alpha_j,0,n}
(-\lambda)^{\alpha_j}+\sum_{k=0}^{p_2} a_{k,1,n}(-\lambda)^k\log(-\lambda),\\
\phi_{\sigma_h}(\lambda)&\sim\sum_{j=0}^\infty b_{\sigma_h,\alpha_j,0}
(-\lambda)^{\alpha_j}+\sum_{k=0}^{p_2} b_{\sigma_h,k,1}(-\lambda)^k\log(-\lambda),
\end{aligned}
\eeq
for large $\lambda$ in $D_{\theta,c}$. We set (see Lemma 3.5 of \cite{Spr9})
\beq\label{fi2}
\begin{aligned}
A_{0,0}(s)&=\sum_{n=1}^\infty \left(a_{0, 0,n} -\sum_{h=0}^\ell
b_{\sigma_h,0,0}u_n^{-\sigma_h}\right)u_n^{-\kappa s},\\
A_{j,1}(s)&=\sum_{n=1}^\infty \left(a_{j, 1,n} -\sum_{h=0}^\ell
b_{\sigma_h,j,1}u_n^{-\sigma_h}\right)u_n^{-\kappa s},
~~~0\leq j\leq p_2.
\end{aligned}
\eeq

We can now state the formula for the derivative at zero of the double zeta function. We give here a modified version of Theorem 3.9 of \cite{Spr9}, more suitable for our purpose here. This is based on the following fact. The key point in the proof of Theorem 3.9 of \cite{Spr9} is the decomposition given in Lemma 3.5 of that paper of the sum
\[
\mathcal{T}(s,\lambda, S,U)=\sum_{n=1}^\infty u_n^{-\kappa s} \log\Gamma(-\lambda, u_n^{-\kappa}S_n),
\]
in two terms: the regular part $\mathcal{P}(s,\lambda,S,U)$ and the remaining singular part. The regular part is obtained subtracting from $\T$ some terms constructed starting from the expansion of the logarithmic Gamma function given in equation (\ref{exp}), namely
\[
\P(s,\lambda,S,u)=\T(s,\lambda, S,U)-\sum_{h=0}^{\ell}
\phi_{\sigma_h}(\lambda) u_n^{-\sigma_h}-\sum_{l=0}^{L}
P_{\rho_l}(\lambda)u_n^{-\rho_l}\log u_n.
\]

Now, assume instead we subtract only the terms such that the zeta function $\zeta(s,U)$ has a pole at $s=\sigma_h$ or at $s=\rho_l$. Let $\hat \P(s,\lambda, S,U)$ be the resulting function. Then the same argument as the one used in Section 3 of \cite{Spr9} in order to prove Theorem 3.9 applies, and we obtain similar formulas for the values of the residue, and of the finite part of the zeta function $\zeta(s,S)$ and of its derivative at zero, with just two differences: first, in the all the sums, all the terms with index $\sigma_h$ such that $s=\sigma_h$ is not a pole of $\zeta(s,U)$ must be omitted; and second, we must substitute the terms $A_{0,0}(0)$ and $A_{0,1}'(0)$, with the finite parts of the analytic continuation of $A_{0,0}(s)$, and $A_{0,1}'(s)$. The first modification is an obvious consequence of the substitution of the function $\P$ by the function $\hat \P$. The second modification, follows by the same reason noting that the function $A_{\alpha_j,k}(s)$ defined in Lemma 3.5 of \cite{Spr9} are no longer regular at $s=0$ themselves. However, they both admits a meromorphic extension regular at $s=0$, using the extension of the zeta function $\zeta(s,U)$, and the expansion of the coefficients $a_{\alpha_j,k,n}$ for large $n$.
Thus we have the following result.

\begin{theo} \label{tt} The formulas of Theorem 3.9 of \cite{Spr9} hold if all the quantities with index $\sigma_h$ such that the zeta function $\zeta(s,U)$ has not a pole at $s=\sigma_h$ are omitted. In such a case, the result must be read by means of the analytic extension of the zeta function $\zeta(s,U)$.
\end{theo}

Next, assuming some simplified pole structure for the zeta function $\zeta(s,U)$, sufficient for the present analysis, we state the main result of this section.

\begin{theo} \label{t4} Let $S$ be spectrally decomposable over $U$ as in Definition \ref{spdec}. Assume that the functions $\Phi_{\sigma_h}(s)$ have at most simple poles for $s=0$. Then,
$\zeta(s,S)$ is regular at $s=0$, and
\begin{align*}
\zeta(0,S)=&-A_{0,1}(0)+\frac{1}{\kappa}{\sum_{h=0}^\ell} \Ru_{s=0}\Phi_{\sigma_h}(s)\Ru_{s=\sigma_h}\zeta(s,U),\\
\zeta'(0,S)=&-A_{0,0}(0)-A_{0,1}'(0)+\frac{\gamma}{\kappa}\sum_{h=0}^\ell\Ru_{s=0}\Phi_{\sigma_h}(s)\Ru_{s=\sigma_h}\zeta(s,U)\\
&+\frac{1}{\kappa}\sum_{h=0}^\ell\Rz_{s=0}\Phi_{\sigma_h}(s)\Ru_{s=\sigma_h}\zeta(s,U)+{\sum_{h=0}^\ell}{^{\displaystyle
'}}\Ru_{s=0}\Phi_{\sigma_h}(s)\Rz_{s=\sigma_h}\zeta(s,U),
\end{align*}
where the notation $\sum'$ means that only the terms such that $\zeta(s,U)$ has a pole at $s=\sigma_h$ appear in the sum.

\end{theo}

This result should be compared with the Spectral Decomposition Lemma  of \cite{Spr5} and Proposition 1 of \cite{Spr6}.

\begin{rem}\label{rqr} We call regular part of $\zeta(0,S)$ the first term appearing in the formula given in the theorem, and regular part of $\zeta'(0,S)$ the first two terms. The other terms gives what we call singular part.
\end{rem}

\begin{corol} \label{c} Let $S_{(j)}=\{\lambda_{(j),n,k}\}_{n,k=1}^\infty$, $j=1,...,J$, be a finite set of  double sequences that satisfy all the requirements of Definition \ref{spdec} of spectral decomposability over a common sequence $U$, with the same parameters $\kappa$, $\ell$, etc., except that the polynomials $P_{(j),\rho}(\lambda)$ appearing in condition (2) do not vanish for $\lambda=0$. Assume that some linear combination $\sum_{j=1}^J c_j P_{(j),\rho}(\lambda)$, with complex coefficients, of such polynomials does satisfy this condition, namely that $\sum_{j=1}^J c_j P_{(j),\rho}(\lambda)=0$. Then, the linear combination of the zeta function $\sum_{j=1}^J c_j \zeta(s,S_{(j)})$ is regular at $s=0$ and satisfies the linear combination of the formulas given in Theorem \ref{t4}.
\end{corol}

\subsection{Zeta invariants of compact Riemannian manifolds}
\label{s2}

We recall in this section some known facts about zeta invariants of a compact manifold. We will rewrite such results in the terminology of zeta functions associated to sequences of spectral type just introduced. Our main reference are the works of P. Gilkey, in particular we refer to the book \cite{Gil}.

Let $(W,g)$ be a 
compact connected Riemannian manifold of dimension
$m$, with metric $g$. Let $\Delta^{(q)}$ denote the metric Laplacian on forms on $W$, and  $\Sp \Delta^{(q)}=\{\lambda_n\}_{n=0}^\infty$ ($\lambda_0=0$) its spectrum. Then,  there exists a full asymptotic expansion for the trace of the heat kernel of $\Delta^{(q)}$ for small $t$, 
\beq\label{eegg}
{\rm Tr}_{L^2}\e^{-t\Delta^{(q)}}=t^{-\frac{m}{2}}\sum_{j=0}^\infty e_{q,j} t^\frac{j}{2},
\eeq
where the coefficients  depend only on local invariants constructed from the metric tensor, and are in
principle calculable from it.


\begin{prop} \label{coeff}
\begin{align*}
e_{q,0}&=\frac{1}{(4\pi)^\frac{m}{2}}\binom{m}{q} \int_W dvol_g ,\\
e_{q,2}&=\frac{1}{6(4\pi)^\frac{m}{2}}\left(\binom{m}{q}-6\binom{m-2}{q-1}\right) \int_W \tau dvol_g ,\\
e_{q,4}&=\frac{1}{360(4\pi)^\frac{m}{2}} \left(5\binom{m}{q}-60\binom{m-2}{q-1}+180\binom{m-4}{p-2}\right) \int_W \tau^2 dvol_g \\
&+\frac{1}{360(4\pi)^\frac{m}{2}} \left(-2\binom{m}{q}+180\binom{m-2}{q-1}-720\binom{m-4}{p-2}\right) \int_W |Ric|^2 dvol_g\\
&+\frac{1}{360(4\pi)^\frac{m}{2}} \left(2\binom{m}{q}-30\binom{m-2}{q-1}+180\binom{m-4}{p-2}\right) \int_W |R|^2 dvol_g.
\end{align*}
\end{prop}


\begin{prop} \label{ss.l1} The sequence $\Sp_+ \Delta^{(q)}$
of the positive eigenvalues of the metric Laplacian on forms 
on a compact
connected Riemannian manifold of dimension $m$, is a totally regular sequence of spectral type, with finite
exponent $\ec=\frac{m}{2}$, genus $\ge=[\ec]$, spectral sector $\Sigma_{\theta,c}$ with  some $0< c< \lambda_1$, $\epsilon<\theta<\frac{\pi}{2}$,  asymptotic domain $D_{\theta,c}=\C-\Sigma_{\theta,c}$, and infinite order.
\end{prop}

\begin{prop} The zeta function $\zeta(s,\Sp_+\Delta^{(q)})$ has a meromorphic continuation
to the whole complex plane up to simple poles at the values of
$s=\frac{m-h}{2}$, $h=0,1,2,\dots$, that are not negative integers
nor zero, with residues
\[
\begin{array}{l}\Ru_{s=\frac{m-h}{2}}
\zeta(s,\Sp_+ \Delta^{(q)})=\frac{e_{q,h}}{\Gamma\left(\frac{m-h}{2}\right)},\\
\end{array}
\]
the point $s=-k=0,-1,-2,\dots$ are regular points and
\begin{align*}
\zeta(0,\Sp_+ \Delta^{(q)})&=e_{q,m}-{\rm dimker}\Delta^{(q)},\\
\zeta(-k,\Sp_+ \Delta^{(q)})&=(-1)^k k! e_{q,m+2k}.
\end{align*}
\label{l3.1}
\end{prop}

\section{Geometric setting and Laplace operator}
\label{Lap1}

\subsection{The finite metric cone}
\label{Lap1.1}

Let $(W,g)$ be an orientable  compact connected Riemannian manifold of finite dimension $m$ without boundary  and with Riemannian structure $g$. We denote by $CW$ the {\it cone} over $W$, namely the mapping cone of the constant map $:W\to \{p\}$. Then, $CW$ is compact connected separable Hausdorff space, but in general is not a  topological manifold. However, if we remove the tip of the cone $p$, then $CW-\{p\}$ is an open differentiable manifold, with the obvious differentiable structure. Embedding $W$ in the opportune Euclidean space $\R^k$, and $\R^k$ in some hyperplane of $\R^{k+h}$, with opportune $h$, disconnected from the origin, a geometric realization of $CW$ is the  given by the set of the finite length $l$ line segments joining the origin to the embedded copy of $W$. Let $x$ the euclidean geodesic distance from the origin, if we equip $CW-\{p\}$ with the Riemannian structure
\beq\label{g1}
dx\otimes dx+x^2 g,
\eeq
this coincides with the metric structure induced by the described embedding. We denote by $C_{(0,l]}W$ the space $(0,l]\times W$ with the metric in equation (\ref{g1}). We denote by $C_l W$ the compact space $\overline{C_{(0,l]}W}=C_{(0,l]}W\cup\{p\}$. We call the space $C_l W$ the {\it (completed finite metric) cone} over $W$. We call the subspace $\{l\}\times W$ of $C_l W$, the {\it boundary of the cone}, and we denote it by $\b C_l W$. This is of course diffeomorphic to $W$, and isometric to $(W,l^2g)$.  For the  global coordinate $x$ corresponds to the local coordinate $x'=l-x$, where $x'$ is the geodesic distance from the boundary. Therefore, $ g_\b(x')=(l-x)^2g$, and if $i:W\to C_{(0,l]}W$ denotes the inclusion, $i^*(dx\otimes dx+x^2g)=g_\b(0)=l^2 g$.  Following common notation, we will call $(W,  g)$ the {\it section } of the cone. Also following usual notation, a tilde will denotes operations on the section (of course $\tilde g=g$), and not on the boundary. All the results of Section \ref{bord} are valid. In particular, given a local coordinate system $y$ on $W$, then $(x,y)$ is a local coordinate system on the cone. 

We now give the explicit form of $\star$, $d^\dagger$ and $\Delta$. See \cite{Che0} \cite{Che2} and \cite{Nag} Section 5 for details. If $\omega\in \Omega^{q}(C_{(0,l]} W)$, set
\[
\omega(x,y)=f_1(x)\omega_1(y)+f_2(x)dx\wedge \omega_2(y),
\]
with smooth functions $f_1$ and $f_2$, and $\omega_j\in \Omega( W)$. Then a straightforward calculation gives 
\begin{align}
\label{f1}\star \omega(x,y)&= x^{m-2q+2} f_2(x)\tilde\star \omega_2(y)+(-1)^q x^{m-2q}f_1(x) dx\wedge\tilde\star \omega_1(y),
\end{align}
\beq\label{f2}
\begin{aligned}d \omega(x,y)   &= f_1(x)\tilde d \omega_1(y) + \b_x f_1(x) dx \wedge \omega_1(y) - f_2(x) dx \wedge d\omega_2(y),\\
d^\dagger \omega(x,y)&= x^{-2} f_1(x)\tilde d^{\dag}\omega_1 (y) -\left((m-2q+2)x^{-1}f_2(x) + \b_x f_2(x)\right)\omega_2(y)\\
&-x^{-2}f_2(x) dx \wedge \tilde d^{\dag} \omega_2(y),
\end{aligned}
\eeq

\beq\label{f3}
\begin{aligned}
\Delta\omega(x,y)&= \left(-\b_x^2 f_1(x) -(m-2q)x^{-1}\b_x f_1(x)\right)\omega_1(y) + x^{-2}f_1(x)\tilde
\Delta\omega_1(y)-2x^{-1}f_2(x)\tilde d\omega_2(y)\\
&+dx\wedge \left(x^{-2}f_2(x) \tilde\Delta\omega_2(y)+\omega_2(y)\left(-\b^2_x f_2(x) -(m-2q+2)x^{-1}\b_x f_2(x)\right.\right.\\
&\left.\left. + (m-2q+2)x^{-2}f_2(x) \right) -2x^{-3}f_1(x)\tilde d^{\dag}\omega_1(y)\right).
\end{aligned}
\eeq

\subsection{Riemannian tensors on the cone}
\label{Lap1.2} 

We give here the explicit form of the main Riemannian quantities on the cone. All calculation are based on the formulas given in Sections \ref{rim} and \ref{cm}. Recall that a tilde denotes quantities relative to the section, that we have local coordinate $(x,y_1,\dots,y_m)$ on $C_lW$, and that the metric is
\[
g_1=dx\otimes dx+x^2 g.
\]

Let $\{b_k\}_{k=1}^m$ be a local orthonormal base of $TW$, and $\{b^*_k\}_{k=1}^m$ the associated dual base.  Then,
\begin{align*}
e_0&=\b_x& e^*_0&=dx,&&\\
e_k&=\frac{1}{x}b_k,&e^*_k&=xb^*_k, &&1\leq k\leq m.
\end{align*}

Direct calculations give Cartan structure constants
\begin{align*}
c_{jk0}&=0,&&1\leq j,k\leq m,\\
c_{0kl}=-c_{k0l}&=-\frac{\delta_{kl}}{x},&&1\leq k,l\leq m,\\
c_{jkl}&=\frac{1}{x}\tilde c_{jkl}, &&1\leq j,k,l\leq m.
\end{align*}

The Christoffel symbols are
\begin{align*}
\Gamma_{0kl}&=0,&&1\leq k,l\leq m,\\
\Gamma_{j0k}=-\Gamma_{jk0}&=\frac{\delta_{jk}}{x},&&1\leq j,k\leq m,\\
\Gamma_{jkl}&=\frac{1}{x}\tilde \Gamma_{jkl}, &&1\leq j,k,l\leq m.
\end{align*}

The connection one form matrix relatively to the metric $g_1$ has components
\beq\label{om1}
\begin{aligned}
\omega_{1,00}&=0,\\
\omega_{1,0j}=-\omega_{1,j0}&=-\frac{1}{x}e^*_j=-b^*_j,&&1\leq j\leq m,\\
\omega_{1,jk}&=\sum_{h=1}^m \Gamma_{hkj}e^*_h=\frac{1}{x}\sum_{h=1}^m \tilde\Gamma_{hkj}e^*_h
=\sum_{h=1}^m \tilde\Gamma_{hkj}b^*_h=\tilde\omega_{jk},&&1\leq j,k\leq m.\\
\end{aligned}
\eeq

To compute the curvature we calculate
\[
d\omega_{1,0j}=-\sum_{l=1}^m (\b_l b^*_j)\wedge dy_l=-\sum_{l,k=1}^m (\b_l b_{kj})dy_k\wedge dy_l,
\]
where $b_j^*=\sum_{k=1}^m b_{kj} dy_k$, and, for $1\leq j,k\leq m$, 
\[
d\omega_{1,jk}=\tilde d\tilde\omega_{jk};
\]
while
\begin{align*}
-(\omega_1\wedge\omega_1)_{k0}&=(\omega_1\wedge\omega_1)_{0k}=\sum_{l=0}^m \omega_{1,0l}\wedge \omega_{1,lk}=\sum_{l=1}^m \omega_{1,0l}\wedge \omega_{1,lk}=-\sum_{l=1}^m b_l^*\wedge\tilde\omega_{lk},\\
(\omega_1\wedge\omega_1)_{jk}&=\sum_{l=0}^m \omega_{1,jl}\wedge \omega_{1,lk}=\omega_{1,j0}\wedge \omega_{1,0k}+\sum_{l=1}^m \omega_{1,jl}\wedge \omega_{1,lk}=-b^*_j\wedge b^*_k+(\tilde\omega\wedge\tilde\omega)_{jk},
\end{align*}
for $1\leq j,k\leq m$. The curvature two form  has components
\begin{align*}
\Omega_{1,00}&=0,\\
\Omega_{1,0j}&=-\sum_{l,k=1}^m (\b_l b_{kj})dy_k\wedge dy_l-\sum_{l=1}^m b_l^*\wedge\tilde\omega_{lk},&&1\leq j\leq m,\\
\Omega_{1,jk}&=\tilde d\tilde\omega_{jk}-b^*_j\wedge b^*_k+(\tilde\omega\wedge\tilde\omega)_{jk}=\tilde\Omega_{jk}-b_j^*\wedge b_k^*,&&1\leq j,k\leq m.\\
\end{align*}

Next, considering  the metric $g_0=dx\otimes dx+g$, similar calculations gives:
\beq\label{om0}
\begin{aligned}
\omega_{0,0j}&=0,&&0\leq j\leq m,\\
\omega_{0,jk}&=\tilde\omega_{jk},&&1\leq j,k\leq m.\\
\end{aligned}
\eeq

By equations (\ref{om1}) and (\ref{om0}),
\begin{align}
\label{sr1}\mathcal{S}_1&=-\frac{1}{2l}\sum_{k=1}^{m}e^*_k\wedge e^*_{k}=-\frac{l}{2}\sum_{k=1}^{m}b^*_k\wedge b^*_{k}=-\frac{1}{2}\sum_{k=1}^{m}b^*_k\wedge e^*_{k},\\
\label{sr2}\mathcal{S}_0&=0.
\end{align}

We also need the curvature two form $\Theta$ on the boundary $\b C_l W$. A similar calculation gives 
\[
\Theta_{jk}=\tilde\Omega_{jk}.
\]

Note in particular that it is easy to verify the equation (1.16) of \cite{BM}: $\hat{\Theta}=\widehat{i^*\Omega_1}-2\S_1^2$. For 
\begin{align*}
2\S_1^2&=-\frac{l^2}{2}\sum_{j,k=1}^m b_j^*\wedge b_k^*\wedge\hat b_j^*\wedge \hat b_k^*,\\
\hat{\Theta}&=\frac{l^2}{2}\sum_{j,k=1}^m  \tilde\Omega_{jk}\wedge \hat b_j^*\wedge \hat b_k^*,
\end{align*}
while $(i^*\Omega)_{jk}=\tilde \Omega_{jk}-b_j^*\wedge b_k^*$,  gives
\[
\widehat{i^*\Omega_{1,jk}}=\frac{l^2}{2}\sum_{j,k=1}^m \left( \tilde\Omega_{jk}-b_j^*\wedge b_k^*\right)\wedge \hat b_j^*\wedge \hat b_k^*.
\]

\subsection{The Laplace operator on the cone and its spectrum}
\label{Lap1.3}

We study the  Laplace operator on forms on the space $C_lW$. This is essentially based on \cite{Che2} and \cite{BS2}. Let denote by $\mathcal{L}$ the formal differential operator defined by equation (\ref{f3}) acting on smooth forms on $C_{(0,l]}W$, $\Gamma(C_{(0,l]}W, \Lambda T^* C_{(0,l]}W)$. We define in Lemma \ref{l1} a self adjont operator $\Delta$ acting on $L^2 (C_l W,\Lambda^{(q)}C_lW)$, and such that $\Delta \omega=\mathcal{L}\omega$, if $\omega\in {\rm dom} \Delta$. Then, in Lemma \ref{l2}, we list all the solutions of the eigenvalues equation for $\mathcal{L}$. Eventually, in Lemma \ref{l3}, we give the spectrum of $\Delta$.

\begin{lem} \label{l1} The formal operator $\mathcal{L}$ in equation (\ref{f3}) with the absolute/relative boundary conditions given in equations (\ref{abs})/(\ref{rel}) on the boundary $\b C_lW$ defines a unique self adjoint semi bounded operator on $L^2 (C_lW,\Lambda^{(q)} T^*C_l W)$, that we denote by the symbol $\Delta_{\rm abs}$/$\Delta_{\rm rel}$, respectively, with pure point spectrum.
\end{lem}
\begin{proof} Let $L^{(q)}$ denote the minimal operator defined by the formal operator $\mathcal{L}^{(q)}$, with domain the $q$-forms with compact support in $C_{(0,l]} W$, namely ${\rm dom} L^{(q)}=\Gamma_0(C_{(0,l]}W, \Lambda T^* C_{(0,l]}W)$. The boundary values problem at the boundary $x=l$, i.e. $\b C_lW$, is trivial, and gives the self adjoint extensions stated. The point $x=0$ requires more work. First, note that $L^{(q)}$ reduces by unitary transformation to an operator of the type
\beq\label{A}
D^2+\frac{A(x)}{x^2}, \hspace{20pt} D=-i\frac{d}{dx},
\eeq
where $A(x)$ is smooth family of symmetric second order elliptic operators \cite{BS2} pg. 370. More precisely, the map 
\begin{align*}
\psi_q&:C_0^\infty ((0,l],\Lambda^{(q)}T^* W\times \Lambda^{(q-1)}T^*W)\to C^\infty(C_{(0,l]}W,\Lambda^{(q)}T^* C_{(0,l]}W),\\
\psi_q&:(\omega^{(q)},\omega^{(q-1)})\mapsto x^{q-m/2}\pi^*\omega^{(q)}(x)+x^{p-1-m/2}\pi^*  \omega^{(q-1)}(x) \wedge
dx,
\end{align*}
where $\pi:C_{(0,l]} W\to W$, is bijective onto forms with compact support. Moreover, $\psi_q$ is unitary with respect to the usual $L^2$ structure on the function space $ \mathcal{F}(C_lW,\Lambda^{(q)}T^* C_lW)$ and the Hilbert space structure on $\mathcal{F}([0,l],\Lambda^{(q)}T^* W\times \Lambda^{(q-1)}T^*W)$ given by
\[
\int_0^l \left(||\omega^{(q)}(x)||^2_{\Lambda^{(q)}T^* W, g}+||\omega^{(q-1)}(x)||^2_{\Lambda^{(q-1)}T^* W, g}\right)dx.
\]

Under the transformation $\psi_q$, $L^{(q)}$ has the form in equation (\ref{A}), with $A(x)$  the constant smooth family of symmetric second order elliptic operators in $\Gamma(W,\Lambda^{(q)}T^* W\times \Lambda^{(q-1)}T^*W)$:
\[
A(x)=A(0)=\left(\begin{matrix}\tilde\Delta^{(q)}+\left(\frac{m}{2}-q\right)\left(\frac{m}{2}-q-1\right)&2 (-1)^q \tilde d\\
2 (-1)^q \tilde d^\dagger&\tilde\Delta^{(q-1)}+\left(\frac{m}{2}+2-q\right)\left(\frac{m}{2}+1-q\right)\end{matrix}
\right)
\]

Next, by its definition, $A(x)$ satisfies all the requirements at pg. 373 of \cite{BS2}, with $p=1$ (in particular this
follows from the fact that $A(x)$ is defined by the Laplacian on forms on a compact space). We can apply the results of
Br\"uning and Seeley \cite{BS1} \cite{BS2}, observing that in the present case we are in what they call ``constant coefficient
case'' (Section 3 of \cite{BS2}). By Theorem 5.1 of \cite{BS2}, the operator $L$ extends to a unique self adjoint
bounded operator $\Delta^{(q)}$. Note that this extension is the Friedrich extension by Theorem 6.1 of \cite{BS2}. Note
also that bundary condition at $x=0$ are necessary in general in the definition of the domain of $\Delta^{(q)}$, see
(L2) (c), pg. 410 of \cite{BS2} for these conditions.

Eventually, by Theorem 5.2 of \cite{BS2}, the square (here $p=1$, so $m=2$) of the resolvent of $\Delta^{(q)}$ is of trace class. This means that the resolvent is Hilbert Schmidt, and consequently the spectrum of $\Delta^{(q)}$ is pure point, by the spectral theorm for compact operators. Note that we do not need the cut off function $\gamma$ appearing in Theorem 5.2 of \cite{BS2}, since here $0<x\leq l$.

\end{proof}



\begin{lem}\label{l2} \cite{Che2}
Let $\{\varphi_{{\rm har},n}^{(q)},\varphi_{{\rm cex},n}^{(q)},\varphi_{{\rm ex},n}^{(q)}\}$ be an orthonormal  base of $\Gamma(W,\Lambda^{(q)}T^* W)$
consisting of harmonic,  coexact and exact eigenforms of $\tilde\Delta^{(q)}$ on $W$. Let $\lambda_{q,n}$
denotes the  eigenvalue of $\varphi_{{\rm cex},n}^{(q)}$ and $m_{{\rm cex},q,n}$ its multiplicity (so that $m_{{\rm cex},q,n}=\dim \E^{(q)}_{{\rm cex},n}=\dim \E^{(q)}_{{\rm ccl},n}$).  Let $J_\nu$ be the Bessel function of index $\nu$. Define
\begin{align*}
\alpha_q &=\frac{1}{2}(1+2q-m),\\
\mu_{q,n} &= \sqrt{\lambda_{q,n}+\alpha_q^2}.
\end{align*}

Then, assuming that $\mu_{q,n}$ is not an integer, all the solutions of the equation $\Delta u=\lambda^2 u$,
with $\lambda\not=0$, are convergent sums of forms of the following six types:
\begin{align*}
\psi^{(q)}_{\pm, 1,n,\lambda} =& x^{\alpha_q} J_{\pm\mu_{q,n}}(\lambda x) \varphi_{{\rm cex},n}^{(q)},\\
\psi^{(q)}_{\pm, 2,n,\lambda} =& x^{\alpha_{q-1}} J_{\pm\mu_{q-1,n}}(\lambda x) \tilde d\varphi_{{\rm cex},n}^{(q-1)} +
\b_x(x^{\alpha_{q-1}}J_{\pm\mu_{q-1,n}}(\lambda x)) dx \wedge \varphi_{{\rm cex},n}^{(q-1)}\\
\psi^{(q)}_{\pm,3,n,\lambda} =& x^{2\alpha_{q-1}+1}\b_x(x^{-\alpha_{q-1}}J_{\pm\mu_{q-1,n}}(\lambda x) )\tilde d\varphi_{{\rm cex},n}^{(q-1)}\\
\nonumber &+ x^{\alpha_{q-1}-1}J_{\pm\mu_{q-1,n}}(\lambda x) dx \wedge \tilde d^{\dag} \tilde d \varphi_{{\rm cex},n}^{(q-1)}\\
\psi^{(q)}_{\pm,4,n,\lambda} =& x^{\alpha_{q-2}+1}J_{\pm\mu_{q-2,n}}(\lambda x) dx \wedge \tilde d \varphi_{{\rm cex},n}^{(q-2)}\\
\psi^{(q)}_{\pm,E,\lambda} =& x^{\alpha_{q}} J_{\pm|\alpha_{q}|}(\lambda x) \varphi_{{\rm har},n}^{(q)}\\
\psi^{(q)}_{\pm,O,\lambda} =& \b_x(x^{\alpha_{q-1}}J_{\pm|\alpha_{q-1}|}(\lambda x)) dx \wedge \varphi_{{\rm har},n}^{(q-1)}.
\end{align*}

When $\mu_{q,n}$ is an integer the $-$ solutions must be modified including some logarithmic term (see for example \cite{Wat} for a set of linear independent solutions of the Bessel equation).

\end{lem}

\begin{proof}
The proof is a direct verification of the assertion, using the definitions in equations (\ref{f1}), (\ref{f2}), and (\ref{f3}). First, by Hodge theorem, there exist an
orthonormal base of $\Lambda^{(q)}T^* W$ as stated. Thus, we decompose any form $\omega$ in this base. Second, we
compute $\Delta\omega$, using this decomposition and the formula in equation (\ref{f3}). This gives some differential
equations in the functions appearing as coefficients of the forms. All these differential equations reduce to equations
of Bessel type. Third, we write all the solutions using Bessel functions. A complete proof for the case of the harmonic
forms can be found in \cite{Nag} Section 5.
\end{proof}

Note that the forms of types $1$ and $3$ are coexact, those of types  $2$ and $4$ exacts. The operator $d$ sends forms
of types $1$ and $3$ in forms of types $2$ and $4$, while $d^{\dag}$ sends forms of types $2$ and $4$ in forms of types
$1$ and $3$, respectively. The Hodge operator sends forms of type $1$ in forms of type $4$, $2$ in $3$, and  $E$ in
$0$.

\begin{corol}
The functions $+$ in Lemma \ref{l1} are square integrable and satisfy the boundary conditions at $x=0$ defining the
domain of $\Delta_{\rm rel/abs}$. The functions $-$ either are not square integrable or do not satisfy these
conditions. 
\end{corol}

\begin{rem} \label{pip} All the $-$ solutions are either not square or their exterior derivative are not square integrable. Requiring the last condition in the definition of the  domain of $\Delta_{\rm rel/abs}$, it follows that   there are not boundary conditions at zero. This was observed by Cheeger for harmonic forms when the dimension is odd in \cite{Che2} Section 3.
\end{rem}

\begin{lem}\label{l3}
The positive part of the spectrum of the Laplace operator on forms on $C_lW$, with
absolute boundary conditions on $\b C_l W$ is:
\begin{align*}
\Sp_+ \Delta_{\rm abs}^{(q)} &= \left\{m_{{\rm cex},q,n} : \hat j^{2}_{\mu_{q,n},\alpha_q,k}/l^{2}\right\}_{n,k=1}^{\infty}
\cup
\left\{m_{{\rm cex},q-1,n} : \hat j^{2}_{\mu_{q-1,n},\alpha_{q-1},k}/l^{2}\right\}_{n,k=1}^{\infty} \\
&\cup \left\{m_{{\rm cex},q-1,n} : j^{2}_{\mu_{q-1,n},k}/l^{2}\right\}_{n,k=1}^{\infty} \cup \left\{m _{q-2,n} :
j^{2}_{\mu_{q-2,n},k}/l^{2}\right\}_{n,k=1}^{\infty} \\
&\cup \left\{m_{{\rm har},q,0}:\hat j^2_{|\alpha_q|,\alpha_q,k}/l^{2}\right\}_{k=1}^{\infty} \cup \left\{ m_{{\rm har},q-1,0}:\hat
j^2_{|\alpha_{q-1}|,\alpha_q,k}/l^{2}\right\}_{k=1}^{\infty}.
\end{align*}

With relative boundary conditions:
\begin{align*}
\Sp_+ \Delta^{(q)}_{\rm rel} &= \left\{m _{{\rm cex},q,n} : j^{-2s}_{\mu_{q,n},k}/l^{-2s}\right\}_{n,k=1}^{\infty} \cup
\left\{m _{{\rm cex},q-1,n} :j^{-2s}_{\mu_{q-1,n},k}/l^{-2s}\right\}_{n,k=1}^{\infty} \\
&\cup \left\{ m_{{\rm cex},q-1,n} : \hat j^{-2s}_{\mu_{q-1,n},-\alpha_{q-1},k}/l^{-2s}\right\}_{n,k=1}^{\infty} \cup
\left\{m _{{\rm cex},q-2,n} :
\hat j^{-2s}_{\mu_{q-1,n},-\alpha_{q-2},k}/l^{-2s}\right\}_{n,k=1}^{\infty} \\
&\cup \left\{m_{{\rm har},q}:j_{|\alpha_q|,k}/l^{-2s}\right\}_{k=1}^{\infty} \cup \left\{m_{{\rm har},q-1}:
j_{|\alpha_{q-1}|,k}/l^{-2s}\right\}_{k=1}^{\infty},
\end{align*}
where  the $j_{\mu,k}$ are the zeros of the Bessel function $J_{\mu}(x)$,  the $\hat j_{\mu,c,k}$ are the zeros of
the function $\hat J_\mu(x) = c J_\mu (x) + x J'_\mu(x)$,  $c\in \R$, $\alpha_q$ and $\mu_{q,n}$ are defined in Lemma \ref{l2}.
\end{lem}

\begin{proof}
By the Lemma \ref{l1}, Lemma \ref{l2}  and its corollary, we know that the + solutions of Lemma \ref{l2} determine a
complete system of square integrable solutions of the eigenvalues equation $\Delta^{(q)}u=\lambda u$, with
$\lambda\not=0$, satisfying the boundary condition at $x=0$. Since $\Delta_{\rm abs/rel}^{(q)}$ has pure point
spectrum, in order to obtain a discrete resolution (more precisely the positive part of it) of $\Delta_{\rm
abs/rel}^{(q)}$, we have to determine among these solutions those that belong to the domain of $\Delta_{\rm
abs/rel}^{(q)}$, namely those that satisfy the boundary condition at $x=l$. We give details for  absolute BC, the analysis for relative BC is analogous. So consider absolute BC, as given in equation (\ref{abs}). For a form of type 1 to satisfy this condition means
\[
\b_x\psi^{(q)}_{1,n,\lambda}|_{x=l} = 0,
\]
i.e.
\[
\alpha_q J_{\mu_{q,n}}(\lambda l) +\lambda l J'_{\mu_{q,n}}(\lambda l) = 0,
\]
and this gives $\lambda = \hat j_{\mu_{q,n},\alpha_q,k}/ l$. For forms of type 2, we get
\[
\b_x\psi^{(q)}_{2}|_{x=l} = 0, 
\]
that gives
\[
\alpha(q-1) J_{\mu_{q-1,n}}(\lambda l) +\lambda l J'_{\mu_{q-1,n}}(\lambda l) =0,
 \]
 so  $\lambda = \hat j_{\mu_n(i-1),k,\alpha(i-1)}/ l$. For forms of type 3, we obtain the system
\begin{equation*}
\left\{\begin{array}{l}
\left.x^{\alpha_{q-1}-1} J_{\mu_{q-1,n}}(\lambda x)\right|_{x=l} =0,\\
\left. \b_x\left(x^{2\alpha_{q-1}+1} \b_x(x^{-\alpha_{q-1}}J_{\mu_{q,n}}(\lambda x))\right) -\tilde \lambda
x^{\alpha_{q-1}-1} J_{\mu_{q-1,n}}(\lambda x)\right|_{x=l} = 0.
\end{array}
\right.
\end{equation*} 

Using classical properties of Bessel functions and their derivative, we obtain $\lambda = j_{\mu_{q-1},n,k} / l$. For forms of type 4, we get
\[
\left. x^{\alpha_{q-2}+1} J_{\mu_{q-2},n,k}(\lambda l)\right|_{r=l} = 0,
\]
that gives $\lambda =j_{\mu_{q-1},n,k}/l$. Similar analysis gives for forms of types  $E$ and $O$: $\lambda = \hat j_{|\alpha_q|,k,\alpha_q}/l$ and $\lambda = \hat j_{|\alpha_{q-1}|,k,\alpha_{q-1}}/l$, respectively.

\end{proof}

We conclude with the harmonic forms of $\Delta$. The proofs are similar to the previous ones, so will be omitted.

\begin{lem}\label{l2bb} \cite{Che2}\cite{Nag} With the notation of Lemma \ref{l2}, and 
\[
a_{\pm, q,n}=\alpha_q\pm \mu_{q,n},
\]
then all the solutions of the harmonic equation $\Delta u=0$, are convergent sums of forms of the following four types:
\begin{align*}
\psi^{(q)}_{\pm, 1,n} =& x^{a_{\pm,q,n}} \varphi_{{\rm ccl},n}^{(q)},\\
\psi^{(q)}_{\pm, 2,n} =& x^{a_{\pm,q-1,n}} \tilde d\varphi_{{\rm ccl},n}^{(q-1)}+a_{\pm,q-1,n}x^{a_{\pm,q-1,n}-1}dx\wedge \varphi_{{\rm ccl},n}^{(q-1)},\\
\psi^{(q)}_{\pm,3,n} =& x^{a_{\pm,q-1,n}+2} \tilde d\varphi_{{\rm ccl},n}^{(q-1)}+a_{\mp,q-1,n}x^{a_{\pm,q-1,n}+1}dx\wedge \varphi_{{\rm ccl},n}^{(q-1)},\\
\psi^{(q)}_{\pm,4,n} =& x^{a_{\pm,q-2,n}+1}dx\wedge \tilde d\varphi_{{\rm ccl},n}^{(q-2)}.
\end{align*}

\end{lem}


\begin{lem}\label{l3bb} Assume $\dim W=2p-1$ is odd. Then
\begin{align*}
\H^q_{\rm abs}(C_l W)&=\begin{cases}\H^q(W), &0\leq q\leq p-1,\\
 \{0\}, & p\leq q\leq 2p-1.\end{cases}\\
\H^q_{\rm rel}(C_l W)&=\begin{cases} \{0\}, & \hspace{17pt}0\leq q\leq p,\\
\left\{x^{2\alpha_q-1}dx \wedge \varphi^{(q-1)}, \varphi^{(q-1)}\in \H^{q-1}(W)\right\}, &p+1\leq  q\leq 2p.
\end{cases}
\end{align*}

\end{lem}

\begin{proof} First, by Remark \ref{pip}, we need only to consider the $+$ solutions in Lemma \ref{l2bb}. The proof then follows by argument similar to the one used in the proof of Lemma \ref{l3}. Let see one case in details. Consider $\psi_{+,1,n}^{(q)}=x^{a_{+,q,n}} \varphi_{{\rm ccl},n}^{(q)}$, where $a_{+,q,n}=\alpha_q+ \mu_{q,n}$. In order that $\psi_{+,1,n}^{(q)}$ satisfies the absolute boundary condition (\ref{abs}), we need that 
\[
\left. (d \psi_{+,1,n}^{(q)})_{\rm norm}\right|_{\b C_l W}=a_{+,q,n}l^{a_{+,q,n}-1}dx \wedge d \varphi_{{\rm ccl},n}^{(q)}=0
\]
and this is true if and only if $a_{+,q,n}=0$. The condition $a_{+,q,n}=0$ is equivalent to the conditions  $\lambda_{q,n}=0$, and $\alpha_q=-|\alpha_q|$. Therefore, $\varphi_{{\rm ccl},n}^{(q)}$ is harmonic, $0\leq q\leq p-1$, and $\psi_{+,1,n}^{(q)}= \varphi_{{\rm ccl},n}^{(q)}$.

\end{proof}

\section{Torsion zeta function and Poincar\'e duality for a cone}
\label{s2c}

Using the description of the spectrum of the Laplace operator on forms $\Delta_{\rm abs/rel}^{(q)}$ given in  Lemma \ref{l3}, we define the zeta function on $q$-forms as in Section \ref{forms}, by
\[
\zeta(s,\Delta_{\rm abs/rel}^{(q)})=\sum_{\lambda\in \Sp_+\Delta_{\rm abs/rel}^{(q)}} \lambda^{-s},
\]
for $\Re(s)>\frac{m+1}{2}$. Even if we can not apply directly Proposition \ref{ss.l1}, the explicit knowledge of the behaviour of the large eigenvalues allows to completely determine the analytic continuation of the zeta function, by using the tools os Section \ref{sb2.2}. In particular , it is possible to prove that there can be at most a simple, pole at $s=0$. 
We will not do this here (but the interested reader can compare with  \cite{Spr9}), because for our purpose it is more convenient to investigate the analytic properties of other zeta functions, resulting by a suitable different decomposition of the analytic torsion, as described here below. For  we  define the {\it torsion zeta function} by
\[
t_{\rm abs/rel}(s)=\frac{1}{2}\sum_{q=1}^{m+1} (-1)^q q \zeta(s,\Delta_{\rm abs/rel}^{(q)}).
\]

It is clear that the analytic torsion of $C_l W$ is (in the following we will use the simplified notation $T(C_l W)$ for $T((C_lW,g);\rho)$)
\[
\log T_{\rm abs/rel}(C_l W)=t_{\rm abs/rel}'(0).
\]

Our first result is a Poincar\'e duality (compare with Proposition \ref{Poinc1}, \cite{Luc} and the result of \cite{Dar}).

\begin{theo}\label{Poinc2} {\bf Poincar\'e duality for the analytic torsion of a cone}. Let $(W,g)$ be an orientable  compact connected Riemannian manifold of dimension $m$, without boundary, then
\[
\log T_{\rm abs}(C_l W)=(-1)^{m}\log T_{\rm rel}(C_lW).
\]
\end{theo}
\begin{proof} By Hodge duality in equation (\ref{it}), the Hodge operator $\star$ sends forms of type $1,2,3,4,E$, and $O$ into forms of type $4,3,2,1,O$, and $E$, respectively. Moreover, $\star$ sends $q$-forms satisfying absolute boundary conditions, as in equation (\ref{abs}), into $m+1-q$-forms satisfying relative boundary conditions, as in equation (\ref{rel}). Therefore, using the explicit description of the eigenvalues given in Lemma \ref{l3}, it follows that $\Sp\Delta_{\rm abs}^{(q)}=\Sp\Delta_{\rm rel}^{(m+1-q)}$. Using the formulas in equations (\ref{f1}), (\ref{f2}), and (\ref{f3}), and the eigenforms in Lemma \ref{l2}, a straightforward calculation shows that the forms of type $1$, $3$, and $E$ are coexact, and those of type $2$, $4$, and $O$ are exact, and that the operator $d$ sends forms of type $1$, $3$, and $E$ in forms of type $2$, $4$, and $O$, respectively, with inverse $d^\dag$. Then, set
\begin{align*}
F^{(q)}_{\rm ccl,abs}=& \left\{m _{{\rm ccl},q,n} : \hat j^{2}_{\mu_{q,n},\alpha_q,k}/l^{2}\right\}_{n,k=1}^{\infty} \cup
\left\{m_{{\rm ccl},q-1,n} : j^{2}_{\mu_{q-1,n},k}/l^{2}\right\}_{n,k=1}^{\infty}\\
&\cup \left\{m_{{\rm ccl},q,0}:\hat j^2_{|\alpha_q|,\alpha_q,k}/l^{2}\right\}_{k=1}^{\infty},\\
F_{{\rm cl,abs}}^{(q)}=& \left\{m _{{\rm cl},q-1,n} : \hat j^{2}_{\mu_{q-1,n},\alpha_{q-1},k}/l^{2}\right\}_{n,k=1}^{\infty}
\cup
\left\{m _{{\rm cl},q-2,n} : j^{2}_{\mu_{q-2,n},k}/l^{2}\right\}_{n,k=1}^{\infty}\\
&\cup \left\{m_{{\rm cl},q-1,0}:\hat j^2_{|\alpha_{q-1}|,\alpha_{q-1},k}/l^{2}\right\}_{k=1}^{\infty}.
\end{align*}

$F^{(q)}_{{\rm ccl,abs}}$ is the set of the eigenvalues of the coclosed $q$-forms with absolute boundary conditions, and 
$F_{{\rm cl,abs}}^{(q)}$ is the set of the eigenvalues of the closed $q$-forms with absolute. Since obviously $\Sp \Delta^{(q)}_{\rm abs} = F_{{\rm ccl, abs}}^{(q)} \cup F_{{\rm cl, abs}}^{(q)}$, and $F_{{\rm ccl,
abs}}^{(q)} = F_{{\rm cl, abs}}^{(q+1)}$, we have that
\begin{align*}
t_{\rm abs}(s) &=\frac{1}{2} \sum^{m+1}_{q=0} (-1)^{q}q \zeta(s,\Delta^{(q)}_{\rm abs})
= \frac{1}{2}\sum^{m+1}_{q=0} (-1)^{q}q \zeta(s,\Delta^{(m+1-q)}_{\rm rel})\\
&=(-1)^{m}t_{\rm rel}(s) +\frac{1}{2} (m+1)\sum^{m+1}_{q=0} (-1)^{m+1-q} \zeta(s,\Delta^{(q)}_{\rm rel})\\
&=(-1)^{m}t_{\rm rel}(s) + \frac{1}{2}(m+1)\sum^{m+1}_{q=0} (-1)^{q} \left(\zeta(s,F_{{\rm ccl,abs}}^{(q+1)}) + \zeta(s,F_{{\rm cl,abs}}^{(q)})\right)\\
&=(-1)^{m}t_{\rm rel}(s).
\end{align*}

Since by definition $\log T_{\rm abs}(W) = t'_{\rm abs}(0) $, the thesis follows.

\end{proof}

\section{The torsion zeta function of the cone over an odd dimensional manifold}
\label{s2d}

In this section we develop the main steps in order to obtain the proof of our theorems. This accounts essentially in the application of the tools described in Section \ref{sb2.2} to some suitable sequences appearing in the definition of the torsion. So our first step is precisely to obtain this suitable description. This we do in this section. In the next two subsections, we will make the calculations necessary for the proof of our main theorems.

We proceed assuming ${\rm dim} W=2p-1$ odd, and assuming absolute boundary condition; for notational convenience, we will omit the ${\it abs}$ subscript. 

\begin{lem} \label{tps}
\begin{align*}
t(s)=& \frac{l^{2s}}{2} \sum^{p-2}_{q=0} (-1)^q \left(\sum^{\infty}_{n,k=1} m_{{\rm cex},q,n}\left(2j^{-2s}_{\mu_{q,n},k} - \hat j^{-2s}_{\mu_{q,n},\alpha_q,k}-\hat j^{-2s}_{\mu_{q,n},-\alpha_q,k}\right)
\right) \\
&+(-1)^{p-1}\frac{l^{2s}}{2}\left(\sum^{\infty}_{n,k=1} m_{{\rm cex},p-1,n}\left(j^{-2s}_{\mu_{p-1,n},k} -
(j'_{\mu_{p-1,n},k})^{-2s}\right)\right)\\
&- \frac{l^{2s}}{2} \sum_{q=0}^{p-1} (-1)^{q} {\rm rk}\H_q(\b C_lW;\Q) \sum_{k=1}^{\infty} \left(j^{-2s}_{-\alpha_{q-1},k} - j^{-2s}_{-\alpha_{q},k}\right).
\end{align*}
\end{lem}

\begin{proof} Using the eigenvalues in Lemma \ref{l3}
\begin{align*}
l^{2s}\zeta(s,\Delta
^{(q)}) =& \sum_{n,k=1}^{\infty}  m _{{\rm cex},q,n} \hat j^{-2s}_{\mu_{q,n},\alpha_q,k}  +\sum_{n,k=1}^{\infty} m_{{\rm cex},q-1,n} \hat j^{-2s}_{\mu_{q-1,n},\alpha_{q-1},k} + \sum_{n,k=1}^{\infty} m_{{\rm cex},q-1,n}j^{-2s}_{\mu_{q-1,n},k} \\
&+ \sum_{n,k=1}^{\infty} m_{{\rm cex},q-2,n}
j^{-2s}_{\mu_{q-2,n},k}+\sum_{k=1}^\infty m_{{\rm har},q,0} \hat j_{|\alpha_q|,\alpha_q,k}^{-2s}+\sum_{k=1}^\infty m_{{\rm har},q-1,0} \hat j_{|\alpha_{q-1}|,\alpha_{q-1},k}^{-2s}.
\end{align*}

Since for each fixed $q$, with $0\leq q\leq 2p-2$,
\begin{align*}
&(-1)^q q \sum^{\infty}_{n,k=1}  m_{{\rm cex},q,n} \hat j^{-2s}_{\mu_{q,n},\alpha_q,k} +
(-1)^{q+1}(q+1) \sum^{\infty}_{n,k=1}  m_{{\rm cex},q,n} \hat j^{-2s}_{\mu_{q,n},\alpha_q,k}\\
&+ (-1)^{q+1}(q+1)  \sum^{\infty}_{n,k=1} m_{{\rm cex},q,n}j^{-2s}_{\mu_{q,n},k}+(-1)^{q+2}(q+2)
\sum^{\infty}_{n,k=1} m_{{\rm cex},q,n}j^{-2s}_{\mu_{q,n},k}\\
&+q(-1)^{q}\sum_{k=1}^{\infty} m_{{\rm har},q,0}\hat j^{-2s}_{|\alpha_{q}|,\alpha_q,k} + (q+1)(-1)^{q+1}\sum_{k=1}^{\infty}
m_{{\rm har},q,0}\hat j^{-2s}_{|\alpha_{q}|,\alpha_q,k} \\
=&(-1)^q \left(\sum^{\infty}_{n,k=1} m_{{\rm cex},q,n}j^{-2s}_{\mu_{q,n},k} - \sum^{\infty}_{n,k=1} m_{{\rm cex},q,n} \hat
j^{-2s}_{\mu_{q,n},\alpha_q,k}\right) + (-1)^{q+1} \sum_{k=1}^{\infty} m_{{\rm har},q,0}\hat j^{-2s}_{|\alpha_{q}|,\alpha_q,k}.
\end{align*}
it follows that
\begin{align*}
t(s)&= \frac{l^{2s}}{2} \sum^{2p-2}_{q=0} (-1)^q \sum^{\infty}_{n,k=1} m_{{\rm cex},q,n}\left( j^{-2s}_{\mu_{q,n},k} - \hat
j^{-2s}_{\mu_{q,n},\alpha_q,k}\right) +\frac{l^{2s}}{2}\sum^{2p-1}_{q=0} (-1)^{q+1} \sum_{k=1}^\infty m_{{\rm har},q,0} \hat
j_{|\alpha_q|,k}^{-2s}.
\end{align*}

Next, by Hodge duality on coexact $q$-forms on the section (see equation (\ref{it}))  $\lambda_{q,n} = \lambda_{2p-2-q,n}$, and recalling the definition of the constants $\alpha_q$ and $\mu_{q,n}$ in Lemma \ref{l2}, we have that $\alpha_{q} = \frac{1}{2}(1+2q-2p+1) = q-p+1 = -\alpha_{2p-2-q}$, and $\mu_{q,n} =
\mu_{2p-2-q,n}$. Thus, fixing $q$ with $0\leq q\leq p-2$,
\begin{align*}
&(-1)^q \sum^{\infty}_{n,k=1} m_{{\rm cex},q,n} \left( j^{-2s}_{\mu_{q,n},k} - \hat j^{-2s}_{\mu_{q,n},\alpha_q,k} \right)
+ (-1)^{(2p-2-q)} \sum^{\infty}_{n,k=1} m_{{\rm cex},q,n} \left(j^{-2s}_{\mu_{q,n},k} - \hat j^{-2s}_{\mu_{q,n},-\alpha_q,k}\right)  \\
=& (-1)^{q} \sum^{\infty}_{n,k=1} m_{{\rm cex},q,n} \left(2 j^{-2s}_{\mu_{q,n},k} -\hat j^{-2s}_{\mu_{q,n},\alpha_q,k} - \hat
j^{-2s}_{\mu_{q,n},-\alpha_q,k}\right),
\end{align*}
while when $q=p-1$, $\lambda_{p-1,n} = \lambda_{2p-1-(p-1),n}$, and  $\alpha_q=0$. Therefore,
\begin{align*}
t(s)=& \frac{l^{2s}}{2} \sum^{p-2}_{q=0} (-1)^q \sum^{\infty}_{n,k=1} m_{{\rm cex},q,n}\left(2 j^{-2s}_{\mu_{q,n},k}
-\hat j^{-2s}_{\mu_{q,n},\alpha_q,k} - \hat j^{-2s}_{\mu_{q,n},-\alpha_q,k}\right)\\
&+ (-1)^{p-1}\frac{l^{2s}}{2} \sum^{\infty}_{n,k=1} m_{{\rm cex},p-1,n} \left( j^{-2s}_{\mu_{p-1,n},k} -
(j'_{\mu_{p-1,n},k})^{-2s} \right)\\
&+\frac{l^{2s}}{2}\sum^{2p-1}_{q=0} (-1)^{q+1} \sum_{k=1}^\infty m_{{\rm har},q,0} \hat j_{|\alpha_q|,\alpha_q,k}^{-2s},
\end{align*}
where $j'_{\nu,k}=\hat j_{\nu,0,k}$ are the zeros of $J'_{nu}$.
Eventually, consider the sum
\[
\sum^{2p-1}_{q=0} (-1)^{q+1} \sum_{k=1}^\infty m_{{\rm har},q,0} \hat j_{|\alpha_q|,\alpha_q,k}^{-2s}.
\] 

We will use some classical properties of Bessel function, see for example \cite{Wat}. 
Recall $m=\dim W =2p-1$, and therefore  $\alpha_q=q-p+1$ is an integer. Moreover, $\alpha_q$ is negative for
$0\leq q < p-1$. Fixed such a  $q$, we study the function $\hat J_{-\alpha_q,\alpha_q}(z) = \alpha_q
J_{-\alpha_q}(z) +z J'_{-\alpha_q}(z)$. Since
\[
zJ'_{\mu}(z) = -z J_{\mu+1}(z) + \mu J_\mu(z)
\] 
it follows that $\hat J_{-\alpha_q,\alpha_q}(z) = -z J_{-\alpha_q + 1}(z) =- zJ_{-\alpha_{q-1}}(z)$, and hence $\hat
j_{|\alpha_{q}|,\alpha_q,k} = j_{-\alpha_{q-1},k}$. Next, fix $q$ with $p-1<q\leq2p-1$, such that  $\alpha_q$ is a positive integer. Then, since
\[
zJ'_{\mu}(z) = z J_{\mu-1}(z) - \mu J_\mu(z),
\]
the function $\hat J_{\alpha_q,\alpha_q}(z) = \alpha_q J_{\alpha_q}(z) +z J'_{\alpha_q}(z)$ coincides with $zJ_{\alpha_q -1}(z)$, and hence $\hat j_{|\alpha_{q}|,\alpha_q,k} = j_{\alpha_{q-1},k}$. 
Note that when $q = p-1$, $\alpha_{p-1} =0$ and hence $j_{\alpha_{p-1},\alpha_{p-1},k} = j'_{0,k} = j_{1,k}$. Summing up,
\begin{align*}
\sum^{2p-1}_{q=0} (-1)^{q+1} \sum_{k=1}^\infty m_{{\rm har},q,0} \hat j_{|\alpha_q|,\alpha_q,k}^{-2s} =&
\sum_{q=0}^{p-2} (-1)^{q+1}\sum_{k=1}^{\infty}\frac{m_{{\rm har},q,0}}{j^{-2s}_{-\alpha_{q-1},k}} + (-1)^{p}
\sum_{k=1}^{\infty}\frac{m_{{\rm har},p-1,0}}{j^{-2s}_{1,k}} \\
&+\sum_{q=p}^{2p-1}(-1)^{q+1}\sum_{k=1}^{\infty}\frac{m_{{\rm har},q,0}}{j^{-2s}_{\alpha_{q-1},k}},
\end{align*}
and since by Hodge duality $m_{q,0} = m_{2p-1-q,0}$, 
\begin{align*}
=& \sum_{q=0}^{p-2}
(-1)^{q+1}\sum_{k=1}^{\infty}m_{{\rm har},q,0}j^{-2s}_{-\alpha_{q-1},k} + (-1)^{p} \sum_{k=1}^{\infty} m_{{\rm har},p-1,0}j^{-2s}_{1,k} \\
&+\sum_{q=0}^{p-1}(-1)^{2p-q}\sum_{k=1}^{\infty}m_{{\rm har},2p-1-q,0}j^{-2s}_{\alpha_{2p-2-q},k}\\
+& \sum_{q=0}^{p-2} (-1)^{q+1}\sum_{k=1}^{\infty} m_{{\rm har},q,0} j^{-2s}_{-\alpha_{q-1},k} + (-1)^{p} \sum_{k=1}^{\infty}
m_{{\rm har},p-1,0}j^{-2s}_{1,k} +
\sum_{q=0}^{p-1}(-1)^{q}\sum_{k=1}^{\infty} m_{{\rm har},q,0} j^{2s}_{-\alpha_{q},k}\\
=& \sum_{q=0}^{p-1} (-1)^{q+1} m_{{\rm har},q,0} \sum_{k=1}^{\infty} \left(j^{-2s}_{-\alpha_{q-1},k} - j^{-2s}_{-\alpha_{q},k}\right).
\end{align*}

Since  $m_{{\rm har},q,0} = {\rm rk} \H_q(\b C_lW;\Q)$, this completes the proof.

 \end{proof}

It is convenient to introduce the following functions.  We set
\beq\label{z1}
\begin{aligned}
Z_q(s)&=\sum^{\infty}_{n,k=1} m_{{\rm cex},q,n}j^{-2s}_{\mu_{q,n},k},\\
\dot Z_q(s)&=\sum^{\infty}_{n,k=1} m_{{\rm cex},q,n}(j'_{\mu_{q,n},k})^{-2s},\\
Z_{q,\pm}(s)&=\sum^{\infty}_{n,k=1} m_{{\rm cex},q,n}\hat j^{-2s}_{\mu_{q,n},\pm\alpha_q,k},\\
z_q(s)&=\sum_{k=1}^{\infty} \left(j^{-2s}_{-\alpha_{q-1},k} - j^{-2s}_{-\alpha_{q},k}\right),
\end{aligned}
\eeq
for $0\leq q \leq p-1$, and
\beq\label{tq}
\begin{aligned}
t_{p-1}(s)&=Z_{p-1}(s)-\dot Z_{p-1}(s),&\\
t_q(s)&=2Z_q(s)-Z_{q,+}(s)-Z_{q,-}(s),&0\leq q\leq p-2.\\
\end{aligned}
\eeq

Then,
\begin{align*}
t(s)=& \frac{l^{2s}}{2} \sum^{p-2}_{q=0} (-1)^q \left(2Z_q(s)-Z_{q,+}(s)-Z_{q,-}(s)\right)+(-1)^{p-1}\frac{l^{2s}}{2}\left(Z_{p-1}(s)-\dot Z_{p-1}(s)\right)\\
&- \frac{l^{2s}}{2} \sum_{q=0}^{p-1} (-1)^{q} {\rm rk}\H_q(\b C_lW;\Q)z_q(s)\\
=&\frac{l^{2s}}{2} \sum^{p-1}_{q=0} (-1)^q t_q(s)- \frac{l^{2s}}{2} \sum_{q=0}^{p-1} (-1)^{q} {\rm rk}\H_q(\b C_lW;\Q)z_q(s),
\end{align*}
and
\beq\label{ttt}
\begin{aligned}
\log T(C_l W)=t'(0)=& \frac{\log l^{2}}{2} \left(\sum_{q=0}^{p-1} (-1)^{q+1} r_qz_q(0)+\sum^{p-1}_{q=0}(-1)^{q} t_q(0)\right) \\
&+ \frac{1}{2} \left(\sum_{q=0}^{p-1} (-1)^{q+1} r_q z_q'(0)+\sum^{p-1}_{q=0}(-1)^{q} t_q'(0)\right) ,
\end{aligned}
\eeq
where $r_q={\rm rk}\H_q(\b C_lW;\Q)$. In order to obtain the value of $\log T(C_l W)$ we use Theorem \ref{t4} and its corollary applied to the functions $z_q(s)$, $Z_q(s)$, $\dot Z_q(s)$, $Z_{q,\pm}(s)$. More precisely, the functions $z_q$ were studied in Section \ref{ss2}, and we will study the functions $t_q$ in Sections \ref{s3} and \ref{s4}, and eventually we sum up on the forms degree $q$ in Section \ref{s5}.


\subsection{The function $t_{p-1}(s)$}
\label{s3}

In this section we study the function $t_{p-1}(s)$. For we apply Theorems \ref{tt} and \ref{t4} to the double sequences $S_{p-1}=\{m_{p-1,n}:j_{\mu_{p-1,n},k}^2\}_{n=1}^{\infty} $ and $\dot
S_{p-1}=\{m_{p-1,n}:(j_{\mu_{p-1,n},k}')^2\}_{n=1}^{\infty}$, since $Z_{p-1}(s)=\zeta(s,S_{p-1})$, $\dot Z_{p-1}(s)=\zeta(s,\dot S_{p-1})$. First, we verify Definition \ref{spdec}. For we introduce the simple sequence $U_{p-1}=\{m_{p-1,n}:\mu_{p-1,n}\}_{n=1}^\infty$.

\begin{lem} \label{lp1} The sequence $U_{p-1}$ is a totally regular sequence of spectral type with infinite order, $\ec(U_{p-1})=\ge(U_{p-1})=2p-1$, and $\zeta(s,U_{p-1})=\zeta_{cex}\left(\frac{s}{2},\tilde\Delta^{(p-1)}\right)$,
with possible simple poles  at $s=2p-1-h$, $h=0,2,4, \dots$.

\end{lem}

\begin{lem} The logarithimic Gamma functions associated to the sequences
$S_{p-1,n}/\mu_{p-1,n}^2$ and $\dot S_{p-1,n}/\mu_{p-1,n}^2$ have the following  representations, with  $\lambda\in D_{\theta,c}$, $0\leq\theta\leq \pi$, $c=\frac{{\rm min}(j_{\mu_{p-1,1}}^2,(j_{\mu_{p-1,1}}')^2)}{2\mu_{p-1,1}^2}$,
\begin{align*}
\log \Gamma(-\lambda,S_{p-1,n}/\mu_{p-1,n}^2)=&-\log\prod_{k=1}^\infty
\left(1+\frac{(-\lambda)\mu_{p-1,n}^2}{j_{\mu_{p-1,n},k}^2}\right)\\
=&-\log I_{\mu_{p-1,n}}(\mu_{p-1,n}\sqrt{-\lambda})+(\mu_{p-1,n})\log\sqrt{-\lambda} \\
&+\mu_{p-1,n}\log (\mu_{p-1,n})-\mu_{p-1,n}\log 2-\log\Gamma(\mu_{p-1,n}+1),\\
\log \Gamma(-\lambda,\hat S_{p-1,n}/\mu_{p-1,n}^2)=&-\log\prod_{k=1}^\infty
\left(1+\frac{(-\lambda)(\mu_{p-1,n})^2}{(j_{\mu_{p-1,n},k}')^2}\right)\\
=&-\log I'_{\mu_{p-1,n}}(\mu_{p-1,n}\sqrt{-\lambda})+(\mu_{p-1,n}-1)\log\sqrt{-\lambda} \\
&+\mu_{p-1,n}\log (\mu_{p-1,n})-\mu_{p-1,n}\log 2-\log\Gamma(\mu_{p-1,n}+1).
\end{align*}
\label{s3.l1}
\end{lem}


Using the expansions of the Bessel functions, it follow from   Lemma \ref{s3.l1} that there is a complete asymptotic expansion for the functions $\log \Gamma(-\lambda,S_{p-1,n}/\mu_{p-1,n}^2)$ and $\log \Gamma(-\lambda,\dot S_{p-1,n}/\mu_{p-1,n}^2)$, and then the sequences $S_{p-1,n}/\mu_{p-1,n}^2$ and $\dot S_{p-1,n}/\mu_{p-1,n}^2$ are sequence of spectral type. A simple calculation shows that they are totally regular sequences of infinite order.

\begin{prop}\label{s3.p1} The double sequences $S_{p-1}$ and $\dot S_{p-1}$ have relative exponents $\left(p,\frac{2p-1}{2},\frac{1}{2}\right)$, relative genus $\left(p,p-1,0\right)$, and are spectrally decomposable over $U_{p-1}$ with power $\kappa=2$, length $\ell=2p$ and domain $D_{\theta,c}$. The coefficients $\sigma_h$ appearing in equation (\ref{exp}) are $\sigma_h=h-1$, with $h=0,1,\dots,\ell=2p$. 
\end{prop}
\begin{proof} The values of the exponents and genus of $S$ follow by classical estimates of the zeros of the Bessel functions \cite{Wat}, and zeta function theory. In particular, to determinate $s_0=p$, we use the Young inequality and the Plana theorem as in \cite{Spr0}. Note that $\alpha >\frac{1}{2}$, since $s_2 = \frac{1}{2}$. The exponents and genus of $\dot S$ are the same, since the zeros of the derivative of the Bessel function are correlated by those of the Bessel function, namely $j_{\mu,k}<j'_{\mu,k}<j_{\mu,k+1}$. As observed, the existence of a complete asymptotic expansion of the Gamma function
$\log \Gamma(-\lambda,S_{p-1,n})$ and $\log \Gamma(-\lambda,\dot S_{p-1,n})$ follows by Lemma \ref{s3.l1}. This implies that $S_{p-1,n}$ and $\dot S_{p-1,n}$ are sequences of spectral type. A direct inspection of the expansions shows that $S_{p-1,n}$ and $\dot S_{p-1,n}$ are totally regular sequences of infinite order. The existence of the uniform expansion follows using the uniform expansions for the Bessel functions and their derivative given for example in \cite{Olv} (7.18) and Ex. 7.2, and classical expansion of the Euler Gamma function \cite{GZ} 8.344. We refer to \cite{HS} Section 5 or to \cite{HMS} Section 4 for details. This proves that $S_{p-1,n}$ and $\dot S_{p-1,n}$ are spectrally decomposable over $U_{p-1}$, with power $\kappa=2$. The length $\ell$ of the decomposition is precisely $2p$. For $\ec(U_{p-1})=2p-1$, and therefore the larger integer such that $\sigma_h=h-1\leq 2p-1$ is $2p$. 
\end{proof}

\begin{rem}\label{rrr} By Theorem \ref{tt}, only the terms with
$\sigma_h=1$, $\sigma_h=3$, $\ldots$, $\sigma_h=2p-1$ namely $h=2,4,\ldots,2p$, appear in the formula of Theorem
\ref{t4}, since the unique non negative poles of $\zeta(s,U_{p-1})$ are at $s=1$, $s=3$, $\ldots$\  $s=2p-1$, by Lemma \ref{lp1}. 
\end{rem}

Since we aim to apply the version of Theorem \ref{t4} given in Corollary \ref{c}, for linear combination of two spectrally decomposable sequences, we need more information on the uniform asymptotic expansion of the difference $S_{p-1}-\dot S_{p-1}$. This will also give the functions $\phi_{\sigma_h}(\lambda)$, necessary in the following.

\begin{lem}\label{s3.l2} The difference of the logarithimic Gamma functions associated to the sequences $S_{p-1,n}/\mu_{p-1}^2$ and $\dot S_{p-1,n}/\mu_{p-1,n}^2$ have the following  uniform asymptotic expansions for large $n$, $\lambda\in D_{\theta,c}$,
\begin{align*}
\log \Gamma(-\lambda, S_{p-1,n}&/(\mu^p_{p-1,n})^2) - \log \Gamma(-\lambda,\dot S_{p-1,n}/(\mu^p_{p-1,n})^2) =
\\&= -\log I(\mu_{p-1,n}\sqrt{-\lambda}) +\log I'(\mu_{p-1,n}\sqrt{-\lambda}) + \log \sqrt{-\lambda}\\
&=\frac{1}{2} \log (1-\lambda) + \sum^{2p-1}_{j=1} \phi_{j,p-1}(\lambda)\frac{1}{(\mu^{p}_{p-1,n})^{j}}
+O\left(\frac{1}{\mu_{p-1,n}^{2p}}\right).
\end{align*}
\end{lem}
\begin{proof} By Lemma \ref{s3.l1}
\begin{equation}\label{eq1}
\begin{aligned}
\log \Gamma(-\lambda, S_{p-1,n}&/(\mu_{p-1,n})^2) - \log \Gamma(-\lambda,\dot S_{p-1,n}/(\mu_{p-1,n})^2)=\\
&= -\log I(\mu_{p-1,n}\sqrt{-\lambda}) +\log I'(\mu_{p-1,n}\sqrt{-\lambda}) + \log \sqrt{-\lambda}.
\end{aligned}
\end{equation}

Recall the uniform expansions for the Bessel functions given
for example in \cite{Olv} (7.18) pg. 376, and Ex. 7.2,
\[
I_{\nu}(\nu z)=\frac{\e^{\nu\sqrt{1+z^2}}\e^{\nu\log\frac{z}{1+\sqrt{1+z^2}}}}{\sqrt{2\pi
\nu}(1+z^2)^\frac{1}{4}}\left(1+\sum_{j=1}^{2p-1}\frac{U_{j}(z)}{\nu^{j}}+O\left(\frac{1}{\nu^{2p}}\right)\right),
\]
where
\begin{align*}
U_0(w)=&1,\\
U_j(w)=&\frac{1}{2}w^2(1-w^2)\frac{d}{dw}U_{j-1}(w)+\frac{1}{8}\int_{0}^{w}(1-5t^2)U_{j-1}(t)dt,\\
\end{align*}
with $w=\frac{1}{\sqrt{1+z^{2}}}$, and
\[
I_{\nu}'(\nu z)=\frac{(1+z^2)^\frac{1}{4}\e^{\nu\sqrt{1+z^2}}\e^{\nu\log\frac{z}{1+\sqrt{1+z^2}}}}{\sqrt{2\pi
\nu}z}\left(1+\sum^{2p-1}_{j=1} \frac{V_j(z)}{\nu^{j}}+O\left(\frac{1}{\nu^{2p}}\right)\right),
\]
\begin{align*}
V_0(w)=&1,\\
V_j(w)=& U_j(w)-\frac{w}{2}(1-w^2)U_{j-1}(w)-w^2(1-w^2)\frac{d}{dw}U_{j-1}(w).\\
\end{align*}

Then, 
\begin{align*}
\log I_{\nu}(\nu z) & = \nu \sqrt{1+z^2} + \nu \log z - \nu\log(1+\sqrt{1+z^2}) \\
&-\frac{1}{2}\log2\pi\nu - \frac{1}{4}\log(1+z^2) +
\log\left(1+\sum_{j=1}^{2p-1}\frac{U_{j}(z)}{\nu^{j}}+O\left(\frac{1}{\nu^{2p}}\right)\right)\\
\log I'_{\nu}(\nu z) & = \nu \sqrt{1+z^2} + \nu \log z - \nu\log(1+\sqrt{1+z^2})-\log z \\
&-\frac{1}{2}\log2\pi\nu + \frac{1}{4}\log(1+z^2) +
\log\left(1+\sum_{j=1}^{2p-1}\frac{V_{j}(z)}{\nu^{j}}+O\left(\frac{1}{\nu^{2p}}\right)\right),
\end{align*}
and substitution in equation (\ref{eq1}) gives
\begin{equation*}
\begin{aligned}
\log \Gamma(-\lambda, S_{p-1,n}/&(\mu_{p-1,n})^2) - \log \Gamma(-\lambda,\dot S_{p-1,n}/(\mu_{p-1,n})^2)\\
=&\frac{1}{2}\log(1-\lambda)- \log\left(1+\sum_{j=1}^{2p-1}\frac{U_{j}(\sqrt{-\lambda})}{\mu_{p-1,n}^{j}}+O\left(\frac{1}{\mu_{p-1,n}^{2p}}\right)\right)\\
&+ \log\left(1+\sum_{j=1}^{2p-1}\frac{V_{j}(\sqrt{-\lambda})}{\mu_{p-1,n}^{j}}+O\left(\frac{1}{\mu_{p-1,n}^{2p}}\right)\right).
\end{aligned}
\end{equation*}

Expanding the logarithm as
\[
\log \left(1+ \sum^{\infty}_{j=1}\frac{a_j}{z^{j}} \right) = \sum^{\infty}_{j=1} \frac{l_j}{z^{j}},
\] 
where $a_0=1$, $a_1 = l_1$ and 
\[
l_j= a_j - \sum_{k=1}^{j-1} \frac{j-k}{j} a_k l_{j-k},
\]
we have that
\begin{equation*}
\begin{aligned}
\log \Gamma(&-\lambda, S_{p-1,n}/(\mu_{p-1,n})^2) - \log \Gamma(-\lambda,\dot S_{p-1,n}/(\mu_{p-1,n})^2)=
\frac{1}{2}\log(1-\lambda)\\
&+ \sum^{2p-1}_{j=1} \frac{1}{\mu_{p-1,n}^{j}}\left(V_j(\sqrt{-\lambda}) - U_j(\sqrt{-\lambda}) + \sum^{j-1}_{k=1}
\frac{j-k}{j}\left(V_k(\sqrt{-\lambda}) \dot l_{j-k}(\lambda) - U_k(\sqrt{-\lambda}) l_{j-k}(\lambda)\right) \right) \\
&+O\left(\frac{1}{\mu_{p-1,n}^{2p}}\right),
\end{aligned}
\end{equation*} 
where we denote by $\dot l_j(\lambda)$ the term in the expansion relative to the sequence $\dot S$ (thus the one containing the  $V_j(z)$) and by   $l_j(\lambda)$
the term relative to $S$ (thus the one containing the $U_j(z)$). Setting
\begin{equation}\label{phi_p-1}
\begin{aligned}
\phi_{p-1,j}(\lambda) &= \dot l_j(\lambda)-l_j(\lambda)\\
&= V_j(\sqrt{-\lambda}) - U_j(\sqrt{-\lambda}) + \sum^{j-1}_{k=1} \frac{j-k}{j}\left(V_k(\sqrt{-\lambda}) \dot l_{j-k}(\lambda) - U_k(\sqrt{-\lambda})
l_{j-k}(\lambda)\right),
\end{aligned}
\end{equation}
we have the formula stated in the thesis.


\end{proof}

\begin{rem} Note that there are no logarithmic terms $\log\mu_{p-1,n}$ in the asymptotic expansion of the difference of the logarithmic Gamma function given in Lemma \ref{s3.l2}. This permits to apply Corollary \ref{c}.
\end{rem}

Next, give some results on the functions $\phi_{j,p-1}(\lambda)$, and on the functions $\Phi_{j,p-1}(s)$ defined in equation (\ref{fi1}).

\begin{lem} For all $j$, the functions $\phi_{j,p-1}(\lambda)$ are odd polynomial in $w=\frac{1}{\sqrt{1-\lambda}}$  
\[
\phi_{j,p-1}(\lambda) =  \sum^{3j+1}_{k=j}a_{j,p-1,k}w^{2k+1}.
\] 
\end{lem}
\begin{proof} This follows by the definition in equation (\ref{phi_p-1}) in the proof of Lemma \ref{s3.l2}.
\end{proof}

\begin{lem} For all $j$, $\phi_{j,p-1}(0)=0$.
\end{lem}
\begin{proof} We use induction on $j$. For $j=1$ 
\begin{align*}
\phi_{1,p-1}(\lambda) &= \dot l_1(\lambda) - l_1(\lambda) = V_1(\sqrt{-\lambda}) - U_1(\sqrt{-\lambda})\\
&=-\frac{1}{2}\frac{1}{(1-\lambda)^{\frac{1}{2}}}+\frac{1}{2}\frac{1}{(1-\lambda)^{\frac{3}{2}}},
\end{align*}
and hence $\phi_{1,p-1}(0) = 0$. Assume $\phi_{k,p-1}(0) = 0$ for $k=1,\ldots,j-1$. For simplicity, write $\phi$ as a function of $w= (1-\lambda)^{-\frac{1}{2}}$. Then, by definition 
\begin{align*}
\phi_{j,p-1}(w)= \dot l_j(w) - l_j(w)&\\
=V_j(w) - U_j(w) &+ \sum^{j-1}_{s=1} \frac{j-s}{j}\left(U_s(w)l_{j-s}(w) - V_s(w)\dot l_{j-s}(w)\right)\\
=V_j(w) - U_j(w) &+ \sum^{j-1}_{s=1} \frac{j-s}{j}\left(U_s(w)(l_{j-s}(w) - \dot l_{j-s}(w))\right)\\
                 &+\sum^{j-1}_{s=1} (1-w^2)\left(\frac{w}{2}U_{s-1}(w)+w^2\frac{d}{dw}U_{s-1}(w)\right)\dot l_{j-s}(w).
\end{align*}

Since $U_j(1)=V_j(1)$ for all $j$, and $\phi_{j-1,p-1}(w)= \dot l_{j-1}(w) - l_{j-1}(w)$,  using the induction's hypothesis, the thesis follows.

\end{proof}

\begin{corol}\label{c33} For all $j$, the Laurent expansion of the functions $\Phi_{2j+1,p-1}(s)$ at $s=0$ has coefficients: for $1\leq j \leq p-1$ \begin{align*}
\Rz_{s=0}\Phi_{2j+1,p-1}(s) &=2\sum^{3j+1}_{k=j+1} a_{j,p-1,k}\sum^{k}_{t=2}\frac{1}{2t-1} ,&\Ru_{s=0}\Phi_{2j+1,p-1}(s) &=0,\\
\end{align*}
and for $j=0$
\begin{align*}
\Rz_{s=0}\Phi_{1,p-1}(s) &=2a_{0,p-1,1}=1,&\Ru_{s=0}\Phi_{1,p-1}(s) &=0.\\
\end{align*}
\end{corol}
\begin{proof} The proof follows from the definition in equation (\ref{fi1}), classical results on the Euler Gamma function (see equation (\ref{residuo Gamma}) in the appendix), and calculation based on the formula (\ref{e}) in the appendix. See \cite{Har} for further details. 
The formula  for $j=0$ follows by explicit knowledge of the coefficients $a_{0,1}$.
\end{proof}

Next, we determine the terms $A_{0,0}(0)$ and $A'_{0,1}(0)$, defined in equation (\ref{fi2}). Using an upper dot to denote the ones relative to the sequence $\dot S_{p-1}$, we have the following result.

\begin{lem} \label{s3.l4}
\begin{align*}
\A_{0,0,p-1}(s)&=A_{0,0,p-1}(s)-\dot A_{0,0,p-1}(s)=0,\\
\A_{0,1,p-1}(s)&=A_{0,1,p-1}(s)- \dot A_{0,1,p-1}(s)
=\frac{1}{2}\zeta(2s, U_{p-1}).
\end{align*}
\end{lem}
\begin{proof} The defining equation (\ref{fi2}) reads 
\begin{align*}
A_{0,0,p-1}(s)&=\sum_{n=1}^\infty m_{{\rm cex},p-1,n}\left(a_{0, 0,n,p-1} -\sum_{j=1}^{p}b_{2j-1,0,0,p-1}\mu_{p-1,n}^{-2j+1}\right)\mu_{p-1,n}^{-2 s},\\
A_{0,1,p-1}(s)&=\sum_{n=1}^\infty m_{{\rm cex},p-1,n}\left(a_{0, 1,n,p-1} -\sum_{j=1}^{p}b_{2j-1,0,1,p-1}\mu_{p-1,n}^{-2j+1}\right)\mu_{p-1,n}^{-2 s}.
\end{align*}
in the present case. We need the expansion  of the functions $\log\Gamma(-\lambda, S_{p-1,n}/\mu_{p-1,n}^2)$,
$l_{2j-1}(\lambda)$, $\log\Gamma(-\lambda, \dot S_{p-1,n}/\mu_{p-1,n}^2)$, and $\dot l_{2j-1}(\lambda)$, for
$j=1,2,\ldots,p$, and large $\lambda$. Using classical expansions for the Bessel functions and their derivative and the formulas in equation
(\ref{form}), we obtain (see \cite{Har} for further details)
\begin{align*}
&a_{0,0,n,p-1}=\frac{1}{2}\log 2\pi+\left(\mu_{p-1,n}+\frac{1}{2}\right)\log\mu_{p-1,n}-\mu_{p-1,n}\log 2-\log\Gamma(\mu_{p-1,n}+1),\\
&a_{0,1,n,p-1}=\frac{1}{2}\left(\mu_{p-1,n}+\frac{1}{2}\right),\\
&b_{2j-1,0,0,p-1}=0,\hspace{5pt} b_{2j-1,0,1,p-1}=0 ,\hspace{5pt}j=1,2,\ldots p,
\end{align*}
note that $b_{2j-1,0,0,p-1}=0$ since $l_{2j-1}(\lambda)$ don't have constant term.
\begin{align*}
&\dot a_{0,0,n,p-1}=\frac{1}{2}\log 2\pi+\left(\mu_{p-1,n}+\frac{1}{2}\right)\log\mu_{p-1,n}-\mu_{p-1,n}\log 2-\log\Gamma(\mu_{p-1,n}+1),\\
&\dot a_{0,1,n,p-1}=\frac{1}{2}\left(\mu_{p-1,n}-\frac{1}{2}\right),\\
&\dot b_{2j-1,0,0,p-1}=0,\hspace{5pt} \dot b_{2j-1,0,1,p-1}=0 ,\hspace{5pt}j=1,2,\ldots p,
\end{align*}
and $\dot b_{2j-1,0,0,p-1}=0$ since $\dot l_{2j-1}(\lambda)$ don't have constant term. The thesis follows.

\end{proof}

We now have all the necessary information to apply Theorem \ref{t4} and its corollary. We obtain the following result, where we distinguish the {\it regular part} and the {\it singular part}, as in Remark \ref{rqr}.

\begin{prop}\label{pro2}  \begin{align*}
t_{p-1}(0)&=t_{p-1,{\rm reg}}(0)+t_{p-1,{\rm sing}}(0),\\
t_{p-1}'(0)&=t_{p-1,{\rm reg}}'(0)+t_{p-1,{\rm sing}}'(0),
\end{align*}
where
\begin{align*}
t_{p-1,{\rm reg}}(0)&=-\frac{1}{2}\zeta(0, U_{p-1})=-\frac{1}{2}\zeta_{\rm cex}\left(0,\tilde\Delta^{(p-1)}\right),\\
t_{p-1,{\rm sing}}(0)&=0,\\
t_{p-1,{\rm reg}}'(0)&=-\zeta'(0, U_{p-1})=-\frac{1}{2}\zeta_{\rm cex}'\left(0,\tilde\Delta^{(p-1)}\right),\\
t_{p-1,{\rm sing}}'(0)&=\frac{1}{2}\sum_{j=0}^{p-1}\Rz_{s=0}\Phi_{2j+1,q}(s)\Ru_{s=2j+1}\zeta(s,U_{p-1})\\
&=\frac{1}{2}\sum_{j=0}^{p-1}\Rz_{s=0}\Phi_{2j+1,q}(s)\Ru_{s=2j+1}\zeta_{\rm cex}\left(\frac{s}{2},\tilde\Delta^{(p-1)}\right).
\end{align*}
\end{prop}
\begin{proof} By definition in equations (\ref{z1}) and (\ref{tq}),
\begin{align*}
t_{p-1}(0)=&Z_{p-1}(0)-\dot Z_{p-1}(0),\\
t_{p-1}'(0)=&Z_{p-1}'(0)-\dot Z_{p-1}'(0),\\
\end{align*}
where $Z_{p-1}(s)=\zeta(s,S_{p-1})$, and $\dot Z_{p-1}(s)=\zeta(s,\dot S_{p-1})$. By Proposition \ref{s3.p1} and Lemma \ref{s3.l2}, we can apply Theorem \ref{t4} and its Corollary to the difference of these double zeta functions. The regular part of $Z_{p-1}(0)-\dot Z_{p-1}(0)$ is then given in Lemma \ref{s3.l4}, while the singular part vanishes, since, by Corollary \ref{c33}, the residues of the functions $\Phi_{k,p-1}(s)$ at $s=0$ vanish. 
The regular part of $Z_{p-1}'(0)-\dot Z_{p-1}'(0)$ again follows by Lemma \ref{s3.l4}. For the singular part,   since by Proposition \ref{s3.p1}, $\kappa=2$, $\ell=2p$, and $\sigma_h=h-1$, with $0\leq h\leq 2p$, by Remark \ref{rrr} we need only the odd values of $h-1=2j+1$, $0\leq j\leq p-1$, and this gives the formula stated for $t_{p-1,{\rm sing}}'(0)$.
\end{proof}

\subsection{The functions $t_q(s)$, $0\leq q\leq p-2$}
\label{s4}

In this section we study the functions $t_q(s)$. For we apply Theorems \ref{tt} and \ref{t4} to the double sequences $S_{q}=\{m_{q,n}:j_{\mu_{q,n},k}^2\}_{n=1}^{\infty} $ and $S_{q,\pm}=\{m_{q,n}:j_{\mu_{q,n}\pm\alpha_q,k}^2\}_{n=1}^{\infty}$, since we have that $Z_q(s)=\zeta(s,S_q)$, $Z_{q,\pm}(s)=\zeta(s, S_{q,\pm})$, where $q=0,1,\ldots,p-2$, $\alpha_q =p-q-1$. 
Note that the sequence $S_q$ coincides with the sequence $S_{p-1}$ analysed in Section \ref{s3}, with $q=p-1$. So we just need to study the other two sequences.
First, we verify Definition \ref{spdec}. For we introduce the simple sequence 
$U_q=\{m_{q,n}:\mu_{q,n}\}_{n=1}^\infty$. 

\begin{lem}\label{lel} For all $0\leq q\leq p-1$, the sequence $U_{q}$ is a totally regular sequence of spectral type with infinite order, $\ec(U_{q})=\ge(U_{q})=2p-1$, and 
\[
\zeta(s,U_{q})=\zeta_{cex}\left(\frac{s}{2},\tilde\Delta^{(q)}+\alpha^2_q\right).
\] 

The  possible poles of $\zeta(s,U_{q})$ are at $s=2p-1-h$, $h=0,2,4,\dots$,  and the residues are completely determined by the residues of the function $\zeta_{cex}(s,\tilde\Delta^{(q)})$, namely:
\begin{align*}
\Ru_{s=2k+1} \zeta(s,U_q) &=  \sum^{p-1-k}_{j=0} \binom{-\frac{2k+1}{2}}{j} \Ru_{s=2(k+j)+1}\zeta_{cex}\left(\frac{s}{2},\tilde \Delta^{(q)}\right)\alpha_q^{2j}.
\end{align*}

\end{lem} 
\begin{proof} By definition $U_q=\{m_{{\rm cex},q,n}:\mu_{q,n}\}_{n=1}^\infty$, where by Lemmas \ref{l2} and \ref{l3}
\[
\mu_{q,n}=\sqrt{\lambda_{q,n}+\alpha_q^2},
\]
and the $\lambda_{q,n}$ are the eigenvalues of the operator $\tilde \Delta^{(q)}$ on the compact manifold $W$. Counting such eigenvalues according to multiplicity of the associated coexact eigenform, since the dimension of the eigenspace of $\lambda_{q,n}$ are finite, by Proposition \ref{ss.l1}, $\lambda_{q,n}\sim n^\frac{2}{m}$ for large $n$. This gives order and genus. Next, by definition
\begin{align*}
\zeta(s,U_q)&=\sum_{n=1}^\infty m_{{\rm cex},q,n}(\lambda_{q,n}+\alpha^2_q)^{-\frac{s}{2}}=\sum_{j=0}^\infty \binom{-\frac{s}{2}}{j} \sum_{n=1}^\infty m_{{\rm cex},q,n}\lambda_{q,n}^{-\frac{s}{2}-j}\alpha_q^{2j}\\
&=\sum_{j=0}^\infty \binom{-\frac{s}{2}}{j} \zeta_{cex}\left(\frac{s}{2}+j,\tilde\Delta^{(q)}\right)\alpha_q^{2j}=\zeta_{cex}\left(\frac{s}{2},\tilde\Delta^{(q)}\right)-\frac{s}{2}\zeta_{cex}\left(\frac{s}{2}+1,\tilde\Delta^{(q)}\right)\alpha_q^2+\dots.
\end{align*}

The last statement follows by Proposition \ref{l3.1}.

\end{proof}

The analysis of the double sequences $S_{q,\pm}$ is little bit harder than that of the sequences of the previous Section \ref{s3}, since now the elements of the sequences are not multiple of zeros of Bessel functions (and their derivative). However, they are the zeros of some linear combinations of Bessel functions and their derivative, and this makes possible the following analysis.

For $c\in \C$, let define the functions
\[
\hat J_{\nu,c}(z)=c J_{\nu}(z)+zJ'_\nu(z).
\]

Recalling the series definition of the Bessel function
\[
J_\nu(z)=\frac{z^\nu}{2^\nu}\sum_{k=0}^\infty \frac{(-1)^kz^{2k}}{2^{2k}k!\Gamma(\nu+k+1)},
\]
we obtain that near $z=0$
\[
\hat J_{\nu,c}(z) =\left(1+\frac{c}{\nu}\right) \frac{z^\nu}{2^\nu\Gamma(\nu)}.
\]

This means that the function $z^{-\nu} \hat J_{\nu,c}(z)$ is an even function of
$z$. Let $\hat j_{\nu,c,k}$ be the positive zeros of $\hat J_{\nu,c}(z)$ arranged in increasing order. By
the Hadamard factorization theorem, we have the product expansion
\[
z^{-\nu} \hat J_{\nu(z),c}=z^{-\nu} \hat
J_{\nu,c}(z){\prod_{k=-\infty}^{+\infty}}\left(1-\frac{z}{\hat j_{\nu,c,k}}\right),
\]
and therefore
\[
\hat J_{\nu,c}(z)=\left(1+\frac{c}{\nu}\right)\frac{z^\nu}{2^\nu\Gamma(\nu)}
\prod_{k=1}^{\infty}\left(1-\frac{z^2}{\hat j^2_{\nu,c,k}}\right).
\]

Next,  recalling that (when $-\pi<\arg(z)<\frac{\pi}{2}$)
\begin{align*}
J_\nu(iz)&=\e^{\frac{\pi}{2}i\nu} I_\nu(z),\\
J'_\nu(iz)&=\e^{\frac{\pi}{2}i\nu}\e^{-\frac{\pi}{2}i} I'_\nu(z),\\
\end{align*}
we obtain
\[
\hat J_{\nu,c}(iz)=\e^{\frac{\pi}{2}i\nu}\left(c I_\nu(z)+zI'_\nu(z)\right).
\]

Thus, we define (for $-\pi<\arg(z)<\frac{\pi}{2}$)
\begin{equation}\label{pop}
\hat I_{\nu,c}(z)=\e^{-\frac{\pi}{2}i\nu} \hat J_{\nu,c}(i z),
\end{equation} 
and hence 
\begin{align}\label{s4.e1}
\hat I_{\nu,\pm\alpha_q}(z)&=\pm\alpha_q I_\nu(z)+zI'_\nu(z)=\left(1\pm\frac{\alpha_q}{\nu}\right)\frac{z^\nu}{2^\nu\Gamma(\nu)}
\prod_{k=1}^{\infty}\left(1+\frac{z^2}{\hat j^2_{\nu,\pm\alpha_q,k}}\right).
\end{align}

Recalling the definition in equation (\ref{gamma}) we have proved the following fact.

\begin{lem}\label{s4.l1} The logarithmic Gamma functions associated to the sequences
$S_{q,\pm,n}$ have the following representations, when $\lambda\in D_{\theta,c'}$, with $c'=\frac{1}{2}{\rm min}(j_{\mu_{q,0}}^2,j_{\mu_{q,0},\pm\alpha_q}^2)$,
\begin{align*}
\log \Gamma(-\lambda,S_{q,\pm,n})=&-\log\prod_{k=1}^\infty \left(1+\frac{(-\lambda)}{\hat j_{\mu_{q,n},\pm\alpha_q,k}^2}\right)\\
=&-\log \hat I_{\mu_{q,n},\pm\alpha_q}(\sqrt{-\lambda})+\mu_{q,n}\log\sqrt{-\lambda}-\mu_{q,n}\log 2\\
&-\log\Gamma(\mu_{q,n})+\log\left(1\pm\frac{\alpha_q}{\mu_{q,n}}\right).
\end{align*}

\end{lem}

\begin{prop}\label{s4.p1}  The double sequences $S_{q,\pm}$  have relative exponents $\left(p,\frac{2p-1}{2},\frac{1}{2}\right)$, relative genus $(p,p-1,0)$, and are is spectrally decomposable over $U_{q}$ with power $\kappa=2$, length $\ell=2p$ and domain $D_{\theta,c'}$. The coefficients $\sigma_h$ appearing in equation (\ref{exp}) are $\sigma_h=h-1$, with $h=1,2,\dots,\ell=2p$. 
\end{prop}
\begin{proof} The proof is the same of the one of Proposition \ref{s3.p1}.

\end{proof}

\begin{rem}\label{rrr4} By Theorem \ref{tt}, only the term with
$\sigma_h=1$, $\sigma_h=3$, $\ldots$, $\sigma_h=2p-1$ namely $h=2,4,\ldots,2p$, appear in the formula of Theorem
\ref{t4}, since the unique poles of $\zeta(s,U_{q})$ are at $s=1$, $s=3$, $\ldots$\  $s=2p-1$. 
\end{rem}

Since we aim to apply the version of Theorem \ref{t4} given in Corollary \ref{c}, for linear combination of two spectrally decomposable sequences, we inspect directly the uniform asymptotic expansion of $2 S_{q}-S_{q,-}-S_{q,+}$. This give the functions $\phi_{\sigma_h}$.

\begin{lem}\label{s4.l2} We have the  the following  asymptotic expansions for large $n$, uniform in $\lambda$, for $\lambda\in D_{\theta,c'}$,
\begin{align*}
2\log&\Gamma(-\lambda,S_{q,n}/\mu_{q,n}^2) - \log \Gamma(-\lambda,S_{q,+,n}/\mu_{q,n}^2) - \log
\Gamma(-\lambda,S_{q,-,n}/\mu_{q,n}^2) \\
=&-2\log I_{\mu_{q,n}}(\mu_{q,n}\sqrt{-\lambda})+\log \hat I_{\mu_{q,n},\alpha_q}(\mu_{q,n}\sqrt{-\lambda})+\log
\hat I_{\mu_{q,n,-\alpha_q}}(\mu_{q,n}\sqrt{-\lambda})\\&- 2\log
\mu_{q,n}-\log\left(1-\frac{\alpha_q^{2}}{\mu_{q,n}^2}\right)\\
=&\log(1-\lambda)+\sum^{2p-1}_{j=1}\phi_{j,q}(\lambda)\frac{1}{\mu_{q,n}^j}+O\left(\frac{1}{(\mu_{q,n})^{2p}}\right).
\end{align*}
\end{lem}
\begin{proof} Using the representations given in Lemmas \ref{s3.l2} and \ref{s4.l1}, we obtain 
\begin{align*}
2\log\Gamma(-\lambda,S_{q,n}/\mu_{q,n}^2)& - \log \Gamma(-\lambda,S_{q,+,n}/\mu_{q,n}^2) - \log
\Gamma(-\lambda,S_{q,-,n}/\mu_{q,n}^2)\\
=&-2\log I_{\mu_{q,n}}(\mu_{q,n}\sqrt{-\lambda})+\log \hat I_{\mu_{q,n},\alpha_q}(\mu_{q,n}\sqrt{-\lambda})+\log \hat
I_{\mu_{q,n,-\alpha_q}}(\mu_{q,n}\sqrt{-\lambda})\\&- 2\log
\mu_{q,n}-\log\left(1-\frac{\alpha_q^{2}}{\mu_{q,n}^2}\right).
\end{align*}

Using the expansion given in Lemma \ref{s3.l2} for  $I_\nu(\nu z)$ and $I'_{\nu}(\nu z)$, we obtain the following expansion for $\hat I_{\nu,\pm \alpha_q}(\nu z)$,
\begin{align*}
\hat I_{\nu,\pm \alpha_q}(\nu z) &= \pm\alpha_q I_{\nu}(\nu z) + \nu z I'_\nu(\nu z)\\
&=\sqrt{\nu}(1+z^2)^\frac{1}{4}\frac{\e^{\nu\sqrt{1+z^2}}\e^{\nu\log\frac{z}{1+\sqrt{1+z^2}}}}{\sqrt{2\pi
}}\left(1+\sum_{j=1}^{2p-1}W_{{\pm\alpha_q},j}(z)\frac{1}{\nu^{j}} + O\left(\frac{1}{\nu^{2p}}\right)\right),
\end{align*} 
where $W_{\pm\alpha_q,j}(z) = V_j(z) \pm \frac{\alpha_q}{\sqrt{1+z^2}}U_j(z)$. Thus,
\begin{align*}
\log \hat I_{\nu,\pm \alpha_q}(\nu z) =& \nu \sqrt{1+z^{2}} + \nu \log z - \nu \log(1+\sqrt{1+z^{2}}) + \log \nu +
\frac{1}{4}\log(1+z^{2})\\
&- \frac{1}{2} \log2\pi\nu + \log\left(1+\sum_{j=1}^{2p-1}W_{{\pm\alpha_q},j}(z)\frac{1}{\nu^{j}} +
O\left(\frac{1}{\nu^{2p}}\right)\right).
\end{align*}

This gives,
\begin{align*}
&2\log\Gamma(-\lambda,S_{q,n}/\mu_{q,n}^2) - \log \Gamma(-\lambda,S_{q,+,n}/\mu_{q,n}^2) - \log
\Gamma(-\lambda,S_{q,-,n}/\mu_{q,n}^2)\\
&= \log(1-\lambda)-2\log\left(1+\sum_{j=1}^{2p-1}\frac{U_{j}(\sqrt{-\lambda})}{\mu_{q,n}^{j}}+O\left(\frac{1}{\mu_{q,n}^{2p}}\right)\right)\\
& +
\log\left(1+\sum_{j=1}^{2p-1} \frac{W_{+\alpha_q,j}(\sqrt{-\lambda})}{\mu_{q,n}^{j}} + O\left(\frac{1}{\mu_{q,n}^{2p}}\right)\right)+\log\left(1+\sum_{j=1}^{2p-1} \frac{W_{-\alpha_q,j}(\sqrt{-\lambda})}{\mu_{q,n}^{j}} + O\left(\frac{1}{\mu_{q,n}^{2p}}\right)\right).
\end{align*}

Using the same expansion for the logarithm as in the proof of Lemma \ref{s3.l2}
\begin{align*}
2\log\Gamma(-\lambda,S_{q,n}/\mu_{q,n}^2) &- \log \Gamma(-\lambda,S_{q,+,n}/\mu_{q,n}^2) - \log
\Gamma(-\lambda,S_{q,-,n}/\mu_{q,n}^2)\\
=& \log(1-\lambda)+\sum_{j=1}^{p} \left(-2 l_{2j-1}(\lambda) + l^{+}_{2j-1}(\lambda) +
l^{-}_{2j-1}(\lambda)\right)\frac{1}{\mu_{q,n}^{2j-1}}\\
&+\sum_{j=1}^{p-1} \left(-2 l_{2j}(\lambda) + l^{+}_{2j}(\lambda) + l^{-}_{2j}(\lambda) +
\frac{\alpha_q^{2j}}{j}\right)\frac{1}{\mu_{q,n}^{2j}} +O\left(\frac{1}{\mu_{q,n}^{2p}}\right),
\end{align*} 
where we denote by $l_j(\lambda)$ the term in the expansion relative to the sequence $ S$ (thus the one containing the  $U_j(z)$) and by   $l^{\pm}_j(\lambda)$ the terms relative to $S_\pm$ (thus the ones containing the $W_{\pm\alpha_q,j}(z)$). Setting
\begin{equation}\label{phi_q}
\begin{aligned}
\phi_{2j-1,q}(\lambda) &=-2 l_{2j-1}(\lambda) + l^{+}_{2j-1}(\lambda) + l^{-}_{2j-1}(\lambda) \\
\phi_{2j,q}(\lambda) &=  -2 l_{2j}(\lambda) + l^{+}_{2j}(\lambda) + l^{-}_{2j}(\lambda)+ \frac{\alpha^{2j}_q}{j},
\end{aligned}
\end{equation}
the result follows.

\end{proof}

\begin{rem} Note that there are no logarithmic terms $\log\mu_{q,n}$ in the asymptotic expansion of the difference of the logarithmic Gamma function given in Lemma \ref{s4.l2}. This permits to apply Corollary \ref{c}.
\end{rem}

Next, we give some results on the functions $\phi_{j,q}(\lambda)$, and on the functions $\Phi_{j,q}(s)$ defined in equation (\ref{fi1}).

\begin{lem}\label{l14} For all $j$ and all $0\leq q\leq p-2$, the functions $\phi_{j,q}(\lambda)$ are odd polynomial in $w=\frac{1}{\sqrt{1-\lambda}}$  
\begin{align*}
\phi_{2j-1,q}(\lambda) &=  \sum^{2j-1}_{k=0}a_{2j-1,q,k}w^{2k+2j-1},\\
\phi_{2j,q}(\lambda) &=  \sum^{2j}_{k=0}a_{2j,q,k}w^{2k+2j}+\frac{\alpha_q^{2j}}{j}.
\end{align*}

The coefficients $a_{j,q,k}$ are completely determined by the coefficients of the expansion given in Lemma \ref{s4.l2}.
\end{lem}
\begin{proof} This follows by direct inspection of the last equality in the statement of Lemma \ref{s4.l2}.
\end{proof}

\begin{lem} \label{L} For all $j$ and all $0\leq q\leq p-2$, $\phi_{j,q}(0)=0$.
\end{lem}
\begin{proof} The proof is by induction on $j$. We will consider all the functions as functions of $w=\frac{1}{\sqrt{1-\lambda}}$. We use the following hypothesis for the induction, for $1\leq k\leq j-1$:
\begin{align}
\label{ii1}\phi_{2k-1,q}(1)&=0,\\
\label{ii2}\phi_{2k,q}(1)&=0,\\
\label{ii3}l^{-}_{2k-1}(1) - l^{+}_{2k-1}(1)& = \frac{-2\alpha_q^{2k-1}}{2k-1},\\
\label{ii4}l^{-}_{2k}(1) - l^{+}_{2k}(1) &= 0,
\end{align} 
where the functions $\phi_{j,q}(\lambda)$ are defined in equation (\ref{phi_q}), and the function $l(\lambda)$ in the course of the proof of Lemma \ref{s4.l2}.

First, we verify the hypothesis for $j=1$. Equations (\ref{ii3}) and (\ref{ii4}) follow by the definition when $k=1$. For equations (\ref{ii1}) and (\ref{ii2}), we have by definition when $k=1$ that
\begin{align*}
\phi_{1,q}(\lambda) &=-2 l_{1}(\lambda) + l^{+}_{1}(\lambda) + l^{-}_{1}(\lambda)\\
&=-2U_1(\sqrt{-\lambda}) + V_1(\sqrt{-\lambda}) + \alpha_qU_0(\sqrt{-\lambda}) +V_1(\sqrt{-\lambda}) -\alpha_qU_0(\sqrt{-\lambda})\\
&=-\frac{1}{(1-\lambda)^{\frac{1}{2}}}+\frac{1}{(1-\lambda)^{\frac{3}{2}}},
\end{align*} 
and
\begin{align*}
\phi_{2,q}(\lambda) &=-2 l_{2}(\lambda) + l^{+}_{2}(\lambda) + l^{-}_{2}(\lambda) + \alpha_q^2\\
&=-2U_2(\sqrt{-\lambda}) + 2V_2(\sqrt{-\lambda}) + U_1(\sqrt{-\lambda})^2 - V_1(\sqrt{-\lambda})^2\\
&=-\frac{3}{2}\frac{1}{(1-\lambda)}+ 2\frac{1}{(1-\lambda)^{2}} -\frac{3}{2} \frac{1}{(1-\lambda)^{3}} + 1,
\end{align*}
and hence equations (\ref{ii1}) and (\ref{ii2}) are also verified when $k=1$. 

Second we prove that the equations (\ref{ii1}), (\ref{ii2}), (\ref{ii3}), and (\ref{ii4}) hold for $k=j$.
Recalling that $U_k(1)=V_k(1)$ for all $k$, we have from the definition that
\begin{align*}
l^{-}_{2j-1}(1) - l^{+}_{2j-1}(1) 
=& U_{2j-1}(1) - \alpha_qU_{2j-2}(1) - U_{2j-1}(1) - \alpha_qU_{2j-2}(1)\\
&-\sum^{2j-2}_{k=1} \frac{2j-1-k}{2j-1} \left(U_{k}(1) (l^{-}_{2j-1-k}(1) - l^{+}_{2j-1-k}(1)) \right)\\
&+\sum^{2j-2}_{k=1} \frac{2j-1-k}{2j-1} \left(\alpha_qU_{k-1}(1)(l^{-}_{2j-1-k}(1) + l^{+}_{2j-1-k}(1)) \right),
\end{align*} 
and hence, using the hypothesis  we obtain
\begin{align*}
l^{-}_{2j-1}(1) &- l^{+}_{2j-1}(1)\\
=&-2\alpha_qU_{2j-2}(1) +\sum^{j-1}_{k=1} \frac{2(j-k)}{2j-1} \alpha_qU_{2k-2}(1)\left(2 l_{2(j-k)} - \frac{\alpha_q^{2(j-k)}}{j-k}\right) \\
&-\sum^{j-1}_{k=1}  U_{2k}(1) \frac{-2\alpha_q^{2(j-k)-1}}{2j-1}
+\sum^{j-1}_{k=1} \frac{2(j-k)-1}{2j-1} 2 \alpha_qU_{2k-1}(1)l_{2(j-k)-1}(1)  \\
=&- \frac{2\alpha_q^{2j-1}}{2j-1} -2\alpha_qU_{2j-2}(1)   +\frac{2\alpha_q}{2j-1}U_{2j-2}(1) \\
&+\frac{2\alpha_q}{2j-1}\left(2(j-1)l_{2j-2} + \sum_{k=1}^{2j-3} (2j-2-k) \alpha_qU_{k}(1)l_{2j-2-k}(1)\right)\\
=&- \frac{2\alpha_q^{2j-1}}{2j-1} -2\alpha_qU_{2j-2}(1)   +\frac{2\alpha_q}{2j-1}U_{2j-2}(1) +
\frac{2\alpha_q(2j-2)U_{2j-2}}{2j-1} \\
=&- \frac{2\alpha_q^{2j-1}}{2j-1},
\end{align*}
thus proving (\ref{ii3}) for $k=j$. For (\ref{ii1}), note that
\begin{align*}
\phi_{2j-1,q}(1) =& -2 l_{2j-1}(1) + l^{+}_{2j-1}(1) + l^{-}_{2j-1}(1)\\
=&\sum^{2j-2}_{k=1} \frac{2j-1-k}{2j-1}\left(U_k(1)(2 l_{2j-1-k}(1) - l^{+}_{2j-1-k}(1) - l^{-}_{2j-1-k}(1))\right)\\
&-\sum^{2j-2}_{k=1} \frac{2j-1-k}{2j-1}\alpha_qU_{k-1}(1)\left(l^{-}_{2j-1-k}(1) - l^{+}_{2j-1-k}(1)\right),
\end{align*}
and using the induction hypothesis, and (\ref{ii3}) with $k=j$ just proved, this means that 
\begin{align*}
\phi_{2j-1,q}(1)=&\sum^{j}_{k=1}\frac{2j-1-(2k-1)}{2j-1} U_{2k-1}(1) \frac{\alpha(i)^{2(j-k)}}{j-k}\\
&-2\sum^{j}_{k=1} \frac{2j-1-2k}{2j-1}\alpha(i)U_{2k-1}(1)\left(\frac{\alpha(i)^{2(j-k)-1}}{2(j-k)-1}\right) =0.
\end{align*}
and (\ref{ii1}) with $k=j$ follows. For (\ref{ii4}), using the hypothesis, we have
\begin{align*}
l^{-}_{2j}&(1) - l^{+}_{2j}(1) \\
=& U_{2j}(1) - \alpha_qU_{2j-1}(1) - U_{2j}(1) - \alpha_qU_{2j-1}(1)\\
&-\sum^{2j-1}_{k=1} \frac{2j-k}{2j} \left(U_{k}(1) (l^{-}_{2j-k}(1) - l^{+}_{2j-k}(1)) \right)\\
&+\sum^{2j-1}_{k=1} \frac{2j-k}{2j} \left(\alpha_qU_{k-1}(1)(l^{-}_{2j-k}(1) + l^{+}_{2j-k}(1)) \right)\\
=&-2\alpha_qU_{2j-1}(1) +\sum^{j-1}_{k=1} \frac{2j-2k}{2j} \alpha_qU_{2k-1}(1)\left(2l_{2j-2k}(1) -\frac{\alpha_q^{2j-2k}}{j-k}\right)\\
&+\sum^{j}_{k=1} \frac{2(j-k)+1}{2j} \left(U_{2k-1}(1) \frac{2\alpha_q^{2(j-k)+1}}{2(j-k)+1} + \alpha_qU_{2k-2}(1)2l_{2(j-k)+1}(1)\right)\\
=&-2\alpha_qU_{2j-1}(1) +  \frac{2\alpha_qU_{2j-1}(1)}{2j} + \frac{2\alpha_q}{2j}(2j-1)l_{2j-1}(1) + 2\alpha_q\sum^{2j-1}_{k=2} \frac{2j-k}{2j} U_{k-1}(1)l_{2j-k}(1)\\
=&-2\alpha_qU_{2j-1}(1) +  \frac{\alpha_q}{j}\left(U_{2j-1}(1)+(2j-1)l_{2j-1}(1) + \sum^{2j-2}_{k=1} (2j-1-k) U_{k}(1)l_{2j-1-k}(1)\right)\\
=&-2\alpha_qU_{2j-1}(1) +   \frac{(\alpha_q(2j-1)+\alpha_q)U_{2j-1}(1)}{j} =0.
\end{align*}

Eventually, for (\ref{ii2})
\begin{align*}
\phi_{2j,q}(1) =&- 2 l_{2j}(1) + l^{+}_{2j}(1) + l^{-}_{2j}(1) + \frac{\alpha_q^{2j}}{j}\\
=&\frac{\alpha_q^{2j}}{j} + \sum^{2j-1}_{k=1} \frac{2j-k}{2j}\left(U_k(1)(2 l_{2j-k}(1) - l^{+}_{2j-k}(1) - l^{-}_{2j-k}(1))\right)\\
&-\sum^{2j-1}_{k=1} \frac{2j-k}{2j}\alpha_q U_{k-1}(1)\left(l^{-}_{2j-k}(1) - l^{+}_{2j-k}(1))\right)\\
=&\sum^{j-1}_{k=1}\frac{2j-2k}{2j} U_{2k}(1) \frac{\alpha_q^{2(j-k)}}{j-k}-2\sum^{j}_{k=2}
\alpha_qU_{2k-2}(1) \frac{\alpha_q^{2(j-k)+1}}{2j} =0,
\end{align*}

\end{proof}

We also give a recurrence relation satisfied by the functions $\phi_{j,q}(\lambda)$, that will be fundamental in the proof of Theorem \ref{t03}.

\begin{lem}\label{l15} For all $j$ and all $0\leq q\leq p-2$, the functions $\phi_{j,q}(w)$ satisfy the following recurrence relations (where $w=\frac{1}{\sqrt{1-\lambda}}$)
\begin{align*}
\phi_{2j-1,q}(\lambda) &= w^{2j-2}\alpha_q^{2j-2}\phi_{q,1}(w) +\sum^{j-2}_{t=1} K_{2j-1,t}(w)\alpha_q^{2t} + 2\phi_{2j-1,p-1}(w)\\
\phi_{2j,q}(\lambda) &= -\frac{(w^{2j}-1)\alpha_q^{2j}}{j} +\sum^{j-1}_{t=1} K_{2j,t}(w)\alpha_q^{2t} + 2\phi_{2j,p-1}(w),
\end{align*} 
where the  $K_{j,t}(w)$ are polynomials in $w$.
\end{lem}
\begin{proof} The proof is by induction on $j$. For $j=1$,
\begin{align*}
\phi_{1,q}(w) &= 2\phi_{1,p-1}(w) = -w + w^{3} \\
\phi_{2,q}(w) &=-(w^{2}-1)\alpha_q^{2} + 2\phi_{2,p-1}(w) = -(w^{2}-1)\alpha_q^{2} + (-\frac{w^2}{2}-2 w^4 -
\frac{3w^6}{2}).
\end{align*} 

Assuming the formulas hold for $1\leq k\leq j-2$. Then, by definition of the  functions $\phi_{j,q}(\lambda)$ and  $l(\lambda)$ in the proof of Lemma \ref{s4.l2}, we have that
\begin{align*}
l^{+}_{2s-1}(w)+l^{-}_{2s-1}(w) &= 2\dot l_{2s-1}(w)+w^{2s-2}\alpha_q^{2s-2}(\phi_{q,1}(w)) +\sum^{s-2}_{t=1} K_{2s-1,t}(w)\alpha_q^{2t},\\
l^{+}_{2s}(w)+l^{-}_{2s}(w)     &= 2\dot l_{2s}(w) - \frac{w^{2s}\alpha_q^{2s}}{s} +\sum^{s-1}_{t=1} K_{2s,t}(w)\alpha_q^{2t},\\
l^{+}_{2s-1}(w) - l^{-}_{2s-1}(w) &=
\frac{2}{2s-1}\alpha_q^{2s-1}w^{2s-1}
+\alpha_q\sum^{s-2}_{t=0}{D}_{2s-1,t}(w)\alpha_q^{2t},\\
l^{+}_{2s}(w) - l^{-}_{2s}(w) &= -\alpha_q^{2s-1}w^{2s-1}\phi_{q,1}(w) +\alpha_q\sum^{s-2}_{t=0}
{D}_{2s,t}(w)\alpha_q^{2t},
\end{align*} 
for all $s=1,2,\ldots,j-1$, and where the  $D_{s,t}$ are polynomials in $w$. We proceed as in the proof of Lemma \ref{L} (see \cite{Har} for further details). For the odd index we have:
\begin{align*}
l^{+}_{2j-1}(w) - l^{-}_{2j-1}(w) =& 2\alpha_q U_{2j-2}(w) -\sum^{2j-2}_{k=1}
\frac{2j-1-k}{2j-1}V_k(w)(l^{+}_{2j-1-k}(w)-l^{-}_{2j-1-k}(w))\\
&+\sum^{2j-2}_{k=1} \frac{2j-1-k}{2j-1}w\alpha_q U_{k-1}(w)(l^{+}_{2j-1-k}(w)+l^{-}_{2j-1-k}(w)),\\
=& 2\alpha_qU_{2j-2}(w) -\sum^{j-1}_{k=1}
\frac{2j-1-2k}{2j-1}V_{2k}(w)(l^{+}_{2j-1-2k}(w)-l^{-}_{2j-1-2k}(w))\\
&-\sum^{j-1}_{k=1} \frac{2j-1-2k}{2j-1}w\alpha_q U_{2k-1}(w)(l^{+}_{2j-1-2k}(w)+l^{-}_{2j-1-2k}(w))
\end{align*}
\begin{align*}
&-\sum^{j-1}_{k=1} \frac{2j-2k}{2j-1}V_{2k-1}(w)(l^{+}_{2j-2k}(w)-l^{-}_{2j-2k}(w))\\
&-\sum^{j-1}_{k=1} \frac{2j-2k}{2j-1}w\alpha_qU_{2k-2}(w)(l^{+}_{2j-2k}(w)+l^{-}_{2j-2k}(w))\\
=&\frac{2}{2j-1}\alpha_q^{2j-1}w^{2j-1}+\alpha_q\sum^{j-2}_{t=0}{D}_{2j-1,t}(w)\alpha_q^{2t},
\end{align*}
and this gives 
\begin{align*}
\phi_{2j-1,q}(w)  
=&- 2U_{2j-1}(w) + 2V_{2j-1}(w) +\sum^{2j-2}_{k=1} \frac{2j-1-k}{2j-1} \left(2 U_{k}(w)l_{2j-1-k}(w)\right)\\
&- \sum^{j-1}_{k=1} \frac{2j-1-2k}{2j-1} \left(V_{2k}(w) (l^{+}_{2j-1-2k}(w) +l^{-}_{2j-1-2k}(w))\right) \\
&- \sum^{j-1}_{k=1} \frac{2j-1-2k}{2j-1}\left(w\alpha_q U_{2k-1}(w)(l^{+}_{2j-1-2k}(w)-l^{-}_{2j-1-2k}(w))\right)\\
&- \sum^{j-1}_{k=1} \frac{2j-2k}{2j-1} \left(V_{2k-1}(w) (l^{+}_{2j-2k}(w) +l^{-}_{2j-2k}(w)) \right)\\
&- \sum^{j-1}_{k=1} \frac{2j-2k}{2j-1} \left(w\alpha_qU_{2k-2}(w)(l^{+}_{2j-2k}(w)-l^{-}_{2j-2k}(w))\right)\\
=&w^{2j-2}\alpha_q^{2j-2}(\phi_{1,q}(w)) +\sum^{j-2}_{t=1} K_{2j-1,t}(w)\alpha_q^{2t} + 2\phi_{2j-1,p-1}(w).
\end{align*}

For the even index, using the result proved for the odd index, we get
\begin{align*}
l^{+}_{2j}(w) - l^{-}_{2j}(w) 
=& 2\alpha_q U_{2j-1}(w) -\sum^{j-1}_{k=1}
\frac{2j-2k}{2j}V_{2k}(w)(l^{+}_{2j-2k}(w)-l^{-}_{2j-2k}(w))\\
&-\sum^{j-1}_{k=1} \frac{2j-2k}{2j}w\alpha_q U_{2k-1}(w)(l^{+}_{2j-2k}(w)+l^{-}_{2j-2k}(w))\\
&-\sum^{j-1}_{k=1} \frac{2j-2k+1}{2j}V_{2k-1}(w)(l^{+}_{2j-2k+1}(w)-l^{-}_{2j-2k+1}(w))\\
&-\sum^{j-1}_{k=1} \frac{2j-2k+1}{2j}w\alpha_q U_{2k-2}(w)(l^{+}_{2j-2k+1}(w)+l^{-}_{2j-2k+1}(w))\\
=&-\alpha_q^{2j-1}w^{2j-1}\phi_{1,q}(w) +\alpha_q\sum^{j-2}_{t=0} {D}_{2j,t}(w)\alpha_q^{2t},
\end{align*}
and proceeding as before, this gives the last formula in the thesis.

\end{proof}

\begin{corol}\label{c44} For all $j$ and all $0\leq q\leq p-2$, the Laurent expansion of the functions $\Phi_{2j+1,q}(s)$ at $s=0$ has coefficients: for $1\leq j \leq p-1$ 
\begin{align*}
\Rz_{s=0}\Phi_{2j+1,q}(s) &=\frac{2}{2j+1}\alpha_q^{2j} + \sum^{j-1}_{t=1} k_{2j+1,t}\alpha_q^{2t} +
2\Rz_{z=0}\Phi_{2j+1,p-1}(s),\hspace{30pt}
\Ru_{s=0}\Phi_{2j+1,q}(s) &=0,\\
\end{align*}
where the $k_{j,t}$ are real numbers, and for $j=0$
\begin{align*}
\Rz_{s=0}\Phi_{1,q}(s) &=2\Rz_{s=0}\Phi_{1,p-1}(s)=2,&\Ru_{s=0}\Phi_{1,q}(s) &=0.\\
\end{align*}
\end{corol}
\begin{proof} By Lemma  \ref{l14},
\[
\phi_{2j+1,q}(\lambda) =  \sum^{2j+1}_{k=0}a_{2j+1,q,k}w^{2k+2j+1},
\] 
where $w = \frac{1}{\sqrt{1-\lambda}}$, and $\phi_{2j+1,q}(0)= 0$, therefore $\sum^{2j-1}_{k=0}a_{2j+1,q,k} = 0$. 
Using the formula in equation  \eqref{e} and the residues for the Gamma function in equation (\ref{residuo Gamma}) in the appendix,  we obtain 
\[
\Ru_{s=0} \Phi_{2j+1,q}(s) = \sum^{2j+1}_{k=0} a_{2j+1,q,k} = 0.
\]

Using the same formulas in the appendix, but the result of Lemma \ref{l15}, we prove the formula for the finite part.
The formula  for $j=0$ follows by explicit knowledge of the coefficients $a_{0,1}$.
\end{proof}

Next, we determine the terms $A_{0,0}(0)$ and $A'_{0,1}(0)$, defined in equation (\ref{fi2}). 

\begin{lem}\label{s4.l4} For all $0\leq q\leq p-2$, 
\begin{align*}
\A_{0,0,q}(s)&=2A_{0,0,q}(s)-A_{0,0,q,+}(s)-A_{0,0,q,-}(s) = - \sum_{n=1}^{\infty}
\log\left(1-\frac{\alpha_q^2}{\mu_{q,n}^2}\right)\frac{m_{q,n}}{\mu_{q,n}^{2s}},\\
\A_{0,1,q}(s)&=2A_{0,1,q}(s)-A_{0,1,q,+}(s)-A_{0,1,q,-}(s) 
= \zeta(2s,U_q).
\end{align*} 
\end{lem}
\begin{proof} For  $S_q$ equation \eqref{fi2} reads
\begin{align*}
A_{0,0,q}(s)&=\sum_{n=1}^\infty m_{{\rm cex},q,n}\left(a_{0, 0,n,q} -\sum_{j=1}^{p}b_{2j-1,0,0,q}\mu_{q,n}^{-2j+1}\right)\mu_{q,n}^{-2 s},\\
A_{0,1,q}(s)&=\sum_{n=1}^\infty m_{{\rm cex},q,n}\left(a_{0, 1,n,q}-\sum_{j=1}^{p}b_{2j-1,0,1,q}\mu_{q,n}^{-2j+1}\right)\mu_{q,n}^{-2s}.
\end{align*} 
for $S_{q,\pm}$:
\begin{align*}
A_{0,0,q,\pm}(s)&=\sum_{n=1}^\infty m_{{\rm cex},q,n}\left(a_{0, 0,n,q,\pm} -\sum_{j=1}^{p}b_{2j-1,0,0,q,\pm}\mu_{q,n}^{-2j+1}\right)\mu_{q,n}^{-2 s},\\
A_{0,1,q,\pm}(s)&=\sum_{n=1}^\infty m_{{\rm cex},q,n}\left(a_{0, 1,n,q,\pm}
-\sum_{j=1}^{p}b_{2j-1,0,1,q,\pm}\mu_{q,n}^{-2j+1}\right)\mu_{q,n}^{-2 s}.
\end{align*}

We need the expansions for large $\lambda$ of $l_{2j-1}(\lambda)$, $l^{\pm}_{2j-1}(\lambda)$,
for $j=1,2,\ldots,p$, $\log\Gamma(-\lambda, S_{q,n}/\mu_{q,n}^2)$ and $\log\Gamma(-\lambda, S_{q,\pm,n}/\mu_{q,n}^2)$. Using classical expansion for Bessel functions and their derivative (see \cite{HMS} or \cite{Har} for details), we obtain 
\begin{align*}
\log\Gamma(-\lambda, S_{q,n}/\mu_{q,n}^2) =& \frac{1}{2}\log
2\pi+\left(\mu_{q,n}+\frac{1}{2}\right)\log\mu_{q,n}-\mu_{q,n}\log2\\
&-\log\Gamma(\mu_{q,n}+1)+\frac{1}{2}\left(\mu_{q,n}+\frac{1}{2}\right)\log(-\lambda) +
O(\e^{-\mu_{q,n}\sqrt{-\lambda}}).
\end{align*}

For $S_{q,\pm}$, by the same expansions in the definition of the function $\hat I$, equation (\ref{s4.e1}), we obtain
\[
\hat I_{\nu,\pm\alpha_q}(z)\sim \frac{\sqrt{z}\e^z}{\sqrt{2\pi}}\left(1+\sum_{k=1}^\infty b_kz^{-k}\right)+O(\e^{-z}),
\] 
and hence
\begin{align*}
\log\Gamma(-\lambda, &S_{q,\pm,n}/\mu_{q,n}^2) = \mu_{q,n}\sqrt{-\lambda}+\frac{1}{2}\log
2\pi+\left(\mu_{q,n}-\frac{1}{2}\right)\log\mu_{q,n}-\mu_{q,n}\log2\\
&-\log\Gamma(\mu_{q,n})+\frac{1}{2}\left(\mu_{q,n}-\frac{1}{2}\right)\log(-\lambda) +
\log\left(1\pm\frac{\alpha_q}{\mu_{q,n}}\right) +O(\e^{-\mu_{q,n}\sqrt{-\lambda}}).
\end{align*}

This gives
\begin{align*}
&a_{0,0,n,q}=\frac{1}{2}\log 2\pi+\left(\mu_{q,n}+\frac{1}{2}\right)\log\mu_{q,n}-\mu_{q,n}\log 2-\log\Gamma(\mu_{q,n}+1),\\
&a_{0,1,n,q}=\frac{1}{2}\left(\mu_{q,n}+\frac{1}{2}\right), \\
&a_{0,0,n,q,\pm}=\frac{1}{2}\log 2\pi+\left(\mu_{q,n}-\frac{1}{2}\right)\log\mu_{q,n}-\log
2^{\mu_{q,n}}\Gamma(\mu_{q,n})
+\log\left(1\pm\frac{\alpha_q}{\mu_{q,n}}\right),\\
&a_{0,1,n,q,\pm}=\frac{1}{2}\left(\mu_{q,n}-\frac{1}{2}\right),
\end{align*}
while the  $b_{2j-1,0,0,q}$, $b_{2j-1,0,0,q,\pm}$ all vanish since the functions $l_{2j-1}(\lambda)$,  $l^{\pm}_{2j-1}(\lambda)$ do not have constant terms. Therefore,
\begin{align*}
2a_{0,0,n,q}-a_{0,0,n,q,+}-a_{0,0,n,q,-}&= -\log\left(1-\frac{\alpha_q^{2}}{\mu^{2}_{q,n}}\right),\\
2a_{0,1,n,q}-a_{0,1,n,q,+}-a_{0,1,n,q,-}&=1,
\end{align*} 
and the thesis follows.

\end{proof}

Applying Theorem \ref{t4} and its corollary, we obtain the values of $t_q(0)$ and $t_q'(0)$. 

\begin{prop}\label{pro3}  For $0\leq q\leq p-2$,
\begin{align*}
t_q(0) &= t_{q,{\rm reg}}(0) + t_{q,{\rm sing}}(0)\\
t'_q(0) &= t'_{q,{\rm reg}}(0) + t'_{q,{\rm sing}}(0),
\end{align*}
where
\begin{align*}
t_{q,{\rm reg}}(0) &=-\zeta(0, U_{q})=-\zeta_{{\rm cex}}\left(0,\tilde\Delta^{(q)}+\alpha_q^2\right),\\
t_{q,{\rm sing}}(0) &=0,\\
t'_{q,{\rm reg}}(0) &=-\A_{q,0,0}(0) - \A'_{q,0,1}(0),\\
t'_{q,{\rm sing}}(0) &=\frac{1}{2}\sum_{j=0}^{p-1}\Rz_{s=0}\Phi_{2j+1,q}(s)\Ru_{s=2j+1}\zeta(s,U_q)\\
&=\frac{1}{2}\sum_{j=0}^{p-1}\Rz_{s=0}\Phi_{2j+1,q}(s)\Ru_{s=2j+1}\zeta_{{\rm cex}}\left(\frac{s}{2},\tilde\Delta^{(q)}+\alpha_q^2\right).
\end{align*}
\end{prop}
\begin{proof} By definition in equations (\ref{z1}) and (\ref{tq}),
\begin{align*}
t_{q}(0)=&2Z_{q}(0)- Z_{q,+}(0)-Z_{q,-}(0),\\
t_{q}'(0)=&2Z_{q}'(0)- Z_{q,+}'(0)-Z_{q,-}'(0).
\end{align*}
where $Z_{q}(s)=\zeta(s,S_{q})$, and $ Z_{q,\pm}(s)=\zeta(s, S_{q,\pm})$. By Proposition \ref{s4.p1} and Lemma \ref{s4.l2}, we can apply Theorem \ref{t4} and its Corollary to the linear combination above of these double zeta functions. The regular part of $2Z_{q}(0)- Z_{q,+}(0)-Z_{q,-}(0)$ is then given in Lemma \ref{s4.l4}, while the singular part vanishes, since, by Corollary \ref{c44}, the residues of the functions $\Phi_{k,q}(s)$ at $s=0$ vanish. 
The regular part of $2Z_{q}'(0)- Z_{q,+}'(0)-Z_{q,-}'(0)$ again follows by Lemma \ref{s3.l4}. For the singular part,   since by Proposition \ref{s4.p1}, $\kappa=2$, $\ell=2p$, and $\sigma_h=h-1$, with $0\leq h\leq 2p$, by Remark \ref{rrr4} we need only the odd values of $h-1=2j+1$, $0\leq j\leq p-1$, and this gives the formula stated for $t_{p-1,{\rm sing}}'(0)$.
\end{proof}

\section{The analytic torsion, and the proof of Theorem \ref{t01}}
\label{s5}

In this section we collect all the results obtained in the previous one in order to produce our formulas for the analytic torsion, thus proving Theorem \ref{t01}, that follows from Propositions \ref{ppaa} and \ref{ppoo} below.  
By equation (\ref{ttt}), the  torsion is 
\begin{align*}
\log T(C_l W)=t'(0)=& \frac{\log l^{2}}{2} \left(\sum_{q=0}^{p-1} (-1)^{q+1} r_qz_q(0)+\sum^{p-1}_{q=0}(-1)^{q} t_q(0)\right) \\
&+ \frac{1}{2} \left(\sum_{q=0}^{p-1} (-1)^{q+1} r_q z_q'(0)+\sum^{p-1}_{q=0}(-1)^{q} t_q'(0)\right).
\end{align*}

However, it is convenient to split the torsion in {\it regular} and {\it singular} part, accordingly to remark \ref{rqr} and the results in Propositions \ref{pro2} and \ref{pro3}. First, observe that the functions $z_q(s)$ where studied in Section \ref{ss2}, where it is showed that  there is no singular contribution to $z_q(0)$ and $z'_q(0)$. So $z_q(0)=z_{q,{\rm reg}}(0)$, and $z_q'(0)=z_{q,{\rm reg}}'(0)$. Therefore, we set
\[
\log T(C_l W)=\log T_{\rm reg}(C_l W)+\log T_{\rm sing}(C_l W),
\]
with
\begin{align}
\label{ttrr}\log T_{\rm reg}(C_l W)=t_{\rm reg}'(0)=& \frac{\log l^{2}}{2} \left(\sum_{q=0}^{p-1} (-1)^{q+1} r_qz_q(0)+\sum^{p-1}_{q=0}(-1)^{q} t_{q,{\rm reg}}(0)\right) \\
\nonumber&+ \frac{1}{2} \left(\sum_{q=0}^{p-1} (-1)^{q+1} r_q z_q'(0)+\sum^{p-1}_{q=0}(-1)^{q} t_{q,{\rm reg}}'(0)\right),\\
\label{ttss}\log T_{\rm sing}(C_l W)=t_{\rm sing}'(0)=& \frac{\log l^{2}}{2} \sum^{p-1}_{q=0}(-1)^{q} t_{q,{\rm sing}}(0) + \frac{1}{2} \sum^{p-1}_{q=0}(-1)^{q} t_{q,{\rm sing}}'(0).
\end{align}


\begin{lem} \label{lp11} For all $0\leq q\leq p-1$,
\begin{align*}
z_q(0)&=-\frac{1}{2},\\
z_q'(0)&=\log 2+\log (p-q).
\end{align*}
\end{lem}
\begin{proof} This follows by equation (\ref{p00}).
\end{proof}

\begin{lem}\label{lp2} 
\begin{align*}
t_{q,{\rm reg}}(0)&=-\zeta_{\rm cex}(0, \tilde \Delta^{(q)}),& 0\leq q\leq p-2,\\
t_{q,{\rm reg}}'(0)&=-\zeta_{\rm cex}'(0, \tilde \Delta^{(q)}),& 0\leq q\leq p-2,\\
t_{p-1,{\rm reg}}(0)&=-\frac{1}{2}\zeta_{\rm cex}(0, \tilde \Delta^{(p-1)}),\\
t_{p-1,{\rm reg}}'(0)&=-\frac{1}{2}\zeta_{\rm cex}'(0, \tilde \Delta^{(p-1)}).
\end{align*}
\end{lem}
\begin{proof} The first and the third formulas follows by  Propositions \ref{pro2} and \ref{pro3}, and the fact that for the zeta function associated to any sequence $S$, and any number $b$, $\zeta(0,S+b)=\zeta(0,S)$. For the derivatives, when $0\leq q\leq p-2$, by Proposition \ref{pro3},
\[
t_{q,{\rm reg}}'(0)=-\A_{0,0,q}(0)-\A_{0,1,q}'(0).
\]

By Lemma \ref{s4.l4}
\begin{align*}
\A_{0,0,q}(s)&= - \sum_{n=1}^{\infty}
\log\left(1-\frac{\alpha_q^2}{\mu_{q,n}^2}\right)\frac{m_{{\rm cex},q,n}}{\mu_{q,n}^{2s}},\\
\A_{0,1,q}(s)&= \zeta(2s,U_q)= \sum_{n=1}^{\infty} \frac{m_{{\rm cex},q,n}}{\mu_{q,n}^{2s}}.\\ 
\end{align*}

Recalling that $\mu_{q,n}=\sqrt{\lambda_{q,n}+\alpha_q}$, and expanding the binomial, we obtain
\begin{align*}
-\A_{0,0,q}(s)-\A_{0,1,q}'(s)&=   \sum^{\infty}_{n=1} \log
\left(1-\frac{\alpha_q^2}{\mu_{q,n}^2}\right) \frac{m_{{\rm cex},q,n}}{\mu_{q,n}^{2s}}
- \sum^{\infty}_{n=1} \frac{m_{{\rm cex},q,n}}{\mu_{q,n}^{2s}}\log \mu_{q,n}^2\\
&=  \sum_{n=1}^{\infty} \log\lambda_{q,n} \frac{m_{{\rm cex},q,n}}{\mu_{q,n}^{2s}}\\
&=  \sum^{\infty}_{n=1} \log\lambda_{q,n} \sum^{\infty}_{j=0}
\binom{-s}{j}\frac{m_{{\rm cex},q,n}}{\lambda_{q,n}^{s+j}}
\alpha_q^{2j}\\
&= -\sum^{\infty}_{j=0} \binom{-s}{j} \zeta_{\rm ccl}'(s+j,\tilde\Delta^{(q)}) \alpha_q^{2j},
\end{align*}
that gives the second formula. Eventually, the result for $t'_{p-1,{\rm reg}}(0)$ follows by Proposition \ref{pro2} and the fact that $\alpha_{p-1}=0$ since the dimension is $m=2p-1$.
\end{proof}

\begin{prop}\label{ppaa}
\begin{align*}
\log T_{\rm reg}(C_l W)=&\frac{1}{2} \sum_{q=0}^{p-1} (-1)^{q} r_q\log \frac{l}{2}
-\frac{1}{2}\sum_{q=0}^{p-1} (-1)^{q} r_q\log (p-q)+\frac{1}{2} \log T(W,g)\\
&-\left(\sum_{q=0}^{p-2} (-1)^q \zeta_{\rm ccl}(0, \tilde\Delta^{(q)})+\frac{1}{2}(-1)^{p-1}\zeta_{\rm ccl}(0, \tilde\Delta^{(p-1)})\right)\log l\\
=&\frac{1}{2} \sum_{q=0}^{p-1} (-1)^{q} r_q\log \frac{l}{2}
-\frac{1}{2}\sum_{q=0}^{p-1} (-1)^{q} r_q\log (p-q)+\frac{1}{2} \log T(W,l^2g),
\end{align*}
where $r_q={\rm rk}\H_q(\b C_lW;\Q)$
\end{prop}

\begin{proof} Substitution in the formula in equation (\ref{ttrr}) of the values given in Lemmas \ref{lp11} and \ref{lp2} gives
\begin{align*}
\log T_{\rm reg}(C_l W)=&\frac{1}{2} \sum_{q=0}^{p-1} (-1)^{q} r_q\log \frac{l}{2}
-\frac{1}{2}\sum_{q=0}^{p-1} (-1)^{q} r_q\log (p-q)\\
&-\left(\sum_{q=0}^{p-2} (-1)^q \zeta_{\rm ccl}(0, \tilde\Delta^{(q)})+\frac{1}{2}(-1)^{p-1}\zeta_{\rm ccl}(0, \tilde\Delta^{(p-1)})\right)\log l\\
&+\frac{1}{4}\left(2\sum_{q=0}^{p-2} (-1)^{q+1} \zeta'_{\rm ccl}(0, \tilde\Delta^{(q)})+(-1)^p\zeta_{\rm ccl}'(0, \tilde\Delta^{(p-1)})\right).\\
\end{align*}

By the second formula in equation (\ref{odd})
\[
\frac{1}{4}\left(2\sum_{q=0}^{p-2} (-1)^{q+1} \zeta'_{\rm ccl}(0, \tilde\Delta^{(q)})+(-1)^p\zeta_{\rm ccl}'(0, \tilde\Delta^{(p-1)})\right)=\frac{1}{2}\log T(W,g),
\]
and this gives the first formula stated. For the second formula, note that the boundary of the cone $\b C_l W$ is the manifold $W$ with metric $l^2 g$. The restriction of the Laplace operator on the boundary is then $\Delta_{\b C_l W}=\frac{\tilde\Delta}{l^2}$. Since for the zeta function associated to any sequence $S$, and any number $a$, 
\[
\zeta'(0,aS)=-\zeta(0,S)\log a+\zeta'(0,S),
\]
a simple calculation shows that
\begin{align*}
&-\left(\sum_{q=0}^{p-2} (-1)^q \zeta_{\rm ccl}(0, \tilde\Delta^{(q)})+\frac{1}{2}(-1)^{p-1}\zeta_{\rm ccl}(0, \tilde\Delta^{(p-1)})\right)\log l^2\\
&+\frac{1}{2}\left(2\sum_{q=0}^{p-2} (-1)^{q+1} \zeta'_{\rm ccl}(0, \tilde\Delta^{(q)})+(-1)^p\zeta_{\rm ccl}'(0, \tilde\Delta^{(p-1)})\right)\\
=&t(0,W)\log l^2+t'(0,W)=\log T(\b C_l W).
\end{align*}

\end{proof}

\begin{prop}\label{ppoo}
\begin{align*}
\log T_{\rm sing}(C_l W)=& \frac{1}{2}\sum_{q=0}^{p-1} (-1)^q \sum_{j=0}^{p-1}\Rz_{s=0}\Phi_{2j+1}(s)\Ru_{s=j+\frac{1}{2}}\zeta_{\rm cex}\left(s,\tilde\Delta^{(q)}+\alpha_q^2\right)\\
=&\frac{1}{2}\sum_{q=0}^{p-1} (-1)^q \sum_{j=0}^{p-1}\Rz_{s=0}\Phi_{2j+1}(s)\sum^{q}_{l=0}(-1)^{l}\Ru_{s=j+\frac{1}{2}}\zeta\left(s,\tilde \Delta^{(l)}+\alpha_q^2\right)\\
=& \frac{1}{2}\sum_{q=0}^{p-1} (-1)^q\sum_{j=0}^{p-1}\sum^{j}_{k=0} \Rz_{s=0}\Phi_{2k+1,q}(s) \binom{-\frac{1}{2}-k}{j-k}
\Ru_{s=j+\frac{1}{2}}\zeta_{cex}\left(s,\tilde \Delta^{(q)}\right)\alpha_q^{2(j-k)}\\
=&\frac{1}{2}\sum_{q=0}^{p-1}  \sum_{j=0}^{p-1}\sum^{j}_{k=0} \Rz_{s=0}\Phi_{2k+1,q}(s)
\binom{-\frac{1}{2}-k}{j-k} \sum^{q}_{l=0}(-1)^{l}\Ru_{s=j+\frac{1}{2}}\zeta\left(s,\tilde \Delta^{(l)}\right)\alpha_q^{2(j-k)}.
\end{align*}

\end{prop}

\begin{proof} The first formula  follows by substitution in equation (\ref{ttss}) of the values given in Propositions \ref{pro2} and \ref{pro3}, and observing that, for the zeta function associated to any sequence $S$
\[
a\Ru_{s=s_0}\zeta(as,S)=\Ru_{s=as_0}\zeta(s,S).
\]

The second by duality, see Section \ref{forms},
\begin{align*}
\zeta_{ccl}\left(s,\tilde \Delta^{(q)}\right) &= \zeta\left(s,\tilde \Delta^{(q)}\right) - \zeta_{cl}\left(s,\tilde
\Delta^{(q)}\right)= \zeta\left(s,\tilde \Delta^{(q)}\right) - \zeta_{ccl}\left(s,\tilde\Delta^{(q-1)}\right)= \sum^{q}_{k=0} (-1)^{q+k}\zeta\left(s,\tilde \Delta^{(k)}\right).
\end{align*}

The third formula follows by Lemmas \ref{lp1} and \ref{lel}, and some combinatorics,   and the last by the previous ones.
\end{proof}


\section{The proof of Theorem \ref{t02}}
\label{s6}

On order to prove Theorem \ref{t02} we calculate the regular and the singular parts of the torsion in the case $W=S^m_{\sin\alpha}$, according to  Propositions \ref{ppaa} and \ref{ppoo}. Recall we are considering the absolute BC case. The result for the regular part follows easily, the one for the singular part requires more works, that will be developed in the following subsections. Here we recall the underlying geometric setting. Let $S^m_b$ be the sphere of radius  $b>0$ in $\R^{m+1}$, $S^{m}_b=\{x\in\R^{m+1}~|~|x|=b\}$ (we simply write $S^m$
for $S^m_1$). Let  $C_l S^m_{\sin\alpha}$ denotes the cone of angle $\alpha$ over $S^m_{\sin\alpha}$ in $\R^{m+2}$. We embed $C_l S^m_{\sin\alpha}$ in $\R^{m+2}$ as the subset of the segments joining the origin to the sphere $S^m_{l\sin\alpha}\times \{(0,\dots,0,l\cos\alpha)\}$. We parametrize the cone by
\begin{equation*}\label{}C_{l}S_{\sin\alpha}^{m}=\left\{
\begin{array}{rcl}
x_1&=&r \sin{\alpha} \sin{\theta_m}\sin{\theta_{m-1}}\cdots\sin{\theta_3}\sin{\theta_2}\cos{\theta_1} \\[8pt]
x_2&=&r \sin{\alpha} \sin{\theta_m}\sin{\theta_{m-1}}\cdots\sin{\theta_3}\sin{\theta_2}\sin{\theta_1} \\[8pt]
x_3&=&r \sin{\alpha} \sin{\theta_m}\sin{\theta_{m-1}}\cdots\sin{\theta_3}\cos{\theta_2} \\[8pt]
&\vdots& \\
x_{m+1}&=&r \sin{\alpha} \cos{\theta_m} \\[8pt]
x_{m+2}&=&r \cos{\alpha}
\end{array}
\right.
\end{equation*}
with $r \in [0,l]$, $\theta_1 \in [0,2\pi]$, $\theta_2,\ldots,\theta_m \in [0,\pi]$, and where  $\alpha$ is a fixed positive real number and  $0<a=\frac{1}{\nu}= \sin{\alpha}\leq 1$. The induced metric is  ($r>0$)
\begin{align*}
g_E &=dr \otimes dr + r^2  g_{S^{m}_{a^2}}\\
&= dr\otimes dr + r^2  a^2\left(\sum^{m-1}_{i=1} \left(\prod^{m}_{j=i+1} \sin^2{\theta_j}\right) d\theta_i \otimes
d\theta_i + d\theta_m \otimes d\theta_m\right),
\end{align*}
and $\sqrt{|\det
g_E|}=(r\sin\alpha)^{m}(\sin\theta_m)^{m-1}(\sin\theta_{m-1})^{m-2}\cdots(\sin\theta_3)^{2}(\sin\theta_2)$.

\subsection{The regular part of the torsion}

\begin{prop} 
\[
\log T_{\rm reg}(C_l S^{2p-1}_{\sin\alpha})=\frac{1}{2} \log{\rm Vol} (C_l S^{2p-1}_{\sin\alpha}).
\]
\end{prop}
\begin{proof} By Proposition \ref{ppaa}, when $W= S^{2p-1}_{\sin\alpha}$ with the standard Euclidean metric $g_E$, 
\[
\log T_{\rm reg}(C_l S^{2p-1}_{\sin\alpha})=\frac{1}{2}\log\frac{l}{2}-\frac{1}{2}\log p+\frac{1}{2}\log  T ( S^{2p-1}_{\sin\alpha},l^2 g_E).
\]

By \cite{MS}, $  T ( S^{2p-1}_{\sin\alpha}, l^2g_E)={\rm Vol}( S^{2p-1}_{l\sin\alpha},g_E)$, and this proves the proposition since, if $W$ has metric $g$ and dimension $m$, then
\[
{\rm Vol} (C_l W)=\int_{C_l W} \sqrt{\det (x^2g)}dx\wedge dvol_g=\int_0^l x^m\int_W dvol_g=\frac{l^{m+1}}{m+1}{\rm Vol} (W,g),
\]
and
\[
{\rm Vol} (S_b^m,g_E)=\frac{2\pi^\frac{m+1}{2}b^m}{\Gamma\left(\frac{m+1}{2}\right)}.
\]

\end{proof}

\subsection{The conjecture for the singular part}

Assuming that the formula for the anomaly boundary term $A_{\rm BM,abs}(\b W)$ of Br\"uning and Ma \cite{BM} is valid in the case of $C_l S^{2p-1}_{\sin\alpha}$, we computed  in \cite{HS} (note the slight different notation), by applying the definition given equation (\ref{anom}) of in Section \ref{cm}, that
\begin{align*}
 A_{\rm BM,abs}(\b C_l S^{2p-1}_{\sin\alpha})=\sum_{j=0}^{p -1} \frac{2^{p-j}}{j!(2(p-j)-1)!!} \sum_{h=0}^{j} \binom{j}{h}
\frac{(-1)^{h}\nu^{-2(p-j+h)+1}}{(2(p-j+h)-1)} \frac{ (2p-1)!}{4^{p} (p-1)!},
\end{align*}
where $\frac{1}{\nu}=\sin\alpha$. Our purpose now is to prove that (this was proved in \cite{HS} for $m<4$)
\beq\label{conj}
\begin{aligned}
\log T_{\rm sing}( C_l &S^{2p-1}_{\sin\alpha})= A_{\rm BM,abs}(\b C_l S^{2p-1}_{\sin\alpha})
\end{aligned}
\eeq
where  $\log T_{\rm sing}( C_l S^{2p-1}_{\sin\alpha})$ is given in Proposition \ref{ppoo}. For it is convenient to rewrite the second term as follows:
\begin{align*}
 A_{\rm BM,abs}(\b C_l S^{2p-1}_{\sin\alpha})
&=\sum_{j=0}^{p -1} \frac{2^{p-j}}{j!(2(p-j)-1)!!} \sum_{h=0}^{j} \binom{j}{h}
\frac{(-1)^{h}\nu^{-2(p-j+h)+1}}{(2(p-j+h)-1)} \frac{ (2p-1)!}{4^{p} (p-1)!}\\
&=\sum_{j=0}^{p-1} \frac{2^{j+1}}{(p-1-j)!(2j+1)!!} \sum_{h=0}^{p-1-j} \binom{p-1-j}{h}
\frac{(-1)^{h}\nu^{-2(j+1+h)+1}}{2(j+1+h)-1} \frac{ (2p-1)!}{4^{p} (p-1)!}\\
&=\frac{(2p-1)!}{4^p (p-1)!}\sum_{k=0}^{p-1} \frac{1}{(2k+1)\nu^{2k+1}} \sum^{k}_{j=0}
\frac{(-1)^{k-j}2^{j+1}}{(p-1-j)!(2j+1)!!}\binom{p-1-j}{k-j}\\
&=\frac{(2p-1)!}{4^p (p-1)!}\sum_{k=0}^{p-1} \frac{1}{(p-1-k)!(2k+1)} \sum^{k}_{j=0}
\frac{(-1)^{k-j}2^{j+1}}{(k-j)!(2j+1)!!}\frac{1}{\nu^{2k+1}},
\end{align*}

\subsection{The eigenvalues of the Laplacian over $C_lS_{\sin\alpha}^{m}$}

Let $\Delta$ be the self adjoint extension of the formal Laplace operator  on $C_l S_{\sin\alpha}^{m}$ as defined in section \ref{Lap1.3}. Then, the positive part of the spectrum of $\Delta$ (with absolute BC) is given in Lemma \ref{l3}, once we know the eigenvalues of the restriction of the Laplacian on the section and their coexact multiplicity, according to Lemma \ref{l2}. These information are available by work of Ikeda and Taniguchi \cite{IT}. The eigenvalues of the Laplacian on $q$-forms on $S^{2p-1}_{\sin\alpha}$ are
\[
\left\{
  \begin{array}{ll}
    \lambda_{0,n} = \nu^2 n(n+2p-2), & \\
    \lambda_{q,n} = \nu^2(n+q)(n+2p-q-2), & 1\leq q < p-2,     \\
    \lambda_{p-2,n} = \nu^2((n-1+p)^2-1), &  \\
    \lambda_{p-1,n} = \nu^2(n-1+p)^2, &  \\
  \end{array}
\right.
\]
with coexact multiplicty
\begin{align*}
    &m_{{\rm cex},0,n} = \frac{2}{(2p-2)!} \prod_{j=2}^{p}(n-1+j)(2p+n-1-j),  \\
    &m_{{\rm cex},q,n} = \frac{2}{q!(2p-q-2)!} \prod^{p}_{\substack{j=1,\\j\neq q+1}} (n-1+j)(2p+n-1-j), ~ 1\leq q < p-2,     \\
    &m_{{\rm cex},p-2,n} = \frac{2}{(p-2)! p!} \prod^{p}_{\substack{j=1 \\ j\neq p-1}} (n-1+j)(2p+n-1-j),   \\
    &m_{{\rm cex},p-1,n} = \frac{2}{[(p-1)!]^2} \prod^{p-1}_{j=1} (n-1+j)(2p+n-1-j),   \\
\end{align*}
thus the indices  $\mu_{q,n}$ are
\[
\left\{
  \begin{array}{ll}
    \mu_{0,n} = \sqrt{\nu^2(n(n+2p-2)) + (p-1)^2}, & \\
    \mu_{q,n} = \sqrt{\nu^2(n+q)(n+2p-q-2)+\alpha_q^2}, & 1\leq q < p-2,     \\
    \mu_{p-2,n} = \sqrt{\nu^2((n-1+p)^2-1)+ 1},&  \\
    \mu_{p-1,n} = \nu(n-1+p). &  \\
    \end{array}
\right.
\]

\subsection{Some combinatorics}
\label{sconj1}

We introduce some notation. Let 
\[
U_{q,S^{2p-1}} = \{m_{{\rm cex},q,n}: \lambda_{q,n,S^{2p-1}}\},
\]
denotes the sequence of the eigenvalues of the coexact $q$-forms of the Laplace operator over the sphere of dimension $2p-1$ and radius $1$. Let $a_1,\ldots,a_m$ be a finite sequence of real numbers. Then, 
\[
\prod^{m}_{j=1} (x + a_j) = \sum^{m}_{j=0} e_{m-j}(a_1,\ldots,a_m) x^j
\]
where the  $e_1,\ldots,e_m$ are elementary symmetric polynomials in $a_1,\ldots,a_m$. Let define the numbers:
\[
d^{q}_{j}:=(j-q-1)(2p-q-j-1),
\]
for $q=0,\ldots,p-1,j\neq q+1$, and
\[
d^{q}:=(d^{q}_{1},d^{q}_{2},\ldots,\hat d^{q}_{q+1},\ldots,d^{q}_{p}),
\]
where, as usual,  the  hat means the underling term is delated. 

\begin{lem}\label{lema Up-1S^2p-1} The sequence  $U_{p-1}$ is a totally regular sequence of spectral type with infinite order, exponent and genus: $e(U_{p-1}) = g(U_{p-1})= 2p-1$, and
\[
\zeta(s,U_{p-1}) =  \frac{2\nu^{-s}}{(p-1)!^2} \sum^{p-1}_{j=0} e_{p-1-j}(d^{p-1}) \zeta_R(s-2j).
\]
\end{lem}

\begin{proof} The first part of the statement follows from Lemma \ref{lp1}. In order to prove the formula, note that $\zeta(s,U_{p-1}) = \nu^{-s}\zeta\left(\frac{s}{2}, U_{p-1,S^{2p-1}}\right)$, where
\[
\zeta\left(\frac{s}{2}, U_{p-1,S^{2p-1}}\right) = \sum^{\infty}_{n=1}
\frac{m_{{\rm cex},p-1,n}}{\lambda_{p-1,n,S^{2p-1}}^{\frac{s}{2}}} = \sum^{\infty}_{n=1} \frac{m_{{\rm cex},p-1,n}}{(n+p-1)^{s}}.
\] 

Shifting  $n$ to $n-p+1$, and observing that the numbers $1,\dots,p-1$ are roots of the polynomial $\sum^{p-1}_{j=0}
e_{p-1-j}(d^{p-1})n^{2j}$, we obtain 
\begin{align*}
\zeta(s,U_{p-1}) &=\nu^{-s}\sum^{\infty}_{n=p} \frac{m_{{\rm cex},p-1,n-p+1}}{n^{s}}=\frac{2\nu^{-s}}{(p-1)!^2}\sum^{\infty}_{n=p} \frac{\prod^{p-1}_{j=1} n^{2}-(p-j)^{2}}{n^{s}}\\
&=\frac{2\nu^{-s}}{(p-1)!^2}\sum^{p-1}_{j=0}    e_{p-1-j}(d^{p-1})\zeta_{R}(s-2j).
\end{align*}
\end{proof}

Note that, using the formula of the lemma,  $\zeta(s,U_{p-1})$ has an expansion near $s=
2k+1$, with $k=0,1,\dots,p-1$, of the following type:
\begin{equation*}
\begin{aligned}
\zeta(s,U_{p-1})&=\frac{2}{\nu^{2k+1}(p-1)!^2} e_{p-1-k}(d^{p-1})\frac{1}{s-2k-1}+ L_{p-1,2k+1}(s),
\end{aligned}
\end{equation*} 
where the  $L_{p-1,2k+1}(s)$ are regular function for $k=0,1,\dots,p-1$.

\begin{corol}\label{residuo zeta Up-1S^2p-1} The function $\zeta(s,U_{p-1})$ has simple poles at  $s=2k+1$, for $k=0,1,\dots,p-1$, with residues
\[
\Ru_{s=2k+1} \zeta(s,U_{p-1}) = \frac{2}{\nu^{2k+1}(p-1)!^2} e_{p-1-k}(d^{p-1}).
\]
\end{corol}

\begin{lem} \label{l26} The sequence $U_q$ is a totally regular sequence of spectral type with infinite order, exponent and genus:
$\ec(U_{q})=\ge(U_{q})=2p-1$, and (where $i=\sqrt{-1}$)
\[
\zeta(s,U_{q})=\frac{2\nu^{-s}}{q!(2p-q-2)!}\sum_{t=0}^{\infty}\binom{-\frac{s}{2}}{t} \sum^{p-1}_{j=0}
e_{p-1-j}(d^{q}) z\left(\frac{s+2t-2j}{2},i \alpha_q \right)\frac{\alpha_q^{2t}}{\nu^{2t}}.
\] 

The function $\zeta(s,U_q)$ has simple poles at $s=2(p-k)-1$, with $k=0,1,2,\ldots$.
\end{lem}

\begin{proof} The first statement follows by Lemma \ref{lel}. For the second one, consider the sequence $H_{q,h}=\left\{m_{{\rm cex},q,n}:\sqrt{\lambda_{q,n,S^{2p-1}}+h}\right\}_{n=1}^{\infty}$. Then $\zeta(s,U_q)=\nu^{-s}\zeta(s,H_{q,\frac{\alpha_q^2}{\nu^2}})$, and
\begin{align*}
\zeta(s,H_{q,h}) &= \sum^{\infty}_{n=1} \frac{m_{{\rm cex},q,n}}{(\lambda_{q,n,S^{2p-1}}+h)^{\frac{s}{2}}}= \sum^{\infty}_{n=1} \sum_{t=0}^{\infty} \binom{-\frac{s}{2}}{t}
\frac{m_{{\rm cex},q,n}}{\lambda_{q,n,S^{2p-1}}^{\frac{s}{2}+t}} h^{t}=\sum_{t=0}^{\infty} \binom{-\frac{s}{2}}{t}\zeta(s+2t,H_{q,0})h^{t}.
\end{align*} 

 Next observe that the zeta function associated to the sequence $H_{q,0}$ is
\begin{align*}
\zeta(2s,H_{q,0}) &= \zeta(s, U_{q,S^{2p-1}}) =\sum^{\infty}_{n=1} \frac{m_{{\rm cex},q,n}}{\lambda_{q,n,S^{2p-1}}^{s}}=\sum^{\infty}_{n=p} \frac{m_{q,n-p+1}}{\lambda_{q,n-p+1,S^{2p-1}}^{s}}\\
&=\frac{2}{q!(2p-q-2)!}\sum^{\infty}_{n=p} \frac{\prod^{p}_{\substack{j=1,\\j\neq q+1}}(n^2 -
(p-j)^2)}{(n^2-\alpha_q^2)}.
\end{align*}

Recall that $\alpha_q^2 = d^q_p$, and note that
\begin{align*}
\sum^{p-1}_{j=0} e_{p-j-1}(d^{q})(n^2-\alpha_q^2)^j &= \sum^{p-1}_{j=0}
e_{p-j-1}(d^q)(n^2-d^q_p)^j= \prod^{p}_{\substack{j=1,\\j\neq q+1}} (n^2 - d^{q}_p + d^q_j)= \prod^{p}_{\substack{j=1,\\j\neq q+1}} (n^2 - (p-j)^2),
\end{align*}
and that the numbers  $n=1,2,\ldots,-\alpha_q$ are roots of this polynomial. Therefore, we can write
\begin{align*}
\zeta(2s,H_{q,0})
&=\frac{2}{q!(2p-q-2)!} \sum^{p-1}_{j=0} e_{p-1-j}(d^{q})
\left(z(s-j,i\alpha_q )-\sum^{p-q-2}_{n=1} (n^2 - \alpha_q^2)^{-s+j}\right)\\
&=\frac{2}{q!(2p-q-2)!} \sum^{p-1}_{j=0} e_{p-1-j}(d^{q}) z(s-j,i\alpha_q ),
\end{align*} 
and 
\[
z(s-j,i\alpha_q ) = \sum_{n=1}^{\infty} \frac{1}{(n^2-\alpha_q^2)^{s-j}}.
\]

Expanding the binomial, as in  \cite{Spr4}, Section 2, $z(s,a) = \sum^{\infty}_{k=0} \binom{-s}{k} a^{2k} \zeta_R(2s+2k)$, 
and hence  $z(s,a)$ has simple poles at  $s=\frac{1}{2}-k$,  $k=0,1,2,\dots$. Since
\[
\zeta(2s,H_{q,0}) = \frac{2}{q!(2p-q-2)!} \sum^{p-1}_{j=0} e_{p-1-j}(d^{q}) z(s-j,i\alpha_q ),
\] 
$\zeta(2s,H_{q,0})$ has simple poles at $s = \frac{1}{2} + p-1 - k$, $k=0,1,2,\dots$, $\zeta(s,H_{q,0})$ has simple poles at $s=2(p-k)-1$, $k=0,1,2,\ldots$, and this completes the proof.
\end{proof}

\begin{corol}\label{r7} The function $\zeta(s,U_q)$ has simple poles at $s=2k+1$, for $k=0,1,\ldots,p-1$, with residues
\[
\Ru_{s=2k+1} \zeta(s,U_q) = \frac{2\nu^{-2k-1}}{q!(2p-q-2)!}\sum^{p-1-k}_{t=0}
\frac{1}{\nu^{2t}}\binom{-\frac{2k+1}{2}}{t} \sum^{p-1}_{j=k+t} e_{p-1-j}(d^{q})
\binom{-\frac{1}{2}}{j-k-t}\alpha_q^{2(j-k)}
\]
\end{corol}

\begin{proof} Since the value of the residue of the Riemann zeta function at $s=1$ is 1, 
\begin{align*}
&\Ru_{s=\frac{1}{2} - k } z(s-j,a) = \Ru_{s=\frac{1}{2}- j - k } z(s,a) = \binom{-\frac{1}{2}+j+k}{j+k}
\frac{a^{2j+2k}}{2},
\end{align*} 
for $k=0,1,2,\ldots$. Considering $\zeta(2s,H_{q,0})$, we have, for $k=0,1,\dots,p-1$, 
\begin{align*}
\Ru_{s=\frac{1}{2}+k}\zeta(2s,H_{q,0}) &=\frac{2}{q!(2p-q-2)!} \sum^{p-1}_{j=k} e_{p-1-j}(d^{q})(-1)^{j-k}
\binom{-\frac{1}{2}+j-k}{j-k}\frac{\alpha_q^{2j-2k}}{2},
\end{align*}
and the thesis follows.

\end{proof}

The result contained in the next lemma follows by geometric reasons. However, we present here a purely combinatoric proof.
\begin{lem} For all $0\leq q\leq p-1$,  $\zeta(0, U_{q,S^{2p-1}}) = (-1)^{q+1}$.
\end{lem}

\begin{proof}
Consider the function
\[
  \zeta_{t,c}(s) = \sum_{n=1}^{\infty} \frac{1}{(n(n+2t))^{s-c}}= \sum^{\infty}_{n=t+1} \frac{1}{(n^2 - t^2)^{s-c}}.
\] 

Since 
\begin{align*}
z(s-c,it) &= \sum^{\infty}_{n=1}  \frac{1}{(n^2 - t^2)^{s-c}}=\sum^{\infty}_{j=0} \binom{-s+c}{j} (-1)^{j} t^{2j} \zeta_{R}(2s+2j-2c), 
\end{align*} 
we have when $s=0$, that $z(-c,it) = (-1)^{c} t^{2c} \zeta_{R}(0) = (-1)^{c+1} \frac{t^{2c}}{2}$, and hence
\[
\zeta_{t,c}(s) = z(s-c,it) -\sum_{n=1}^{t} \frac{1}{(n^2 - t^2)^{s-c}},
\] 
and for $c=0$ and $s=0$ $\zeta_{t,0}(0) = - \frac{1}{2} - t$. Next, consider $c>0$, then:
\[
\zeta_{t,c}(0) =  (-1)^{c+1} \frac{t^{2c}}{2} - \sum^{t-1}_{n=1} (n^2 - t^2)^c.
\]

For  $q=0,\ldots,p-1$, we have
\begin{align*}
\zeta(s, U_{q,S^{2p-1}}) &= \sum^{\infty}_{n=1} \frac{m_{{\rm cex},q,n}}{\lambda_{q,n,S^{2p-1}}}=\sum_{n=1}^{\infty}
\frac{m_{{\rm cex},q,n}}{((n+q)(n+2p-q-2))^s}\\
&=\sum_{n=q+1}^{\infty} \frac{m_{{\rm cex},q,n-q}}{(n(n-2\alpha_q))^s}.
\end{align*} 

Recalling the relation given in Section \ref{sconj1}
\begin{align*}
m_{{\rm cex},q,n-q} &= \frac{2}{q!(2p-q-2)!}\prod^{p}_{\substack{j=1,\\j\neq q+1}} (n-q-1+j)(n+2p-q-1-j)\\
&=\frac{2}{q!(2p-q-2)!}\prod^{p}_{\substack{j=1,\\j\neq q+1}} n(n-2\alpha_q)+ d^{q}_{j} \\
&=\frac{2}{q!(2p-q-2)!} \sum^{p-1}_{j=0} e_{p-1-j}(d^{q})(n(n-2\alpha_q))^j.
\end{align*}
Thus
\begin{align*}
\zeta(s, U_{q,S^{2p-1}}) &= \frac{2}{q!(2p-q-2)!} \sum_{j=0}^{p-1} e_{p-j-1}(d^q)
\left(\zeta_{-\alpha_q,j}(s) - \sum^{q}_{n=1} \frac{1}{(n(n-2\alpha_q))^{s-j}}\right)\\
&= \frac{2}{q!(2p-i-2)!} \sum_{j=0}^{p-1} e_{p-j-1}(d^q) \zeta_{-\alpha_q,j}(s).
\end{align*} 
where
\[
\sum_{j=0}^{p-1} e_{p-j-1}(d^q) \frac{1}{(n(n+2p-2q-2))^{s-j}} = 0,
\] 
for  $1\leq n \leq q$, by result of \cite{WY}. For $s=0$, we obtain
\begin{align*}
\zeta(0,U_{q,S^{2p-1}}) =& \frac{2}{q!(2p-q-2)!} \sum_{j=0}^{p-1} e_{p-j-1}(d^q)
\zeta_{-\alpha_q,j}(0)\\
=&\frac{2}{q!(2p-q-2)!} \left(e_{p-1}(d^{q})\left(-\frac{1}{2} -((p-q-2)+1) \right) \right.\\
&\left.+\sum_{j=1}^{p-1} e_{p-j-1}(d^q)\left((-1)^{j+1} \frac{\alpha_q^{2j}}{2} -
\sum^{p-q-2}_{n=1}(n^2 -\alpha_q^2)^j\right)\right)\\
\end{align*}
\begin{align*}
=&\frac{2}{q!(2p-q-2)!} \left(-e_{p-1}(d^{q}) \right.\\
&\left.+ \sum^{p-1}_{j=0}
e_{p-j-1}(d^{q})\left((-1)^{j+1} \frac{\alpha_q^{2j}}{2} - \sum_{n=1}^{p-q-2}(n^2-\alpha_q^2)^j\right)\right)\\
=&\frac{2}{q!(2p-q-2)!} \left((-1)^{q+1}\frac{q!(2p-q-2)!}{2} \right.\\
&\left.+ \sum^{p-1}_{j=0} e_{p-j-1}(d^{q})\left((-1)^{j+1} \frac{\alpha_q^{2j}}{2} -
\sum_{n=1}^{p-q-2}(n^2-\alpha_q^2)^j\right)\right).
\end{align*}

To conclude the proof, note that the second term vanishes. For first, as showed in the proof of Lemma \ref{l26}, the numbers $n=1,2,\ldots,-\alpha_q$ are roots of the polynomial $\sum^{p-1}_{j=0} e_{p-j-1}(d^{q}) (n^2-\alpha_q^2)^j$,  
and second: 
\begin{align*}
\sum^{p-1}_{j=0} e_{p-j-1}(d^{q})(-1)^j\alpha_q^{2j}&= \sum^{p-1}_{j=0}
e_{p-j-1}(d^{q})(-d^q_p)^{j}=  \prod^{p}_{\substack{j=1,\\j\neq q+1}} (-d^{q}_{p} + d^{q}_{j}) =-\prod^{p}_{\substack{j=1,\\j\neq i+1}} (p-j)^2= 0.
\end{align*} 
\end{proof}

\subsection{The proof of the conjecture} We need some notation. Set
\begin{align*}
D(q,k,t) =&\frac{2}{q!(2p-q-2)!}\binom{-\frac{2k+1}{2}}{t} \sum^{p-1}_{l=k+t} e_{p-1-l}(d^{q})(-1)^{l-k}
\binom{-\frac{1}{2}-k-t+l}{l-k-t}\alpha_q^{2(l-k)},  \\
F(q,k) =& \Rz_{s=0}\Phi_{2k+1,q}(s), \hspace{40pt}1\leq k \leq p-1,
\end{align*}
for $0\leq q \leq p-1$,  $F(q,0) = 2$ for $0\leq q \leq p-2$, and $F(p-1,0) = 1$. Then, by Corollary \ref{r7}, the residues of  $\zeta(s,U_q)$, for $0\leq q \leq p-2$, are
\[
\Ru_{s=2k+1}\zeta(s,U_q ) = \frac{1}{\nu^{2k+1}} \sum^{p-1-k}_{t=0} \frac{1}{\nu^{2t}}D(q,k,t),
\]
for $k=0,\ldots,p-1$, and when $q=p-1$:
\[
\Ru_{s=2k+1} \zeta(s,U_{p-1}) = \frac{1}{\nu^{2k+1}} D(p-1,k,0),
\]
with $k=0,\ldots,p-1$. Now, for $0\leq q \leq p-1$, it is easy to see that
\begin{equation*}
\Rz_{s=0}\Phi_{2k+1,q}(s)\Ru_{s=2k+1}\zeta(s,U_q ) =\frac{F(q,k)}{\nu^{2k+1}} \sum^{p-1-k}_{t=0}
\frac{1}{\nu^{2t}}D(q,k,t)
\end{equation*} 
and hence
\[
t'_{q,{\rm sing}}(0) = \frac{1}{2}\sum^{p-1}_{k=0}\Rz_{s=0}\Phi_{2k+1,q}(s)\Ru_{s=2k+1}\zeta(s,U_q)
=\frac{1}{2}\sum^{p-1}_{k=0}\frac{F(q,k)}{\nu^{2k+1}} \sum^{p-1-k}_{t=0} \frac{1}{\nu^{2t}}D(q,k,t).
\]

On the other side, set:
\[
A_{\rm BM,abs}(C_l S^{2p-1}_{\sin\alpha})=  \sum^{p-1}_{k=0} \frac{1}{\nu^{2k+1}} \tilde Q_p(k), \hspace{30pt} \tilde Q_p(k) = \sum^{k}_{j=0} N_j(p,k),
\] 
where
\[
N_j(p,k) =\frac{(2p-1)!}{4^p (p-1)!}\frac{1}{(p-1-k)!(2k+1)}\frac{(-1)^{k-j}2^{j+1}}{(k-j)!(2j+1)!!}.
\]

\begin{lem} $\frac{1}{2} \sum^{p-1}_{q=0} (-1)^{q} t'_{q,{\rm sing}}(0)$ is an odd polynomial in $\frac{1}{\nu}$.
\end{lem}

\begin{proof} This follows by rearrangement of the finite sum:
\begin{align*}
\frac{1}{2} \sum^{p-1}_{q=0} (-1)^{q} t'_{q,{\rm sing}}(0) &=\frac{1}{4}\sum^{p-1}_{q=0} (-1)^{q}
\sum^{p-1}_{k=0}F(q,k) \sum^{p-1-k}_{t=0}
\frac{1}{\nu^{2(t+k)+1}}D(q,k,t) \\
&=\frac{1}{4} \sum^{p-1}_{k=0}\frac{1}{\nu^{2k+1}}  \sum^{p-1}_{q=0} (-1)^{q} \sum^{k}_{j=0}F(q,j)D(q,j,k-j)\\
&=\frac{1}{4} \sum^{p-1}_{k=0}\frac{1}{\nu^{2k+1}}   \sum^{k}_{j=0} \sum^{p-1}_{q=0} (-1)^{q}F(q,j)D(q,j,k-j).
\end{align*}
\end{proof}

Then, set:
\[
\frac{1}{2} \sum^{p-1}_{q=0} (-1)^{q} t'_{q,{\rm sing}}(0)= \sum^{p-1}_{k=0} \frac{1}{\nu^{2k+1}} Q_p(k), \hspace{30pt} Q_p(k)=
\sum^{k}_{j=0} M_j(p,k),
\] 
where
\begin{align*}
 M_j(p,k) &= \sum^{p-1}_{q=0} (-1)^{q}F(q,j)D(q,j,k-j)\\
&=\sum^{p-1}_{q=0} (-1)^{q}\frac{2F(q,j)}{4(2p-2)!}\binom{2p-2}{q}\binom{-\frac{1}{2}-j}{k-j}\alpha_q^{-2j}
\sum^{p-1}_{l=k} e_{p-1-l}(d^{q})\alpha_q^{2l} \binom{-\frac{1}{2}}{l-k}.
\end{align*}

This shows that all we need to prove to prove the conjecture is the identity: $M_{j}(p,k) = N_{j}(p,k)$. This is in the next two lemmas. Before, we need some further notation and combinatorics. First, recall that if
\[
f_h(x) = e_h\left(x^2-(p-1)^2,x^2-(p-2)^2,\ldots,x^2-1^2,x^2\right),
\] 
then $f_h(\alpha_q) = e_h(d^q)$, and $f_h(x)$, for $h\geq 1$, is a polynomial of the following type:
\beq\label{F}
\begin{aligned}
f_h(x) &=\sum_{0\leq j_1\leq j_2 \leq \ldots \leq j_h \leq p-1} (x^2 - j_1^2)(x^2 - j_2^2)\ldots(x^2 - j_h^2) =\binom{p}{h} x^{2h} + \sum^{h-1}_{s=0} c^{h}_s x^{2s}.
\end{aligned} 
\eeq

Second, we have the following four identities. The first three can be found in  \cite{GZ},  $0.151.4$, $0.154.5$ and $0.154.6$, but see \cite{Kra} for a proof. The fourth is in \cite{GR}, equation (5.3).

\begin{align}
\label{ida1}\sum^{n}_{k=0}  \frac{(-1)^{k}}{(2n)!}\binom{2n}{k} =&  \frac{(-1)^{n}}{(2n)!}\binom{2n-1}{n}= \frac{(-1)^{n}}{2(2n)!}\binom{2n}{n},\\
\label{idA2}\sum^{n}_{k=0} (-1)^{n} \binom{n}{k} (\alpha + k)^{n} =& (-1)^{n}n!\\
\label{idA3}\sum^{N}_{k=0} (-1)^{n} \binom{N}{k} (\alpha + k)^{n-1} =& 0,\\
\label{idB}\sum^{n}_{l=0}  \binom{n+1}{l+1}\binom{-\frac{1}{2}}{l-k} =& \binom{n+\frac{1}{2}}{n-k} =
\frac{(2n+1)!!}{2^{n-k}(n-k)!(2k+1)!!}.
\end{align}
with $1\leq n \leq N$ and $\alpha \in \R$.

\begin{lem}\label{lema M0pk} For  $0\leq k \leq p-1$, we have that $M_{0}(p,k) = N_{0}(p,k)$.
\end{lem}

\begin{proof} Since $j=0$, 
\begin{align*}
M_0(p,k) =& \sum^{p-1}_{q=0} (-1)^{q}\frac{2F(q,0)}{4(2p-2)!}\binom{2p-2}{q}\binom{-\frac{1}{2}}{k}
\sum^{p-1}_{l=k} e_{p-1-l}(d^{q})\alpha_q^{2l} \binom{-\frac{1}{2}}{l-k},\\
N_0(p,k) =&\frac{(2p-1)!}{2^{2p-1}
(p-1)!}\frac{1}{(p-1-k)!(2k+1)}\frac{(-1)^{k}}{k!}.
\end{align*}

Consider first $k\neq 0$. Then,
\begin{align*}
M_0(p,k) =& \binom{-\frac{1}{2}}{k}\sum^{p-1}_{q=0} (-1)^{q}\frac{1}{(2p-2)!}\binom{2p-2}{q}
\sum^{p-1}_{l=k} f_{p-1-l}(\alpha_q)\alpha_q^{2l} \binom{-\frac{1}{2}}{l-k}\\
=& \binom{-\frac{1}{2}}{k}\sum^{p-1}_{q=0} (-1)^{q}\frac{1}{(2p-2)!}\binom{2p-2}{q}
\sum^{p-1}_{l=k} \binom{p}{p-1-l}\alpha_q^{2p-2-2l}\alpha_q^{2l} \binom{-\frac{1}{2}}{l-k}\\
&+\binom{-\frac{1}{2}}{k}\sum^{p-1}_{q=0} (-1)^{q}\frac{1}{(2p-2)!}\binom{2p-2}{q}
\sum^{p-2}_{l=k} \sum^{p-2-l}_{s=0}c^{p-1-l}_s \alpha_q^{2s}\alpha_q^{2l} \binom{-\frac{1}{2}}{l-k}\\
=& \binom{-\frac{1}{2}}{k}\sum^{p-1}_{q=0} (-1)^{q}\frac{1}{(2p-2)!}\binom{2p-2}{q}\alpha_q^{2p-2}
\sum^{p-1}_{l=k} \binom{p}{p-1-l} \binom{-\frac{1}{2}}{l-k}\\
&+\sum^{p-2}_{l=k}\sum^{p-2-l}_{s=0}c^{p-1-l}_s\binom{-\frac{1}{2}}{k}\binom{-\frac{1}{2}}{l-k}\sum^{p-1}_{q=0}
(-1)^{q}\frac{1}{(2p-2)!}\binom{2p-2}{q} \alpha_q^{2s+2l}.
\end{align*}

Using the identity in equation (\ref{idA3}), the second term in the last line vanishes since  $2s+2l<2p-2$. Thus,
\begin{align*}
M_0(p,k) &= \binom{-\frac{1}{2}}{k}\sum^{p-1}_{q=0} (-1)^{q}\frac{1}{(2p-2)!}\binom{2p-2}{q}\alpha_q^{2p-2}
\sum^{p-1}_{l=k} \binom{p}{p-1-l} \binom{-\frac{1}{2}}{l-k}\\
&= \frac{1}{2}\binom{-\frac{1}{2}}{k}
\sum^{p-1}_{l=k} \binom{p}{p-1-l} \binom{-\frac{1}{2}}{l-k}\\
&= \frac{1}{2}\binom{-\frac{1}{2}}{k}\binom{p}{k+1}\frac{(k+1)!}{p!} \frac{(2p-1)!!}{(2k+1)!!}\frac{2^{k+1}}{2^p}\\
&= \frac{(-1)^k}{k!}\frac{1}{(p-k-1)!}\frac{(2p-1)!}{(2k+1)}\frac{1}{2^{2p-1}(p-1)!}= N_0(p,k).
\end{align*}

Next, consider $k=0$. Then,
\begin{align*}
M_0(p,0) =& \sum^{p-2}_{q=0} (-1)^{q}\frac{1}{(2p-2)!}\binom{2p-2}{q}
\sum^{p-1}_{l=0} f_{p-1-l}(\alpha_q)\alpha_q^{2l} \binom{-\frac{1}{2}}{l} + \frac{1}{2}\\
=& \sum^{p-1}_{q=0} (-1)^{q}\frac{1}{(2p-2)!}\binom{2p-2}{q}
\sum^{p-1}_{l=0} f_{p-1-l}(\alpha_q)\alpha_q^{2l} \binom{-\frac{1}{2}}{l} -1 + \frac{1}{2}\\
=& \sum^{p-1}_{q=0} (-1)^{q}\frac{1}{(2p-2)!}\binom{2p-2}{q}
\sum^{p-1}_{l=0} \binom{p}{p-1-l}\alpha_q^{2p-2} \binom{-\frac{1}{2}}{l}\\
&+\sum^{p-1}_{q=0} (-1)^{q}\frac{1}{(2p-2)!}\binom{2p-2}{q}
c^{p-1}_0 -\frac{1}{2}\\
 =& \sum^{p-1}_{q=0} (-1)^{q}\frac{1}{(2p-2)!}\binom{2p-2}{q}
\sum^{p-1}_{l=0} \binom{p}{p-1-l}\alpha_q^{2p-2} \binom{-\frac{1}{2}}{l}\\
&+\frac{(-1)^{p-1}}{2(2p-2)!}\binom{2p-2}{p-1} (-1)^{p-1}(p-1)!(p-1)! - \frac{1}{2}\\
=& \sum^{p-1}_{q=0} (-1)^{q}\frac{1}{(2p-2)!}\binom{2p-2}{q} \sum^{p-1}_{l=0} \binom{p}{p-1-l}\alpha_q^{2p-2}
\binom{-\frac{1}{2}}{l} +\frac{1}{2} - \frac{1}{2}\\
=& \sum^{p-1}_{q=0} (-1)^{q}\frac{1}{(2p-2)!}\binom{2p-2}{q} \alpha_q^{2p-2} \sum^{p-1}_{l=0} \binom{p}{p-1-l}
\binom{-\frac{1}{2}}{l}\\
=& \frac{1}{2} \sum^{p-1}_{l=0} \binom{p}{p-1-l} \binom{-\frac{1}{2}}{l}= \frac{1}{2^{p}} \frac{(2p-1)!!}{(p-1)!} =
\frac{2p-1}{2^{2p-1}}\binom{2p-2}{p-1} = N_0(p,0)
\end{align*}

\end{proof}

\begin{lem}\label{Mjpk} For  $1\leq j \leq p-1$, we have that $M_{j}(p,k) = N_{j}(p,k)$.
\end{lem}

\begin{proof} Note that $j\leq k$, and hence $1\leq j \leq k \leq p-1$. Recall that
\[
F(q,j) = \frac{2}{2j+1}\alpha_q^{2j}+\sum^{j-1}_{t=1} k_{2j+1,t} \alpha_q^{2t} + 2\Rz_{z=0}\Phi_{2j+1,p-1}(s), 
\] 
by Corollary \ref{c44}. Set $k_{2j+1,0} = 2\Rz_{z=0}\Phi_{2j+1,p-1}(s)$.  We split the proof in three cases. First, for $j = k<p-1$, we have
\begin{align*}
M_j(p,j) =& \sum^{p-1}_{q=0} (-1)^{q}\frac{F(q,j)}{2(2p-2)!}\binom{2p-2}{q}
\sum^{p-1}_{l=j} e_{p-1-l}(d^{q})\alpha_q^{2l-2j} \binom{-\frac{1}{2}}{l-j}\\
=& \sum^{p-1}_{q=0} (-1)^{q}\frac{\alpha_q^{2j}}{(2j+1)(2p-2)!}\binom{2p-2}{q}
\sum^{p-1}_{l=j} f_{p-1-l}(\alpha_q)\alpha_q^{2l-2j} \binom{-\frac{1}{2}}{l-j}\\
&+\sum^{j-1}_{t=0} k_{t,j} \sum^{p-1}_{q=0} (-1)^{q}\frac{\alpha_q^{2t}}{(2p-2)!}\binom{2p-2}{q} \sum^{p-1}_{l=j}
f_{p-1-l}(\alpha_q)\alpha_q^{2l-2j} \binom{-\frac{1}{2}}{l-j}
\end{align*}

Using the formula in equation (\ref{F}) for the functions 
 $f_{p-1-l}(\alpha_q)$, we get
\begin{align*}
M_j(p,j)=& \sum^{p-1}_{q=0} (-1)^{q}\frac{1}{(2j+1)(2p-2)!}\binom{2p-2}{q}
\sum^{p-1}_{l=j} \binom{p}{p-1-l}\alpha_q^{2p-2-2l}\alpha_q^{2l} \binom{-\frac{1}{2}}{l-j}\\
&+\sum^{p-1}_{q=0} (-1)^{q}\frac{1}{(2j+1)(2p-2)!}\binom{2p-2}{q}
\sum^{p-2}_{l=j} \sum^{p-2-l}_{s=0}c_s \alpha_q^{2s + 2l} \binom{-\frac{1}{2}}{l-j}\\
&+\sum^{j-1}_{t=0} k_{2j+1,t} \sum^{p-1}_{q=0} (-1)^{q}\frac{1}{(2p-2)!}\binom{2p-2}{q}
\sum^{p-1}_{l=j} \binom{p}{p-1-l}\alpha_q^{2p-2+2t-2j} \binom{-\frac{1}{2}}{l-j}\\
&+\sum^{j-1}_{t=0} k_{2j+1,t} \sum^{p-1}_{q=0} (-1)^{q}\frac{1}{(2p-2)!}\binom{2p-2}{q}
\sum^{p-2}_{l=j} \sum^{p-2-l}_{s=0}c_s\alpha_q^{2s+2l+2t-2j} \binom{-\frac{1}{2}}{l-j}\\
= &\frac{1}{(2j+1)}\sum^{p-1}_{q=0} (-1)^{q}\frac{1}{(2p-2)!}\binom{2p-2}{q}\alpha_q^{2p-2}
\sum^{p-1}_{l=j} \binom{p}{p-1-l} \binom{-\frac{1}{2}}{l-j}\\
=& \frac{1}{2(2j+1)}\sum^{p-1}_{l=j} \binom{p}{p-1-l} \binom{-\frac{1}{2}}{l-j}
=\frac{1}{2(2j+1)}\frac{1}{(p-1-j)!}\frac{(2p-1)!!}{(2j+1)!!}\frac{2^{j+1}}{2^{p}}\\
&= \frac{1}{(2j+1)(p-1-j)!}\frac{(2p-1)!}{(2j+1)!!}\frac{2^{j}}{2^{2p-1}}\frac{}{(p-1)!} =N_j(p,j),
\end{align*}
where the first three terms in the first equation vanish because $s+l<p-1$ and $t-j\leq-1$. The second case is
$j=k=p-1$. Then,
\begin{align*}
M_{p-1}(p,p-1) =& \sum^{p-1}_{q=0} (-1)^{q}\frac{F(q,p-1)}{2(2p-2)!}\binom{2p-2}{q}\\
=& \sum^{p-1}_{q=0} (-1)^{q}\frac{\alpha_q^{2p-2}}{(2p-1)(2p-2)!}\binom{2p-2}{q}\\
&+\sum^{p-2}_{t=0} k_{2j+1,t}
\sum^{p-2}_{q=0} (-1)^{q}\frac{\alpha_q^{2t}}{2(2p-2)!}\binom{2p-2}{q}\\
&+ \sum^{p-2}_{t=0} \frac{ k_{2j+1,t}}{2} (-1)^{p-1}\frac{\alpha_{p-1}^{2t}}{2(2p-2)!}\binom{2p-2}{p-1}\\
= &\frac{1}{2(2p-1)}+ k_{2j+1,0} \sum^{p-1}_{q=0} (-1)^{q}\frac{1}{2(2p-2)!}\binom{2p-2}{q}\\
&-k_{2j+1,0}(-1)^{p-1}\frac{1}{2(2p-2)!}\binom{2p-2}{p-1}\\
&+ \frac{k_{2j+1,0}}{2} (-1)^{p-1}\frac{1}{2(2p-2)!}\binom{2p-2}{p-1}\\
 =& \frac{1}{2(2p-1)}+ \frac{k_{2j+1,0}}{2} \frac{(-1)^{p-1}}{(p-1)!(p-1)!} -
\frac{k_{2j+1,0}}{2}\frac{(-1)^{p-1}}{(p-1)!(p-1)!}\\
=& \frac{1}{2(2p-1)} = N_{p-1}(p,p-1).
\end{align*}

The last case is  $1\leq j <k$. Then,
\begin{align*}
M_j(p,k) =& \sum^{p-1}_{i=0} (-1)^{q}\frac{2F(q,j)}{4(2p-2)!}\binom{2p-2}{q}\binom{-\frac{1}{2}-j}{k-j}\alpha_q^{-2j}
\sum^{p-1}_{l=k} e_{p-1-l}(d^{q})\alpha_q^{2l} \binom{-\frac{1}{2}}{l-k},\\
=& \binom{-\frac{1}{2}-j}{k-j}\sum^{p-1}_{q=0} (-1)^{q}\frac{1}{(2j+1)(2p-2)!}\binom{2p-2}{q}
\sum^{p-1}_{l=k} \binom{p}{p-1-l}\alpha_q^{2p-2-2l}\alpha_q^{2l} \binom{-\frac{1}{2}}{l-k}\\
&+\binom{-\frac{1}{2}-j}{k-j}\sum^{p-1}_{q=0} (-1)^{q}\frac{1}{(2j+1)(2p-2)!}\binom{2p-2}{q}
\sum^{p-2}_{l=k} \sum^{p-2-l}_{s=0}c_s \alpha_q^{2s + 2l} \binom{-\frac{1}{2}}{l-k}\\
&+\binom{-\frac{1}{2}-j}{k-j}\sum^{j-1}_{t=0} k_{2j+1,t} \sum^{p-1}_{q=0} (-1)^{q}\frac{1}{(2p-2)!}\binom{2p-2}{q}
\sum^{p-1}_{l=k} \binom{p}{p-1-l}\alpha_q^{2p-2+2t-2j} \binom{-\frac{1}{2}}{l-k}\\
&+\binom{-\frac{1}{2}-j}{k-j}\sum^{j-1}_{t=0} k_{2j+1,t} \sum^{p-1}_{q=0} (-1)^{q}\frac{1}{(2p-2)!}\binom{2p-2}{q}
\sum^{p-2}_{l=k} \sum^{p-2-l}_{s=0}c_s\alpha_q^{2s+2l+2t-2j} \binom{-\frac{1}{2}}{l-k}\\
\end{align*}
\begin{align*}
=& \binom{-\frac{1}{2}-j}{k-j}\sum^{p-1}_{q=0} (-1)^{q}\frac{1}{(2j+1)(2p-2)!}\binom{2p-2}{q}\alpha_q^{2p-2}
\sum^{p-1}_{l=k} \binom{p}{p-1-l} \binom{-\frac{1}{2}}{l-k}\\
=& \binom{-\frac{1}{2}-j}{k-j}\frac{1}{2(2j+1)}\sum^{p-1}_{l=k} \binom{p}{p-1-l} \binom{-\frac{1}{2}}{l-k}\\
= &\binom{-\frac{1}{2}-j}{k-j}\frac{1}{2(2j+1)}\frac{(2p-1)!!}{(p-1-k)!(2k+1)!!2^{p-k-1}}\\
= &\frac{(-1)^{k-j}}{(k-j)!}\frac{2^{j}}{2^{k}}\frac{(2k-1)!!}{(2j-1)!!}\frac{1}{2(2j+1)}\frac{(2p-1)!!}{(p-1-k)!(2k+1)!!2^{p-k-1}}\\
=& \frac{(-1)^{k-j}}{(k-j)!}\frac{2^{j}}{2^{2p-1}(p-1)!}\frac{1}{(2j+1)!!}\frac{(2p-1)!}{(p-1-k)!(2k+1)} = N_j(p,k).
\end{align*}

\end{proof}

\section{The proof of Theorem \ref{t03}: low dimensional cases}
\label{s7}

We prove the two cases  $m=2p-1=3$, and $m=2p-1=5$ independently, in the following two subsections. In the last subsection we give some remarks on the general case. Even if the proofs of the two cases $m=3$ and $m=5$ follow the same line, we prefer to give details separately, for two reasons: first, in order to improve readability, and second to avoid a problem that will be made clear in Section \ref{gene} below. In each case the proof is in two parts: in the first we compute the anomaly boundary term, as defined in Section \ref{Lap1.2}, in the second we compute the singular term in the analytic torsion, using Proposition \ref{ppoo}.

\subsection{Case $m=3$}
\subsubsection{Part 1} Since $m=3$, the unique terms that give non trivial contribution in the Berezin integral appearing in equation (\ref{ebm1}) are those homogeneous of degree 3. By definition
\[
\e^{-\frac{1}{2}\hat{\tilde\Omega}-u^2 \S_1^2}=1-\frac{1}{2}\hat{\tilde\Omega}-u^2 \S_1^2+\dots,
\]
(recall that $\Theta=\tilde\Omega$, see Section \ref{Lap1.2}) and therefore the terms of degree 3 in 
\[
\e^{-\frac{1}{2}\hat{\tilde\Omega}-u^2 \S^2_1}\sum_{k=1}^\infty \frac{1}{\Gamma\left(\frac{k}{2}+1\right)}u^{k-1} \S_1^k,
\]
are
\[
-\frac{2}{3\sqrt{\pi}} u^2 \S_1^3-\frac{1}{\sqrt{\pi}} \hat{\tilde\Omega} S_1.
\]

Applying the definition in equation (\ref{ebm1}), this gives
\beq\label{n}\begin{aligned}
B(\nabla_1)&=\frac{1}{2}\int_0^1\int^B
\e^{-\frac{1}{2}\hat{\tilde \Omega}-u^2 \mathcal{S}_j^2}\sum_{k=1}^\infty \frac{1}{\Gamma\left(\frac{k}{2}+1\right)}u^{k-1}
\mathcal{S}_j^k du\\
&=\frac{1}{2}\int_0^1\int^B \left(-\frac{2}{3\sqrt{\pi}} u^2 \S_1^3-\frac{1}{\sqrt{\pi}} \hat{\tilde\Omega} S_1\right)du\\
&=-\frac{1}{2\sqrt{\pi}}\int^B \hat{\tilde\Omega} S_1-\frac{1}{9\sqrt{\pi}}\int^B \S_1^3.
\end{aligned}
\eeq

By equation (\ref{sr1})
\beq\label{tot}
\S_1^3=-\frac{1}{8}\left(\sum_{k=1}^{m}b^*_k\wedge\hat e^*_{k}\right)^3=\frac{3}{4} dvol_g \wedge \hat e_1^*\wedge\hat e_2^*\wedge\hat e_3^*.
\eeq

This follows by direct calculation. For example
\beq\label{ss1}\begin{aligned}
\S_1^2&=\frac{1}{4}\left(\sum_{k=1}^{3}b^*_k\wedge\hat e^*_{k}\right)^2\\
&=\frac{1}{4}(b_1^*\wedge\hat e_1^*+b_2^*\wedge\hat e_2^*+b_3^*\wedge\hat e_3^*)^2\\
&=\frac{1}{4}(b_1^*\wedge\hat e_1^*\wedge b_2^*\wedge\hat e_2^*+\dots +b_2^*\wedge\hat e_2^*\wedge b_1^*\wedge\hat e_1^*+\dots)\\
&=\frac{1}{4}(-b_1^*\wedge b_2^*\wedge \hat e_1^*\wedge\hat e_2^*+\dots -b_2^*\wedge b_1^*\wedge \hat e_2^*\wedge\hat e_1^*+\dots)\\
&=\frac{1}{4}(- 2 b_1^*\wedge b_2^*\wedge \hat e_1^*\wedge\hat e_2^*+\dots )\\
&=-\frac{1}{2}\sum_{j<k=1}^3 b_j^*\wedge b_k^*\wedge\hat e_j^*\wedge\hat e_k^*,
\end{aligned}
\eeq
while
\[
(b_1^*\wedge b_2^*\wedge \hat e_1^*\wedge \hat e_2^*)\wedge(b_3^*\wedge \hat e_3^*)
=b_1^*\wedge b_2^*\wedge  b_3^*\wedge \hat e_1^*\wedge\hat e_2^*\wedge \hat e_3^*.
\]

Thus,
\[
\int^B \S^3_1=\frac{3}{4\pi^\frac{3}{2}}d vol_g.
\]

By equations (\ref{sr1}) and (\ref{pippo}), 
\begin{align*}
\hat{\tilde\Omega}\S_1&=-\frac{1}{4}\left(\sum_{k,l=1}^{3}\tilde\Omega_{kl} \wedge \hat  e^*_{k}\wedge  \hat e^*_{l}\right)
\wedge\left(\sum_{k=1}^{3}b^*_k\wedge\hat e^*_{k}\right).\\
\end{align*}

Direct calculations give
\begin{align*}
\hat{\tilde\Omega}\S_1=&-\frac{1}{2}(\Omega_{23}\wedge b_1^*-\Omega_{13}\wedge b_2^*+\Omega_{12}\wedge b_3^*)\wedge
\hat e_1^*\wedge \hat e_2^*\wedge \hat e_3^*\\
=&-\frac{1}{2}(R_{2332}+R_{1331}+R_{1221})\hat e_1^*\wedge \hat e_2^*\wedge \hat e_3^*\\
=&-\frac{1}{4}\tilde\tau \hat e_1^*\wedge \hat e_2^*\wedge \hat e_3^*,
\end{align*}
and hence
\[
\int^B \hat{\tilde\Omega}\S_1 =\frac{1}{4\pi^\frac{3}{2}}\sum_{k,l=1}^3 \tilde R_{kllk} dvol_g.
\]

Substitution in equation (\ref{n}) gives
\begin{align*}
B(\nabla_1)&=\frac{1}{4\sqrt{\pi}}\int^B \hat{\tilde\Omega} S_1-\frac{1}{9\sqrt{\pi}}\int^B \S_1^3\\
&=\frac{1}{8\pi^2}\tilde\tau dvol_g-\frac{1}{12\pi^2} dvol_g.
\end{align*}

By the formula in equation (\ref{anom}), the anomaly boundary term is
\[
A_{\rm BM,abs}(\b C_lW)=\frac{1}{16\pi^2}\int_{\b C_lW}\tilde\tau dvol_g-\frac{1}{24\pi^2} \int_{\b C_lW} dvol_g.
\]

\subsubsection{Part 2} By Proposition \ref{ppoo}, with $p=2$, 
\begin{align*}
\log T_{\rm sing}(C_l W)=& \frac{1}{2}\sum_{q=0}^{1} (-1)^q \sum_{j=0}^{1}\Rz_{s=0}\Phi_{2j+1,q}(s)\Ru_{s=j+\frac{1}{2}}\zeta_{\rm cex}\left(s,\tilde\Delta^{(q)}+\alpha_q^2\right).
\end{align*}

Since $p=2$, $\alpha_0=-1$ and $\alpha_1=0$. Since there are no exact 0-forms
\[
\zeta_{\rm cex}\left(s,\tilde\Delta^{(0)}+\alpha_0^2\right)=\zeta\left(s,\tilde\Delta^{(0)}+\alpha_0^2\right).
\]

By Lemma \ref{lel}, 
\begin{align*}
\Ru_{s=\frac{3}{2}}\zeta\left(s,\tilde\Delta^{(0)}+\alpha_0^2\right)&=\Ru_{s=\frac{3}{2}}\zeta\left(s,\tilde\Delta^{(0)}\right),\\
\Ru_{s=\frac{1}{2}}\zeta\left(s,\tilde\Delta^{(0)}+\alpha_0^2\right)&=\Ru_{s=\frac{1}{2}}\zeta\left(s,\tilde\Delta^{(0)}\right)-\frac{1}{2}\Ru_{s=\frac{3}{2}}\zeta\left(s,\tilde\Delta^{(0)}\right).
\end{align*}

By duality (see Section \ref{forms}) 
\[
\zeta_{\rm cex}(s,\tilde\Delta^{(1)})=\zeta(s,\tilde\Delta^{(1)})-\zeta_{\rm ex}(s,\tilde\Delta^{(1)})=
\zeta(s,\tilde\Delta^{(1)})-\zeta_{\rm cex}(s,\tilde\Delta^{(0)}),
\]
and also
\begin{align*}
\Ru_{s=\frac{1}{2}}\zeta(s,\tilde\Delta^{(1)})=-3\Ru_{s=\frac{1}{2}}\zeta(s,\tilde\Delta^{(0)}),\\
\Ru_{s=\frac{3}{2}}\zeta(s,\tilde\Delta^{(1)})=3\Ru_{s=\frac{3}{2}}\zeta(s,\tilde\Delta^{(0)}).\\
\end{align*}

Putting all together, we obtain
\begin{align*}
&\log T_{\rm sing}(C_l W)=\frac{1}{2} \left(\Rz_{s=0}\Phi_{1,0}(s)+\Rz_{s=0}\Phi_{1,1}(s)+3\Rz_{s=0}\Phi_{1,1}(s)\right)\Ru_{s=\frac{1}{2}}\zeta(s,\tilde\Delta^{(0)})\\
&+\frac{1}{2}\left(\Rz_{s=0}\Phi_{3,1}(s)+\Rz_{s=0}\Phi_{3,0}(s)-\frac{1}{2}\Rz_{s=0}\Phi_{1,0}(s)-3\Rz_{s=0}\Phi_{3,1}(s)\right)\Ru_{s=\frac{3}{2}}\zeta(s,\tilde\Delta^{(0)}).
\end{align*}

By Corollaries \ref{c33} (when $q=1$), and \ref{c44} (when $q=0$)
\begin{align*}
\Rz_{s=0}\Phi_{1,1}(s)&=1,&\Rz_{s=0}\Phi_{3,1}(s)&=\frac{2}{315},\\
\Rz_{s=0}\Phi_{1,0}(s)&=2,&\Rz_{s=0}\Phi_{3,0}(s)&=\frac{214}{315}.
\end{align*}

This gives
\[
\log T_{\rm sing}(C_l W)=3\Ru_{s=\frac{1}{2}}\zeta(s,\tilde\Delta^{(0)})-\frac{1}{6}\Ru_{s=\frac{3}{2}}\zeta(s,\tilde\Delta^{(0)}),
\]
and by Propositions \ref{coeff} and \ref{l3.1} 
\[
\log T_{\rm sing}(C_l W)=\frac{1}{16\pi^2}\int_{\b C_lW}\tilde\tau dvol_g-\frac{1}{24\pi^2}\int_{\b C_lW} dvol_g.
\]

\subsection{Case $m=5$}
\subsubsection{Part 1} Since $m=5$, the unique terms that gives non trivial contribution in the Berezin integral appearing in equation (\ref{ebm1}) are those homogeneous of degree 5. After some calculation, the terms of degree 3 in 
\[
\e^{-\frac{1}{2}\hat{\tilde\Omega}-u^2 \S^2_1}\sum_{k=1}^\infty \frac{1}{\Gamma\left(\frac{k}{2}+1\right)}u^{k-1} \S_1^k,
\]
are
\[
-\frac{1}{5\sqrt{\pi}} u^4 \S_1^5+\frac{1}{3\sqrt{\pi}} u^2 \hat{\tilde\Omega} S^3_1+\frac{1}{4\sqrt{\pi}}  \hat{\tilde\Omega}^2 S_1.
\]

Integration in $u$, as in  equation (\ref{ebm1}) gives
\beq\label{n1}\begin{aligned}
B(\nabla_1)
&=\frac{1}{2\sqrt{\pi}}\int^B \left(\frac{1}{25}  \S_1^5+\frac{1}{9} u^2 \hat{\tilde\Omega} S^3_1+\frac{1}{4}  \hat{\tilde\Omega}^2 S_1\right).
\end{aligned}
\eeq

We calculate the three terms appearing in the integrand.

Proceeding as for equations (\ref{ss1}), starting from equation (\ref{sr1}), we obtain
\begin{align*}
\S_1^2&=-\frac{1}{2}\sum_{j<k=1}^5 b_j^*\wedge b_k^*\wedge\hat e_j^*\wedge\hat e_k^*,
\end{align*}
and
\begin{align*}
\S_1^5&=(\S_1^2)^2 \S_1=-\frac{1}{2^5} \left(-\frac{1}{2}\sum_{j<k=1}^5 b_j^*\wedge b_k^*\wedge\hat e_j^*\wedge\hat e_k^*\right)^2\sum_{k=1}^{5}b^*_k\wedge\hat e^*_{k}\\
&=-\frac{5}{4} dvol_g \wedge \hat e_1^*\wedge\hat e_2^*\wedge\hat e_3^*\wedge \hat e_4^*\wedge \hat e_5^*,
\end{align*}
and recalling the definition in  (\ref{pippo}) of $\hat{\tilde \Omega}$,
\begin{align*}
\hat{\tilde\Omega}\S^3_1&=\frac{3}{8}\left(\sum_{k,l=1}^{3}\tilde\Omega_{kl} \wedge \hat  e^*_{k}\wedge  \hat e^*_{l}\right)
\wedge dvol_g \wedge \hat e_1^*\wedge \hat e_2^*\wedge \hat e_3^*\wedge \hat e_4^*\wedge \hat e_5^*\\
&=\frac{3}{8}\tilde\tau \wedge dvol_g  \wedge \hat e_1^*\wedge \hat e_2^*\wedge \hat e_3^*\wedge \hat e_4^*\wedge \hat e_5^*.
\end{align*}

The last term is
\begin{align*}
\hat{\tilde\Omega}^2 \S_1&=-\frac{1}{2^5} \left(\sum_{k,l=1}^{3}\tilde\Omega_{kl} \wedge \hat  e^*_{k}\wedge  \hat e^*_{l}\right)^2\left(\sum_{k=1}^{5}b^*_k\wedge\hat e^*_{k}\right).
\end{align*}

This term requires some noisy explicit calculations, that we omit here. The result is
\begin{align*}
\hat{\tilde\Omega}^2 \S_1&=-\frac{1}{2^5} \left(4|\tilde R|^2-16|\tilde Ric |^2+4\tilde \tau^2\right)dvol_g \wedge \hat e_1^*\wedge \hat e_2^*\wedge \hat e_3^*\wedge \hat e_4^*\wedge \hat e_5^*\\
&=\left(-\frac{1}{8} |\tilde R|^2+\frac{1}{2}|\tilde Ric|^2-\frac{1}{8}\tilde \tau^2\right)dvol_g \wedge \hat e_1^*\wedge \hat e_2^*\wedge \hat e_3^*\wedge \hat e_4^*\wedge \hat e_5^*.
\end{align*}

Substitution in equation (\ref{n1}) gives
\begin{align*}
B(\nabla_1)&=\frac{3dvol_g}{40\pi^3}-\frac{\tilde\tau dvol_g}{48\pi^3} +\frac{|\tilde R|^2dvol_g}{64\pi^3}  
-\frac{|\tilde Ric|^2dvol_g}{16\pi^3} +\frac{\tilde \tau^2dvol_g}{64\pi^3}.
\end{align*}

By the formula in equation (\ref{anom}), the anomaly boundary term is
\begin{align*}
A_{\rm BM,abs}(\b C_l W)=&\frac{3}{80\pi^3}\int_{\b C_lW} dvol_g-\frac{1}{96\pi^3} \int_{\b C_lW} \tilde\tau dvol_g+\frac{1}{128\pi^3} \int_{\b C_lW} |\tilde R|^2 dvol_g\\
&-\frac{1}{32\pi^3} \int_{\b C_lW} |\tilde Ric |^2 dvol_g+\frac{1}{128\pi^3}\int _{\b C_lW} \tilde \tau^2 dvol_g.
\end{align*}

\subsubsection{Part 2} By Proposition \ref{ppoo}, with $p=3$, 
\begin{align*}
\log T_{\rm sing}(C_l W)
=& \frac{1}{2}\sum_{q=0}^{2} (-1)^q \sum_{j=0}^{2}\Rz_{s=0}\Phi_{2j+1,q}(s)\Ru_{s=j+\frac{1}{2}}\zeta_{\rm cex}\left(s,\tilde\Delta^{(q)}+\alpha_q^2\right).
\end{align*}

Since $p=3$, $\alpha_0=-2$,  $\alpha_1=-1$, and $\alpha_2=0$. 
By duality (see Section \ref{forms}) 
\begin{align*}
\zeta_{\rm cex}(s,\tilde\Delta^{(0)})&=\zeta(s,\tilde\Delta^{(0)}),\\
\zeta_{\rm cex}(s,\tilde\Delta^{(1)}&=\zeta(s,\tilde\Delta^{(1)})-\zeta_{\rm ex}(s,\tilde\Delta^{(1)})=
\zeta(s,\tilde\Delta^{(1)})-\zeta_{\rm cex}(s,\tilde\Delta^{(0)}),\\
\zeta_{\rm cex}(s,\tilde\Delta^{(2)})&=\zeta(s,\tilde\Delta^{(2)})-\zeta_{\rm ex}(s,\tilde\Delta^{(2)})=
\zeta(s,\tilde\Delta^{(2)})-\zeta_{\rm cex}(s,\tilde\Delta^{(1)})\\
&=
\zeta(s,\tilde\Delta^{(2)})-\zeta(s,\tilde\Delta^{(1)})+\zeta(s,\tilde\Delta^{(0)}).
\end{align*}

By Lemma \ref{lel}, with $q=0$ and $1$,
\begin{align*}
\Ru_{s=\frac{5}{2}}\zeta\left(s,\tilde\Delta^{(q)}+\alpha_q^2\right)&=\Ru_{s=\frac{5}{2}}\zeta\left(s,\tilde\Delta^{(q)}\right),\\
\Ru_{s=\frac{3}{2}}\zeta\left(s,\tilde\Delta^{(q)}+\alpha_q^2\right)&=\Ru_{s=\frac{3}{2}}\zeta\left(s,\tilde\Delta^{(q)}\right)
-\frac{3}{2}\Ru_{s=\frac{5}{2}}\zeta\left(s,\tilde\Delta^{(q)}\right),\\
\Ru_{s=\frac{1}{2}}\zeta\left(s,\tilde\Delta^{(q)}+\alpha_q^2\right)&=\Ru_{s=\frac{1}{2}}\zeta\left(s,\tilde\Delta^{(q)}\right)-\frac{1}{2}\Ru_{s=\frac{3}{2}}\zeta\left(s,\tilde\Delta^{(q)}\right)
+\frac{3}{8}\Ru_{s=\frac{5}{2}}\zeta\left(s,\tilde\Delta^{(q)}\right).
\end{align*}

By Proposition \ref{coeff}
\begin{align*}
\Ru_{s=\frac{1}{2}}\zeta(s,\tilde\Delta^{(0)})&=\frac{e_{0,4}}{\Gamma(1/2)}=\frac{1}{2^5 3^2 5\pi^3}\left(5\int_{\b C_lW} \tilde\tau^2 dvol_g-2\int_{\b C_lW} |\tilde Ric|^2 dvol_g+2\int_{\b C_lW} |\tilde R|^2 dvol_g\right),\\
\Ru_{s=\frac{3}{2}}\zeta(s,\tilde\Delta^{(0)})&=\frac{e_{0,2}}{\Gamma(3/2)}=\frac{1}{96\pi^3}\int_{\b C_lW} \tilde\tau dvol_g,\\
\Ru_{s=\frac{5}{2}}\zeta(s,\tilde\Delta^{(0)})&=\frac{e_{0,0}}{\Gamma(5/2)}=\frac{1}{24\pi^3}\int_{\b C_lW}  dvol_g;
\end{align*}

\begin{align*}
\Ru_{s=\frac{1}{2}}\zeta(s,\tilde\Delta^{(1)})&=\frac{e_{1,4}}{\Gamma(1/2)}\\
&=\frac{1}{2^5 3^2 5\pi^3}\left(-35\int_{\b C_lW} \tilde\tau^2 dvol_g+170\int_{\b C_lW} |\tilde Ric|^2 dvol_g-20\int_{\b C_lW} |\tilde R|^2 dvol_g\right),\\
\Ru_{s=\frac{3}{2}}\zeta(s,\tilde\Delta^{(1)})&=\frac{e_{1,2}}{\Gamma(3/2)}=-\frac{1}{96\pi^3}\int_{\b C_lW} \tilde\tau dvol_g
=-\Ru_{s=\frac{3}{2}}\zeta(s,\tilde\Delta^{(0)}),\\
\Ru_{s=\frac{5}{2}}\zeta(s,\tilde\Delta^{(1)})&=\frac{e_{1,0}}{\Gamma(5/2)}=\frac{5}{24\pi^3}\int_{\b C_lW}  dvol_g=5\Ru_{s=\frac{5}{2}}\zeta(s,\tilde\Delta^{(0)});
\end{align*}

\begin{align*}
\Ru_{s=\frac{1}{2}}\zeta(s,\tilde\Delta^{(2)})&=\frac{e_{2,4}}{\Gamma(1/2)}\\
&=\frac{1}{2^5 3^2 5\pi^3}\left(50\int_{\b C_lW} \tilde\tau^2 dvol_g-200\int_{\b C_lW} |\tilde Ric|^2 dvol_g+110\int_{\b C_lW} |\tilde R|^2 dvol_g\right),\\
\Ru_{s=\frac{3}{2}}\zeta(s,\tilde\Delta^{(2)})&=\frac{e_{2,2}}{\Gamma(3/2)}=-\frac{8}{96\pi^3}\int_{\b C_lW} \tilde\tau dvol_g
=-8\Ru_{s=\frac{3}{2}}\zeta(s,\tilde\Delta^{(0)}),\\
\Ru_{s=\frac{5}{2}}\zeta(s,\tilde\Delta^{(2)})&=\frac{e_{2,0}}{\Gamma(5/2)}=\frac{10}{24\pi^3}\int_{\b C_lW}  dvol_g=10\Ru_{s=\frac{5}{2}}\zeta(s,\tilde\Delta^{(0)}).
\end{align*}

Summing up, after some calculations, we obtain
\begin{align*}
\log &T_{\rm sing}(C_l W)\\
=&\frac{1}{2} \left(\Rz_{s=0}\Phi_{1,0}(s)+\Rz_{s=0}\Phi_{1,1}(s)+\Rz_{s=0}\Phi_{1,2}(s)\right)\Ru_{s=\frac{1}{2}}\zeta(s,\tilde\Delta^{(0)})\\
&-\frac{1}{2} \left(\Rz_{s=0}\Phi_{1,1}(s)+\Rz_{s=0}\Phi_{1,2}(s)\right)\Ru_{s=\frac{1}{2}}\zeta(s,\tilde\Delta^{(1)})\\
&+\frac{1}{2} \Rz_{s=0}\Phi_{1,2}(s)\Ru_{s=\frac{1}{2}}\zeta(s,\tilde\Delta^{(2)})\\
&+\Rz_{s=0}\left(\frac{1}{2}\Phi_{3,0}(s)+\Phi_{3,1}(s)-\Phi_{3,2}(s)
-\Phi_{1,0}(s)-\frac{1}{2}\Phi_{1,1}(s)\right)\Ru_{s=\frac{3}{2}}\zeta(s,\tilde\Delta^{(0)})\\
&+\Rz_{s=0}\left(\frac{1}{2}\Phi_{5,0}(s)-2\Phi_{5,1}(s)+3\Phi_{5,2}(s)-3\Phi_{3,0}(s)+\frac{3}{4}\Phi_{1,0}(s)+\frac{3}{4}\Phi_{1,1}(s)\right)\Ru_{s=\frac{5}{2}}\zeta(s,\tilde\Delta^{(0)}).
\end{align*}

By Corollaries \ref{c33} (when $q=2$), and \ref{c44} (when $q=0,1$)
\begin{align*}
\Rz_{s=0}\Phi_{1,2}(s)&=1,&\Rz_{s=0}\Phi_{3,2}(s)&=\frac{2}{315},&\Rz_{s=0}\Phi_{5,2}(s)&=-\frac{346}{22522}\\
\Rz_{s=0}\Phi_{1,1}(s)&=2,&\Rz_{s=0}\Phi_{3,1}(s)&=\frac{214}{315},&\Rz_{s=0}\Phi_{5,1}(s)&=\frac{31706}{75075}\\
\Rz_{s=0}\Phi_{1,0}(s)&=2,&\Rz_{s=0}\Phi_{3,0}(s)&=\frac{844}{315},&\Rz_{s=0}\Phi_{5,0}(s)&=\frac{487876}{75075}.
\end{align*}

This gives
\begin{align*}
\log T_{\rm sing}(C_l W)=&\frac{5}{2}\Ru_{s=\frac{1}{2}}\zeta(s,\tilde\Delta^{(0)})
-\frac{3}{2}\Ru_{s=\frac{1}{2}}\zeta(s,\tilde\Delta^{(1)})+\frac{1}{2}\Ru_{s=\frac{1}{2}}\zeta(s,\tilde\Delta^{(2)})\\
&-\Ru_{s=\frac{3}{2}}\zeta(s,\tilde\Delta^{(0)})+\frac{9}{10}\Ru_{s=\frac{5}{2}}\zeta(s,\tilde\Delta^{(0)}),
\end{align*}
and by Propositions \ref{coeff} and \ref{l3.1} 
\begin{align*}
\log T_{\rm sing}(C_l W)=&\frac{3}{80\pi^3}\int_{\b C_lW} dvol_g-\frac{1}{96\pi^3} \int_{\b C_lW} \tilde\tau dvol_g+\frac{1}{128\pi^3} \int_{\b C_lW} |\tilde R|^2 dvol_g\\
&-\frac{1}{32\pi^3} \int_{\b C_lW} |\tilde Ric|^2 dvol_g+\frac{1}{128\pi^3}\int _{\b C_lW} \tilde \tau^2 dvol_g.
\end{align*}

\subsection{A remark on the  general case}
\label{gene}

Assume $m=2p-1$ is odd, $p\geq 1$. Then,  $\log T_{\rm sing} (C_l W)$ depends only  the functions $\Phi_{k,q}(s)$ and $\zeta(s,\tilde \Delta^{(q)})$, by Proposition \ref{ppoo}. It follows from the definition in Section \ref{sb2.2} that the functions $\Phi_{k,q}(s)$ are universal functions that depend only on the decomposition of the spectrum of the Laplace operator on forms on the cone on the spectrum of the Laplace operator on forms on the section. This decomposition is independent from the section, so the functions $\Phi_{k,q}(s)$ do not depend on the particular section (they obviously depend on the dimension). This follows also from Corollaries \ref{c33} and \ref{c44}. Therefore, we can use the functions $\Phi_{k,q}(s)$ calculated when the section is a sphere of odd dimension. It follows that $\log T_{\rm sing}(C_l W)$ is a polynomial in the residues of the functions  $\zeta(s,\tilde \Delta^{(q)})$, with coefficients that are the same as in the case when $W$ is a sphere.  Now, the residues of $\zeta(s,\tilde \Delta^{(q)})$ are polynomials in the coefficients $e_{q,j}$ of the asymptotic expansion of the heat kernel of the Laplacian on forms  on $(W,g)$. In turn, the $e_{q,j}$ are the integrals of some polynomial in  the metric tensor $g$, its inverse and its derivatives. By work of P. Gilkey  \cite{Gil}, the coefficients of these polynomial are universal, namely are the same for any manifold $W$ \cite{Gil} Theorem 1.8.3. 
More precisely, by the above considerations and invariance theory as developed in \cite{Gil} Theorem 4.1.9, it follows that 
\[
\log T_{\rm sing}(C_l W)= \int_{\b C_lW} P(x),
\]
where $P$ belongs to the ring of all invariant polynomials in the derivative of the metric defined for manifolds of dimension $m$, $ \B_{m}(g)$ (see \cite{Gil} Section 2.1.4, and Lemma 2.4.2). Bases for this ring are given in terms of covariant derivatives of the curvature tensor using H. Weyl invariants of the orthogonal groups \cite{Gil} Section 2.4.3. It is possible to prove that these polynomial are universal up to a constant factor depending on the dimension $P=c_m Q$, where $Q$ does not depends on the dimension \cite{Gil} Lemma 4.1.4 and Theorem  4.1.9.

On the other side, by inspection of \cite{BM}, the term $A_{\rm BM,abs}(\b C_l W)$ is also the integral of some universal polynomial in some tensorial quantities constructed from the metric $g$.  Therefore
\[
A_{\rm BM,abs}(\b C_l W)=\int_{\b C_lW} R(x),
\]
$R\in\B_{m}(g)$. Fixing a base for $\B_{m}(g)$, the proof of Theorem \ref{t03} in the general case follows if we are able to prove that $P(x)-R(x)=0$, for $(W,g)$. 

We have some remarks on this point. First, by Theorem \ref{t02}, it follows that $P(x)-R(x)=0$, for $(W,g)=(S^{2p-1}_{\sin\alpha}, g_E)$. Unfortunately, this does not implies the general case. For using the base for $\B_{m}(g)$ in Theorem 4.1.9 of \cite{Gil}, we see that there are variable in $P(x)-R(x)$ that involves derivative of the curvature when $m>3$. Second, by the same argument, the proof of Theorem \ref{t03} in the case $m=5$ is a fundamental indication for the general case: indeed if $m<5$ the proof of Theorem \ref{t03} follows by that of Theorem \ref{t02}, by the previous considerations. Third, the proof of the general case along this line depends on the availability of some further information on the coefficients of the heat asymptotic. A recursion relation should probably be sufficient. However, this seems at present an hard problem (see remarks and references in Section 4.1.7 of \cite{Gil}.

\section{The proof of Theorem \ref{t03}: the general case}
\label{ultima}

Since the argument is very closed to the one described in details in the previous sections, we will just sketch it here. We consider the {\it conical frustum} (or more precisely its external surface) that is the compact connected oriented Riemannian manifold 
\[
 C_{[l_1,l_1]} W=[l_1,l_2]\times W,
 \]
with $0<l_1<l_2$, and with metric
\[
g_1=dx\otimes dx +x^2 g.
\]

We study the analytic torsion of $C_{[l_1,l_2]}$ with relative boundary conditions at $x=l_1$ and absolute boundary condition at $x=l_2$, and we respect to the  trivial representation for the fundamental group. We denote by $\b_{1/2} \tr$, or simply $\b_{1/2}$,  the two boundaries, and by $\log T_{\rm rel~\b_1,abs~\b_2}(\tr )$ the torsion.

\subsection{Spectrum} First, we describe the spectrum of the Laplace operator on forms. The proofs of the next lemmas are analogous to the proofs of Lemmas \ref{l2} and \ref{l3} and will be omitted.

\begin{lem}\label{l2.1} With the notation of Lemma \ref{l2}, 
assuming that $\mu_{q,n}$ is not an integer, all the solutions of the equation $\Delta u=\lambda^2 u$,
with $\lambda\not=0$, are convergent sums of forms of the following twelve types:
\begin{align*}
\psi^{(q)}_{+, 1,n,\lambda} =& x^{\alpha_q} J_{\mu_{q,n}}(\lambda x) \varphi_{{\rm cex},n}^{(q)},\\
\psi^{(q)}_{-, 1,n,\lambda} =& x^{\alpha_q} Y_{\mu_{q,n}}(\lambda x) \varphi_{{\rm cex},n}^{(q)},\\
\psi^{(q)}_{+, 2,n,\lambda} =& x^{\alpha_{q-1}} J_{\mu_{q-1,n}}(\lambda x) \tilde d\varphi_{{\rm cex},n}^{(q-1)} +
\b_x(x^{\alpha_{q-1}}J_{\mu_{q-1,n}}(\lambda x)) dx \wedge \varphi_{{\rm cex},n}^{(q-1)}\\
\psi^{(q)}_{-, 2,n,\lambda} =& x^{\alpha_{q-1}} Y_{\mu_{q-1,n}}(\lambda x) \tilde d\varphi_{{\rm cex},n}^{(q-1)} +
\b_x(x^{\alpha_{q-1}}Y_{\mu_{q-1,n}}(\lambda x)) dx \wedge \varphi_{{\rm cex},n}^{(q-1)}\\
\psi^{(q)}_{+,3,n,\lambda} =& x^{2\alpha_{q-1}+1}\b_x(x^{-\alpha_{q-1}}J_{\mu_{q-1,n}}(\lambda x) )\tilde d\varphi_{{\rm cex},n}^{(q-1)}\\
\nonumber &+ x^{\alpha_{q-1}-1}J_{\mu_{q-1,n}}(\lambda x) dx \wedge \tilde d^{\dag} \tilde d \varphi_{{\rm cex},n}^{(q-1)}\\
\psi^{(q)}_{-,3,n,\lambda} =& x^{2\alpha_{q-1}+1}\b_x(x^{-\alpha_{q-1}}Y_{\mu_{q-1,n}}(\lambda x) )\tilde d\varphi_{{\rm cex},n}^{(q-1)}\\
\nonumber &+ x^{\alpha_{q-1}-1}Y_{\mu_{q-1,n}}(\lambda x) dx \wedge \tilde d^{\dag} \tilde d \varphi_{{\rm cex},n}^{(q-1)}\\
\psi^{(q)}_{+,4,n,\lambda} =& x^{\alpha_{q-2}+1}J_{\mu_{q-2,n}}(\lambda x) dx \wedge \tilde d \varphi_{{\rm cex},n}^{(q-2)}\\
\psi^{(q)}_{-,4,n,\lambda} =& x^{\alpha_{q-2}+1}Y_{\mu_{q-2,n}}(\lambda x) dx \wedge \tilde d \varphi_{{\rm cex},n}^{(q-2)}\\
\psi^{(q)}_{+,E,\lambda} =& x^{\alpha_{q}} J_{|\alpha_{q}|}(\lambda x) \varphi_{{\rm har},n}^{(q)}\\
\psi^{(q)}_{-,E,\lambda} =& x^{\alpha_{q}} Y_{|\alpha_{q}|}(\lambda x) \varphi_{{\rm har},n}^{(q)}\\
\psi^{(q)}_{+,O,\lambda} =& \b_x(x^{\alpha_{q-1}}J_{|\alpha_{q-1}|}(\lambda x)) dx \wedge \varphi_{{\rm har},n}^{(q-1)}\\
\psi^{(q)}_{-,O,\lambda} =& \b_x(x^{\alpha_{q-1}}Y_{|\alpha_{q-1}|}(\lambda x)) dx \wedge \varphi_{{\rm har},n}^{(q-1)}.
\end{align*}

When $\mu_{q,n}$ is an integer the $-$ solutions must be modified including some logarithmic term (see for example \cite{Wat} for a set of linear independent solutions of the Bessel equation).

\end{lem}

Note that the forms of types $1$, $3$ and $E$ are coexact, those of types  $2$, $4$  and $O$ exacts. The operator $d$ sends forms
of types $1$, $3$ and $E$ in forms of types $2$, $4$ and $O$, while $d^{\dag}$ sends forms of types $2$, $4$ and $O$ in forms of types
$1$, $3$ and $E$, respectively. The Hodge operator sends forms of type $1$ in forms of type $4$, $2$ in $3$, and  $E$ in $0$. Define the functions, for $c\neq 0$, 
\begin{align*}
F_{\mu,c} (x)&= J_{\mu}(l_2 x) (c Y_{\mu}(l_1 x) + l_1 x Y'_{\mu}(l_1 x)) - Y_{\mu}(l_2 x) (c J_{\mu}(l_1 x) + l_1 x J'_{\mu}(l_1 x)),\\
\hat F_{\mu,c} (x)&= J_{\mu}(l_1 x) (c Y_{\mu}(l_2 x) + l_2 x Y'_{\mu}(l_2 x)) - Y_{\mu}(l_l x) (c J_{\mu}(l_2 x) + l_2 x J'_{\mu}(l_2 x)) ,
\end{align*} and when $c=0$,
\begin{align*}
F_{\mu,0} (x)&= J_{\mu}(l_2 x) Y'_{\mu}(l_1 x) - Y_{\mu}(l_2 x) J'_{\mu}(l_1 x),\\
\hat F_{\mu,0} (x)&=  J_{\mu}(l_1 x) Y'_{\mu}(l_2 x) - Y_{\mu}(l_l x) J'_{\mu}(l_2 x).
\end{align*}

\begin{lem}\label{l3.1.1}
The positive part of the spectrum of the Laplace operator on forms on $\tr$, with
relative boundary conditions on $\b_1 \tr$ and absolute boundary conditions on $\b_2 \tr$ is:
\begin{align*}
\Sp_+ \Delta_{\rm rel~\b_1,abs~\-b_2}^{(q)} &= \left\{m_{{\rm cex},q,n} : \hat f^{2}_{\mu_{q,n},\alpha_q,k}\right\}_{n,k=1}^{\infty}
\cup
\left\{m_{{\rm cex},q-1,n} : \hat f^{2}_{\mu_{q-1,n},\alpha_{q-1},k}\right\}_{n,k=1}^{\infty} \\
&\cup \left\{m_{{\rm cex},q-1,n} : f^{2}_{\mu_{q-1,n},-\alpha_{q-1},k}\right\}_{n,k=1}^{\infty} \cup \left\{m _{{\rm cex},q-2,n} :
f^{2}_{\mu_{q-2,n},-\alpha_{q-2},k}\right\}_{n,k=1}^{\infty} \\
&\cup \left\{m_{{\rm har},q,0}:\hat f^2_{|\alpha_q|,\alpha_q,k}\right\}_{k=1}^{\infty} \cup \left\{ m_{{\rm har},q-1,0}:\hat
f^2_{|\alpha_{q-1}|,\alpha_{q-1},k}\right\}_{k=1}^{\infty}.
\end{align*}

With absolute boundary conditions on $\b_1 \tr$ and relative boundary conditions on $\b_2 \tr$ is:
\begin{align*}
\Sp_+ \Delta^{(q)}_{\rm abs~\b_1,rel ~\b_2} &= \left\{m _{{\rm cex},q,n} : f^{-2s}_{\mu_{q,n},\alpha_q,k}\right\}_{n,k=1}^{\infty} \cup
\left\{m _{{\rm cex},q-1,n} :f^{-2s}_{\mu_{q-1,n},\alpha_{q-1},k}\right\}_{n,k=1}^{\infty} \\
&\cup \left\{ m_{{\rm cex},q-1,n} : \hat f^{-2s}_{\mu_{q-1,n},-\alpha_{q-1},k}\right\}_{n,k=1}^{\infty} \cup
\left\{m _{{\rm cex},q-2,n} :
\hat f^{-2s}_{\mu_{q-1,n},-\alpha_{q-2},k}\right\}_{n,k=1}^{\infty} \\
&\cup \left\{m_{{\rm har},q}:f_{|\alpha_q|,\alpha_q,k}\right\}_{k=1}^{\infty} \cup \left\{m_{{\rm har},q-1}:
f_{|\alpha_{q-1}|,\alpha_{q-1},k}\right\}_{k=1}^{\infty},
\end{align*}
where  the $f_{\mu,c,k}$ are the zeros of the function $F_{\mu,c}(x)$,  the $\hat f_{\mu,c,k}$ are the zeros of
the function $\hat F_{\mu,c}(x) $,  $c\in \R$, $\alpha_q$ and $\mu_{q,n}$ are defined in Lemma \ref{l2}.
\end{lem}

\subsection{Torsion zeta function} 

We  define the {\it torsion zeta function} as in Section \ref{forms} by (for $\Re(s)>\frac{m+1}{2}$)
\[
t_{\rm abs,rel}(s)=\frac{1}{2}\sum_{q=1}^{m+1} (-1)^q q \zeta(s,\Delta_{\rm abs,rel}^{(q)}).
\]

By a proof similar to the one of  Theorem \ref{Poinc2}, we obtain the expected duality:
\[
\log T_{\rm abs~\b_1,rel~\b_2}(\tr)=(-1)^{m}\log T_{\rm rel~\b_1 ,abs~\b_2}(\tr).
\]

We proceed assuming ${\rm dim} W=2p-1$ odd, and assuming relative boundary condition on $\b_1 \tr$ and absolute boundary condition on $\b_2 $; for notational convenience, we will omit the ${\it abs,rel}$ subscript. We define the functions
\begin{align*}
\hat F_{c} (x)&= J_{c}(l_2 x) Y_{c-1}(l_1 x) - Y_{c}(l_2 x) J_{c-1}(l_1 x),\\
F_{c} (x)&=  J_{c}(l_1 x) Y_{c-1}(l_2 x) - Y_{c}(l_l x) J_{c-1}(l_2 x).
\end{align*}
Note that, with these definitions $\hat F_0(x) = F_1(x)$ and $F_0(x) = \hat F_1(x)$ (remember that $Y_{-n}(x) =(-1)^n Y_n(x)$ and $J_{-n}(x) =(-1)^n J_n(x)$). The proof of the following lemma is analogous to the proof of Lemma \ref{tps}. The main step is to prove that  $\hat f_{|\alpha_{q}|,\alpha_q,k} =  f_{-\alpha_{q-1},k}$, that  $\hat f_{|\alpha_{q}|,\alpha_q,k} = \hat f_{\alpha_{q},k}$, when $p-1<q\leq2p-1$, and that $\hat f_{0,0,k} =   f_{1,k}$, where the $f_{c,k},hat f_{c,k}$ are the zeros of the functions $F_c, \hat F_c$, respectively.

\begin{lem} \label{tps1}
\begin{align*}
t(s)=& \frac{1}{2} \sum^{p-2}_{q=0} (-1)^q \sum^{\infty}_{n,k=1} m_{{\rm cex},q,n}\left(f^{-2s}_{\mu_{q,n},\alpha_q,k} +f^{-2s}_{\mu_{q,n},-\alpha_q,k} -\hat f^{-2s}_{\mu_{q,n},\alpha_q,k} - \hat
f^{-2s}_{\mu_{q,n},-\alpha_q,k}\right)\\
&+ (-1)^{p-1}\frac{1}{2} \sum^{\infty}_{n,k=1} m_{{\rm cex},p-1,n} \left( f^{-2s}_{\mu_{p-1,n},0,k} -
\hat f_{\mu_{p-1,n},0,k}^{-2s} \right)\\
&- \frac{1}{2}  \sum_{q=0}^{p-1} (-1)^{q} {\rm rk} \H_q(W;\Q) \sum_{k=1}^{\infty} \left(f^{-2s}_{-\alpha_{q-1},k} - \hat
f^{-2s}_{-\alpha_{q-1},k}\right).
\end{align*}
\end{lem}

We set
\beq\label{z1.1}
\begin{aligned}
Z_{q,\pm}(s)&=\sum^{\infty}_{n,k=1} m_{{\rm cex},q,n} f^{-2s}_{\mu_{q,n},\pm\alpha_q,k},&
\hat Z_{q,\pm}(s)&=\sum^{\infty}_{n,k=1} m_{{\rm cex},q,n}\hat f_{\mu_{q,n},\pm\alpha_q,k}^{-2s},\\
Z_{p-1}(s)&=\sum^{\infty}_{n,k=1} m_{{\rm cex},p-1,n} f^{-2s}_{\mu_{p-1,n},0,k},&
\hat Z_{p-1,\pm}(s)&=\sum^{\infty}_{n,k=1} m_{{\rm cex},p-1,n}\hat f_{\mu_{p-1,n},0,k}^{-2s},\\
z_q(s)&=\sum_{k=1}^{\infty} \left(f^{-2s}_{-\alpha_{q-1},k} - \hat f^{-2s}_{-\alpha_{q-1},k}\right),
\end{aligned}
\eeq
for $0\leq q \leq p-1$, and
\beq\label{tq.1}
\begin{aligned}
t_{p-1}(s)&=Z_{p-1}(s)-\hat Z_{p-1}(s),&\\
t_q(s)&=Z_{q,+}(s)+Z_{q,-}(s)-\hat Z_{q,+}(s)-\hat Z_{q,-}(s),&0\leq q\leq p-2.\\
\end{aligned}
\eeq

Then,
\begin{align*}
t(s)=& \frac{1}{2} \sum^{p-2}_{q=0} (-1)^q \left(Z_{q,+}(s)+Z_{q,-}(s)-\hat Z_{q,+}(s)-\hat Z_{q,-}(s)\right)+(-1)^{p-1}\frac{1}{2}\left(Z_{p-1}(s)-\hat Z_{p-1}(s)\right)\\
&- \frac{1}{2} \sum_{q=0}^{p-1} (-1)^{q} {\rm rk}\H_q(W;\Q) z_q(s)\\
=&\frac{1}{2} \sum^{p-1}_{q=0} (-1)^q t_q(s)- \frac{1}{2} \sum_{q=0}^{p-1} (-1)^{q} {\rm rk}\H_q(\b C_lW;\Q)z_q(s),
\end{align*}
and
\beq\label{ttt.1}
\begin{aligned}
\log T_{\rm rel~\b_1,abs~\b_2}(\tr )=t'(0)=& \frac{1}{2} \sum^{p-1}_{q=0} (-1)^q t'_q(0)- \frac{1}{2} \sum_{q=0}^{p-1} (-1)^{q} {\rm rk}\H_q(\b C_lW;\Q)z'_q(0).
\end{aligned}
\eeq

\subsection{Expansions of the logarithmic Gamma functions} \label{ora}

We study the zeta functions $Z_{q,\pm}$, $\hat Z_{q,\pm}$,  by the method of Section \ref{sb2}. The double series associated to these zeta functions, as defined in equation (\ref{z1.1}), are denoted 
by $S_{\pm \alpha_q}$, $\hat S_{\pm\alpha_q}$. We show that all these double sequences are spectrally decomposable on the sequence $U_q$, defined at the beginning of Section \ref{s4}. We verify all requirements precisely as in Sections \ref{s3} and \ref{s4}. First, we need suitable representation for the associated logarithmic Gamma functions. Proceeding as in Section \ref{s4}, consider for example the function
\[
F_{\mu,c} (z)=  J_{\mu}(l_2 z) (c Y_{\mu}(l_1 z) + l_1 z Y'_{\mu}(l_1 z)) - Y_{\mu}(l_2 z) (c J_{\mu}(l_1 z) + l_1 z J'_{\mu}(l_1 z)).
\]
Recalling the series definition of the Bessel function \cite{GZ}pg. 910, near $z=0$,
\[
F_{\mu, c} (z) = \frac{1}{\pi}\left(\left(\frac{l_2^\mu}{l_1^\mu}+\frac{l_1^\mu}{l_2^\mu}\right)-\frac{c}{\mu}\left(\frac{l_2^\mu}{l_1^\mu} - \frac{l_1^\mu}{l_2^\mu}\right)\right)
\]

Thus $F_{\mu,c}(z)$ is an even function of $z$, and  we obtain the product representation
\[
F_{\mu,c}(z)  = \frac{1}{\pi}\left(\left(\frac{l_2^\mu}{l_1^\mu}+\frac{l_1^\mu}{l_2^\mu}\right)-\frac{c}{\mu}\left(\frac{l_2^\mu}{l_1^\mu} - \frac{l_1^\mu}{l_2^\mu}\right)\right)\prod^{+\infty}_{k=1} \left(1 -\frac{z^2}{f^2_{\mu,c,k}}\right).
\]

Recalling that
\[
Y_\mu(z) = \frac{\cos\mu\pi}{\sin\mu\pi} J_{\mu}(z) - \frac{1}{\sin\mu\pi} J_{-\mu}(z), \qquad I_{-\mu}(z) = \frac{2}{\pi} \sin\mu\pi K_\mu(z) + I_\mu (z),
\] 
and that (when $-\pi<\arg(z)\leq\frac{\pi}{2}$) $J_\mu(iz)= \e^{\frac{\pi  }{2}i \mu} I_\mu(z)$, and
$J'_\mu(iz)=\e^{\frac{\pi}{2}i\mu}\e^{-\frac{\pi}{2}i} I'_\mu(z)$, we obtain
\begin{align*}
Y_\mu(iz)&= \left(\frac{\cos\mu\pi}{\sin\mu\pi}\e^{\frac{\pi}{2}i\mu} + \frac{\e^{-\frac{\pi}{2}i\mu}}{\sin\mu\pi}\right)I_\mu(z) - \frac{2}{\pi} \e^{-\frac{\pi}{2}i\mu} K_\mu(z),\\
Y'_\mu(iz)&= e^{-\frac{\pi}{2}i}\left(\frac{\cos\mu\pi}{\sin\mu\pi}\e^{\frac{\pi}{2}i\mu} + \frac{\e^{-\frac{\pi}{2}i\mu}}{\sin\mu\pi}\right)I'_\mu(z) - \frac{2}{\pi}e^{-\frac{\pi}{2}i} \e^{-\frac{\pi}{2}i\mu} K'_\mu(z).
\end{align*} 

So
\[
F_{\mu,c}(iz) = \frac{2}{\pi} \left(-K_{\mu}(l_2 z)(cI_{\mu}(l_1 z) + l_1z I'_\mu(l_1 z)) + I_\mu(l_2z)(cK_\mu(l_1 z) + l_1 z K'_\mu(l_1 z))\right),
\] and if we define (for $-\pi < \arg(z)\leq \frac{\pi}{2}$) $G_{\mu,c}(z) =  i^2F_{\mu,c}(iz)$, 
\[
G_{\mu,c}(z) = \frac{1}{\pi}\left(\left(\frac{l_2^\mu}{l_1^\mu}+\frac{l_1^\mu}{l_2^\mu}\right)-\frac{c}{\mu}\left(\frac{l_2^\mu}{l_1^\mu} - \frac{l_1^\mu}{l_2^\mu}\right)\right)\prod^{+\infty}_{k=1} \left(1 +\frac{z^2}{f^2_{\mu,c,k}}\right).
\]

Proceeding in the similar way

\begin{align*}
\hat F_{\mu,c}(iz) &= \frac{2}{\pi} \left(K_{\mu}(l_1 z)(cI_{\mu}(l_2 z) + l_2z I'_\mu(l_2 z)) - I_\mu(l_1z)(cK_\mu(l_2 z) + l_2 z K'_\mu(l_2 z))\right),\\
\hat G_{\mu,c}(z) &=  \hat F_{\mu,c}(iz)= \frac{1}{\pi}\left(\left(\frac{l_2^\mu}{l_1^\mu}+\frac{l_1^\mu}{l_2^\mu}\right)+\frac{c}{\mu}\left(\frac{l_2^\mu}{l_1^\mu} - \frac{l_1^\mu}{l_2^\mu}\right)\right)\prod^{+\infty}_{k=1} \left(1 +\frac{z^2}{\hat f^2_{\mu,c,k}}\right);\\
F_{\mu,0}(iz) &= \frac{2}{\pi} \left(-K_{\mu}(l_2 z)I'_\mu(l_1 z) + I_\mu(l_2z) K'_\mu(l_1 z)\right),\\
G_{\mu,0}(z) &=  i^2F_{\mu,0}(iz)= \frac{1}{l_1 z \pi}\left(\frac{l_2^\mu}{l_1^\mu}+\frac{l_1^\mu}{l_2^\mu}\right)\prod^{+\infty}_{k=1} \left(1 +\frac{z^2}{f^2_{\mu,c,k}}\right);\\
\hat F_{\mu,0}(iz) &= \frac{2}{\pi} \left(K_{\mu}(l_1 z) I'_\mu(l_2 z) - I_\mu(l_1z)K'_\mu(l_2 z))\right),\\
\hat G_{\mu,0}(z) &=  \hat F_{\mu,0}(iz)
 = \frac{1}{l_2 z \pi}\left(\frac{l_2^\mu}{l_1^\mu}+\frac{l_1^\mu}{l_2^\mu}\right)\prod^{+\infty}_{k=1} \left(1 +\frac{z^2}{\hat f^2_{\mu,0,k}}\right).
\end{align*}

These give the following representations for the logarithmic Gamma functions
with $z=\sqrt{-\lambda}$, 
\begin{align*}
\log \Gamma(-\lambda,S_{{\pm\alpha_q}})&=-\log\prod_{k=1}^\infty \left(1+\frac{(-\lambda)}{f_{\mu_{q,n},\pm\alpha_q,k}^2}\right)\\
&=-\log  G_{\mu_{q,n},\pm\alpha_q}(\sqrt{-\lambda})  + \log\frac{1}{ \pi}+ \log \left(\left(\frac{l_2^{\mu_{q,n}}}{l_1^{\mu_{q,n}}}+\frac{l_1^{\mu_{q,n}}}{l_2^{\mu_{q,n}}}\right)\mp\frac{\alpha_q}{\mu_{q,n}}\left(\frac{l_2^{\mu_{q,n}}}{l_1^{\mu_{q,n}}} - \frac{l_1^{\mu_{q,n}}}{l_2^{\mu_{q,n}}}\right)\right);\\
\log \Gamma(-\lambda,\hat S_{{\pm\alpha_q}})&=-\log\prod_{k=1}^\infty \left(1+\frac{(-\lambda)}{\hat f_{\mu_{q,n},\pm\alpha_q,k}^2}\right)\\
&=-\log \hat G_{\mu_{n,q},\pm\alpha_q}(\sqrt{-\lambda})  + \log\frac{1}{ \pi}+ \log \left(\left(\frac{l_2^{\mu_{q,n}}}{l_1^{\mu_{q,n}}}+\frac{l_1^{\mu_{q,n}}}{l_2^{\mu_{q,n}}}\right)\pm\frac{\alpha_q}{\mu_{q,n}}\left(\frac{l_2^{\mu_{q,n}}}{l_1^{\mu_{q,n}}} - \frac{l_1^{\mu_{q,n}}}{l_2^{\mu_{q,n}}}\right)\right)
\end{align*}

\begin{align*}
\log \Gamma(-\lambda,S_{{0}})&=-\log\prod_{k=1}^\infty \left(1+\frac{(-\lambda)}{f_{\mu_{p-1,n},0,k}^2}\right)\\
&=-\log  G_{\mu_{n,p-1},0}(\sqrt{-\lambda}) - \frac{1}{2}\log -\lambda -\log l_1 + \log\frac{1}{ \pi}+ \log \left(\frac{l_2^{\mu_{q,n}}}{l_1^{\mu_{q,n}}}+\frac{l_1^{\mu_{q,n}}}{l_2^{\mu_{q,n}}}\right)\\
\log \Gamma(-\lambda,\hat S_{{0}})&=-\log\prod_{k=1}^\infty \left(1+\frac{(-\lambda)}{\hat f_{\mu_{p-1,n},0,k}^2}\right)\\
&=-\log \hat G_{\mu_{n,p-1},0}(\sqrt{-\lambda})  - \frac{1}{2}\log -\lambda -\log l_2+ \log\frac{1}{ \pi}+ \log \left(\frac{l_2^{\mu_{p-1,n}}}{l_1^{\mu_{p-1,n}}}+\frac{l_1^{\mu_{p-1,n}}}{l_2^{\mu_{p-1,n}}}\right)
\end{align*}

These representations and uniform asymptotic expansions of Bessel functions and their derivative (see the proof of Lemma \ref{s3.l2} for the functions $I_\nu$ and \cite{Olv} pg. 376 for the functions $K_\nu$) will give the expansion required in equation (\ref{exp}) of Definition \ref{spdec}. Let see one case in some details. We have
\begin{align*}
&\log \Gamma(-\lambda,S_{n,{\pm\alpha_q}}/\mu_{q,n}^2) = \\
&-\log G_{\mu_{n,q},\pm\alpha_q}(\mu_{q,n}\sqrt{-\lambda})  + \log\frac{1}{ \pi}+ \log \left(\left(\frac{l_2^{\mu_{q,n}}}{l_1^{\mu_{q,n}}}+\frac{l_1^{\mu_{q,n}}}{l_2^{\mu_{q,n}}}\right)\mp\frac{\alpha_q}{\mu_{q,n}}\left(\frac{l_2^{\mu_{q,n}}}{l_1^{\mu_{q,n}}} - \frac{l_1^{\mu_{q,n}}}{l_2^{\mu_{q,n}}}\right)\right).
\end{align*}

Using the cited expansions we obtain
\begin{align*}
&\log G_{\mu,c}(\mu z) = \log\frac{1}{\pi} + \mu\left(\sqrt{1+l_2^2 z^2} -\sqrt{1+l_1^2 z^2} \right) + \mu\log\frac{l_2(1+\sqrt{1+l_1^2z^2})}{l_1(1+\sqrt{1+l_2^2z^2})}+ \frac{1}{4} \log\frac{(1+l_1^2 z^2)}{(1+l_2^2 z^2)} \\ 
&+ \log\left(1 + \sum_{j=1}^{2p-1}\frac{1}{\mu^j} \left(U_j(l_2z)  +(-1)^j W_{c,j}(l_1z) +\sum_{k=1}^{j-1} (-1)^{j-k}U_k(l_2z)  W_{c,j-k}(l_1z)\right) + O(\mu^{-2p})\right),
\end{align*} where 
\[
W_{c,j}(z) = V_j( z) - \frac{c}{\sqrt{1 + z^2}} U_{j-1}( z).
\] 

Thus
\begin{align*}
\log  G_{\mu_{q,n},\pm\alpha_q}(\mu_{q,n} \sqrt{-\lambda}) &= \mu_{q,n}\left(\sqrt{1-l_2^2 \lambda} -\sqrt{1-l_1^2 \lambda} \right) + \mu_{q,n}\log\frac{l_2(1+\sqrt{1 - l_1^2\lambda})}{l_1(1+\sqrt{1-l_2^2\lambda})} \\ 
&+\log\frac{1}{\pi} + \frac{1}{4} \log\frac{(1-l_1^2 \lambda)}{(1-l_2^2 \lambda)}+ \sum_{j=1}^{2p-1} \frac{l_{j,\pm\alpha_q}(\lambda)}{\mu_{q,n}^j} + O(\mu_{q,n}^{-2p}),
\end{align*} with
\begin{align*}
a_{0,\pm\alpha_q}(\lambda) &=1,\\
a_{j,\pm\alpha_q}(\lambda) &=  U_j(l_2\sqrt{-\lambda})+(-1)^j W_{\pm\alpha_q,j}(l_1\sqrt{-\lambda}) +\sum_{k=1}^{j-1} U_k(l_2\sqrt{-\lambda})(-1)^{j-k}W_{\pm\alpha_q,j-k}(l_1\sqrt{-\lambda}),\\
l_{1,\pm\alpha_q}(\lambda) &= a_{1,\pm\alpha_q}(\lambda),\\
l_{j,\pm\alpha_q}(\lambda) &= a_{j,\pm\alpha_q}(\lambda) - \sum_{k=1}^{j-1} \frac{j-k}{j} a_{k,\pm\alpha_q}(\lambda) l_{j-k,\pm\alpha_q}(\lambda).
\end{align*}

Substituting in the $\log \Gamma(-\lambda,S_{n,{\pm\alpha_q}}/\mu_{q,n}^2)$, we have
\begin{align*}
&\log \Gamma(-\lambda,S_{n,{\pm\alpha_q}}/\mu_{q,n}^2) = -\mu_{q,n}\left(\sqrt{1-l_2^2 \lambda} -\sqrt{1-l_1^2 \lambda} \right) - \mu_{q,n}\log\frac{l_2(1+\sqrt{1 - l_1^2\lambda})}{l_1(1+\sqrt{1-l_2^2\lambda})} \\ 
&- \frac{1}{4} \log\frac{(1-l_1^2 \lambda)}{(1-l_2^2 \lambda)} - \sum_{j=1}^{2p-1} \frac{ l_{j,\pm\alpha_q}(\lambda)}{\mu_{q,n}^j} + \log \left(\left(\frac{l_2^{\mu_{q,n}}}{l_1^{\mu_{q,n}}}+\frac{l_1^{\mu_{q,n}}}{l_2^{\mu_{q,n}}}\right)\mp\frac{\alpha_q}{\mu_{q,n}}\left(\frac{l_2^{\mu_{q,n}}}{l_1^{\mu_{q,n}}} - \frac{l_1^{\mu_{q,n}}}{l_2^{\mu_{q,n}}}\right)\right)+ O(\mu_{q,n}^{-2p}).
\end{align*}

Proceeding in a similar way we obtain

\begin{align*}
&\log \Gamma(-\lambda,\hat S_{n,{\pm\alpha_q}}/\mu_{q,n}^2) = -\mu_{q,n}\left(\sqrt{1-l_2^2 \lambda} -\sqrt{1-l_1^2 \lambda} \right) - \mu_{q,n}\log\frac{l_2(1+\sqrt{1 - l_1^2\lambda})}{l_1(1+\sqrt{1-l_2^2\lambda})} \\ 
&- \frac{1}{4} \log\frac{(1-l_1^2 \lambda)}{(1-l_2^2 \lambda)} - \sum_{j=1}^{2p-1} \frac{\hat l_{j,\pm\alpha_q}(\lambda)}{\mu_{q,n}^j} + \log \left(\left(\frac{l_2^{\mu_{q,n}}}{l_1^{\mu_{q,n}}}+\frac{l_1^{\mu_{q,n}}}{l_2^{\mu_{q,n}}}\right)\pm\frac{\alpha_q}{\mu_{q,n}}\left(\frac{l_2^{\mu_{q,n}}}{l_1^{\mu_{q,n}}} - \frac{l_1^{\mu_{q,n}}}{l_2^{\mu_{q,n}}}\right)\right)+ O(\mu_{q,n}^{-2p}),
\end{align*}
with 
\begin{align*}
\hat a_{0,\pm\alpha_q}(\lambda) &=1,\\
\hat a_{j,\pm\alpha_q}(\lambda) &= \hat W_{\pm\alpha_q,j}(l_2\sqrt{-\lambda}) +(-1)^j U_j(l_1\sqrt{-\lambda}) +\sum_{k=1}^{j-1} (-1)^kU_k(l_1\sqrt{-\lambda})\hat W_{\pm\alpha_q,j-k}(l_2\sqrt{-\lambda}),\\
\hat l_{1,\pm\alpha_q}(\lambda) &= \hat a_{1,\pm\alpha_q}(\lambda),\\
\hat l_{j,\pm\alpha_q}(\lambda) &= \hat a_{j,\pm\alpha_q}(\lambda) - \sum_{k=1}^{j-1} \frac{j-k}{j} \hat a_{k,\pm\alpha_q}(\lambda) \hat l_{j-k,\pm\alpha_q}(\lambda).
\end{align*}

\begin{align*}
&\log \Gamma(-\lambda,\hat S_{n,{0}}/\mu_{p-1,n}^2) = -\mu_{p-1,n}\left(\sqrt{1-l_2^2 \lambda} -\sqrt{1-l_1^2 \lambda} \right) - \mu_{p-1,n}\log\frac{l_2(1+\sqrt{1 - l_1^2\lambda})}{l_1(1+\sqrt{1-l_2^2\lambda})} \\ 
&- \frac{1}{4} \log\frac{(1-l_1^2 \lambda)}{(1-l_2^2 \lambda)} - \sum_{j=1}^{2p-1} \frac{\hat l_{j,0}(\lambda)}{\mu_{p-1,n}^j} + \log \left(\frac{l_2^{\mu_{p-1,n}}}{l_1^{\mu_{p-1,n}}}+\frac{l_1^{\mu_{p-1,n}}}{l_2^{\mu_{p-1,n}}}\right)+ O(\mu_{p-1,n}^{-2p}),
\end{align*}
with
\begin{align*}
\hat a_{0,0}(\lambda) &=1,\\
\hat a_{j,0}(\lambda) &=V_{j}(l_2\sqrt{-\lambda}) +(-1)^j U_j(l_1\sqrt{-\lambda}) +\sum_{k=1}^{j-1} (-1)^kU_k(l_1\sqrt{-\lambda})V_{j-k}(l_2\sqrt{-\lambda}),\\
\hat l_{1,0}(\lambda) &= \hat a_{1,0}(\lambda),\\
\hat l_{j,0}(\lambda) &= \hat a_{j,0}(\lambda) - \sum_{k=1}^{j-1} \frac{j-k}{j} \hat a_{k,0}(\lambda) \hat l_{j-k,0}(\lambda).
\end{align*}

\begin{align*}
&\log \Gamma(-\lambda,S_{n,{0}}/\mu_{q,n}^2) = -\mu_{p-1,n}\left(\sqrt{1-l_2^2 \lambda} -\sqrt{1-l_1^2 \lambda} \right) - \mu_{p-1,n}\log\frac{l_2(1+\sqrt{1 - l_1^2\lambda})}{l_1(1+\sqrt{1-l_2^2\lambda})} \\ 
&- \frac{1}{4} \log\frac{(1-l_1^2 \lambda)}{(1-l_2^2 \lambda)} - \sum_{j=1}^{2p-1} \frac{ l_{j,0}(\lambda)}{\mu_{p-1,n}^j} + \log \left(\frac{l_2^{\mu_{p-1,n}}}{l_1^{\mu_{p-1,n}}}+\frac{l_1^{\mu_{p-1,n}}}{l_2^{\mu_{p-1,n}}}\right)+ O(\mu_{p-1,n}^{-2p}),
\end{align*}
with
\begin{align*}
a_{0,0}(\lambda) &=1,\\
a_{j,0}(\lambda) &=  U_j(l_2\sqrt{-\lambda})+(-1)^j V_{j}(l_1\sqrt{-\lambda}) +\sum_{k=1}^{j-1} U_k(l_2\sqrt{-\lambda})(-1)^{j-k}V_{j-k}(l_1\sqrt{-\lambda}),\\
l_{1,0}(\lambda) &= a_{1,0}(\lambda),\\
l_{j,0}(\lambda) &= a_{j,0}(\lambda) - \sum_{k=1}^{j-1} \frac{j-k}{j} a_{k,0}(\lambda) l_{j-k}(\lambda).
\end{align*}

We conclude this section with the expansions for large $\lambda$, accordingly to equation (\ref{form}). Using classical expansions of Bessel functions $I_nu$ and $K_\nu$ and their derivative for large argument, we obtain the expansions of the functions $G$ and $\hat G$, and then those for the Gamma functions:
\begin{align*}
&\log \Gamma(-\lambda,S_{n,{\pm\alpha_q}}/\mu_{q,n}^2) \sim \\
&\sim-\mu_{q,n} (l_2 - l_1)\sqrt{-\lambda} - \frac{1}{2}\log\frac{l_1}{l_2} + \log \left(\left(\frac{l_2^{\mu_{q,n}}}{l_1^{\mu_{q,n}}}+\frac{l_1^{\mu_{q,n}}}{l_2^{\mu_{q,n}}}\right)\mp\frac{\alpha_q}{\mu_{q,n}}\left(\frac{l_2^{\mu_{q,n}}}{l_1^{\mu_{q,n}}} - \frac{l_1^{\mu_{q,n}}}{l_2^{\mu_{q,n}}}\right)\right) + O\left(\frac{1}{\sqrt{-\lambda}}\right),\\
&\log \Gamma(-\lambda,\hat S_{n,{\pm\alpha_q}}/\mu_{q,n}^2) \sim \\
&\sim-\mu_{q,n} (l_2 - l_1)\sqrt{-\lambda} - \frac{1}{2}\log\frac{l_2}{l_1} + \log \left(\left(\frac{l_2^{\mu_{q,n}}}{l_1^{\mu_{q,n}}}+\frac{l_1^{\mu_{q,n}}}{l_2^{\mu_{q,n}}}\right)\pm\frac{\alpha_q}{\mu_{q,n}}\left(\frac{l_2^{\mu_{q,n}}}{l_1^{\mu_{q,n}}} - \frac{l_1^{\mu_{q,n}}}{l_2^{\mu_{q,n}}}\right)\right) + O\left(\frac{1}{\sqrt{-\lambda}}\right),\\
&\log \Gamma(-\lambda,S_{n,{0}}/\mu_{p-1,n}^2) \sim -\mu_{p-1,n} (l_2 - l_1)\sqrt{-\lambda} +\frac{1}{2}\log\frac{l_2}{l_1} + \log \left(\frac{l_2^{\mu_{p-1,n}}}{l_1^{\mu_{p-1,n}}}+\frac{l_1^{\mu_{p-1,n}}}{l_2^{\mu_{p-1,n}}}\right) + O\left(\frac{1}{\sqrt{-\lambda}}\right),\\
&\log \Gamma(-\lambda,\hat S_{n,{0}}/\mu_{p-1,n}^2) 
\sim-\mu_{p-1,n} (l_2 - l_1)\sqrt{-\lambda} +\frac{1}{2}\log\frac{l_1}{l_2}  + \log \left(\frac{l_2^{\mu_{p-1,n}}}{l_1^{\mu_{p-1,n}}}+\frac{l_1^{\mu_{p-1,n}}}{l_2^{\mu_{p-1,n}}}\right)+ O\left(\frac{1}{\sqrt{-\lambda}}\right).
\end{align*}

\subsection{The function $t_q(s)$}

By definition in equation (\ref{tq.1}), we need to consider the difference between $\log \Gamma(-\lambda, S_{n,\pm\alpha_q}/\mu_{q,n})$ and $\log \Gamma(-\lambda, \hat S_{n,\pm\alpha_q}/\mu_{q,n})$. The expansions given in the previous subsection give  expansion for large $\mu$ 
\begin{align*}
&\log \Gamma(-\lambda, S_{n,\alpha_q}/\mu_{q,n})+\log \Gamma(-\lambda, S_{n,\alpha_q}/\mu_{q,n}) - \log \Gamma(-\lambda, \hat S_{n,\alpha_q}/\mu_{q,n}) - 
\log \Gamma(-\lambda, \hat S_{n,-\alpha_q}/\mu_{q,n})=\\
&=\log\frac{(1-\lambda l_2^2)}{(1-\lambda l_1^2)} +\sum_{j=1}^{2p-1}\frac{1}{\mu_{q.n}^j} (\hat l_{j,\alpha_q}(\lambda)+\hat l_{j,-\alpha_q}(\lambda)-l_{j,\alpha_q}(\lambda)-l_{j,-\alpha_q}(\lambda)) + O(\mu_{q,n}^{-2p}),
\end{align*}
and for large $\lambda$
\begin{align*}
&\log \Gamma(-\lambda, S_{n,\alpha_q}/\mu_{q,n})+\log \Gamma(-\lambda, S_{n,\alpha_q}/\mu_{q,n}) - \log \Gamma(-\lambda, \hat S_{n,\alpha_q}/\mu_{q,n}) - 
\log \Gamma(-\lambda, \hat S_{n,-\alpha_q}/\mu_{q,n})=\\
&=2\log\frac{l_2}{l_1 } + O\left(\frac{1}{\sqrt{-\lambda}}\right).
\end{align*}

Proceeding as in the proof of Lemma \ref{s4.l4}, we obtain
\begin{align*}
a_{0,0,q,n}&=2\log \frac{l_2}{l_1},\\
a_{0,1,q,n}&=0,\\
b_{2j-1,0,0,q}&=0,\hspace{50pt}b_{2j-1,0,1,q}=0,
\end{align*}
and hence
\begin{align*}
A_{0,0,q}(s)&=2\log\frac{l_2}{l_1}\sum_{n=1}^\infty \frac{m_{q,n}}{\mu_{q,n}^{2 s}} = 2\log\frac{l_2}{l_1}\sum_{j=0}^{\infty}\binom{-s}{j}\alpha_q^2j \zeta_{{\rm ccl}}(s+j,\tilde \Delta^{(q)}),\\
A_{0,1,q}(s)&=0.
\end{align*} 

This gives
\begin{align*}
A_{0,0,q}(0)&=2\log\frac{l_2}{l_1}\zeta_{{\rm ccl},q}(0,\tilde \Delta^{(q)})=2(-1)^q\log\frac{l_2}{l_1} \sum_{k=0}^{q} (-1)^k {\rm rk}\H^k(W,\Q),
\end{align*}
and 
\[
t'_{q,{\rm reg}}(0) = 2(-1)^{q+1}\log\frac{l_2}{l_1} \sum_{k=0}^{q} (-1)^k {\rm rk}\H^k(W,\Q)
\]


Similarly, we consider the difference  of $\log \Gamma(-\lambda, S_{n,0}/\mu_{p-1,n})$ and $\log \Gamma(-\lambda, \hat S_{n,0}/\mu_{p-1,n})$ for the function $t_{p-1}$, and we obtain
\begin{align*}
a_{0,0,n,p-1}&=\log \frac{l_2}{l_1},\\
a_{0,1,n,p-1}&=0,\\
b_{2j-1,0,0,p-1}&=0,\hspace{50pt}b_{2j-1,0,1,p-1}=0,
\end{align*}
\begin{align*}
A_{0,0,p-1}(s)&=\log\frac{l_2}{l_1}\sum_{n=1}^\infty \frac{m_{p-1,n}}{\mu_{p-1,n}^{2 s}} = \log\frac{l_2}{l_1}\zeta_{{\rm ccl},p-1}(s,\tilde \Delta^{(q)}),\\
A_{0,1}(s)&=0,
\end{align*} 
\begin{align*}
A_{0,0,p-1}(0)&=\log\frac{l_2}{l_1}\zeta_{{\rm ccl},q}(0,\tilde \Delta^{(p-1)})=(-1)^{p-1}\log\frac{l_2}{l_1} \sum_{k=0}^{p-1} (-1)^k {\rm rk}\H^k(W,\Q),
\end{align*}
and
\[
t'_{p-1,{\rm reg}}(0) = (-1)^{p}\log\frac{l_2}{l_1} \sum_{k=0}^{p-1} (-1)^k {\rm rk}\H^k(W,\Q)
\]


\subsection{The regular term of the torsion }

We use equation (\ref{ttt.1}). First, note that as in Section \ref{s5} there is no singular contribution by the functions $z_q(s)$. Using equation \ref{p00}, and recalling that $-\alpha_{q-1} = -(q-1-p+1) = p-q$, we compute as in Lemma \ref{lp11}
\[
z'_q(0) =  \log\frac{l_2}{l_1} -2(p-q)\log\frac{l_2}{l_1}. 
\]

Therefore, substitution in equation (\ref{ttt.1}) gives
\[
\log T_{\rm rel~\b_1,abs~\b_2, reg}(\tr )=t_{reg}'(0)=0.
\]

\subsection{The singular term of the torsion}

We show that the singular part of the torsion is twice the singular part of the torsion on the cone, namely that
\beq\label{sos}
\log T_{\rm rel~\b_1,abs~\b_2, sing}(\tr )=2\log T_{\rm abs, sing}(C_l W ).
\eeq

For we need the following lemma.
\begin{lem}\label{formul} We have the following equations:
\begin{align*}
l_{j,\pm\alpha_q}(\lambda) &=  l_j(l_2^2 \lambda) + (-1)^j  l_j^{\mp}(l_1^2\lambda),\\
\hat l_{j,\pm\alpha_q}(\lambda) &=  l_j^{\pm}(l_2^2 \lambda) + (-1)^j l_j(l_1^2\lambda),\\
l_{j,0}(\lambda) &=  l_j(l_2^2 \lambda) + (-1)^j \dot l_j(l_1^2\lambda),\\
\hat l_{j,0}(\lambda) &= \dot l_j(l_2^2 \lambda) + (-1)^j l_j(l_1^2\lambda),\\ 
\end{align*}
where the functions $l_j$, $\dot l_j$ are defined in the proof of Lemma \ref{s3.l2}, the functions $l^\pm_j$ in the proof of Lemma \ref{s4.l2}, and the other function in Subsection \ref{ora}.
\end{lem}
\begin{proof} The proof is by induction. We give details for the first equation.  For $j=1$, we have
\begin{align*}
l_{1,\pm\alpha_q}(\lambda) &= U_1(l_2\sqrt{-\lambda}) - W_{\mp,1}(l_1\sqrt{-\lambda})\\
&= l_1(l_2^2\lambda) + (-1)^1 l_1(l_1^2\sqrt{-\lambda}).
\end{align*} 

Assume the equation is valid for all $n<j$. By definition
\begin{align*}
l_{j,\pm\alpha_q}(\lambda) -& \sum_{k=1}^{j-1} U_k(l_2\sqrt{-\lambda})(-1)^{j-k}W_{\mp\alpha_q,j-k}(l_1\sqrt{-\lambda}) \\
&=U_j(l_2\sqrt{-\lambda})+(-1)^j W_{\mp\alpha_q,j}(l_1\sqrt{-\lambda}) - \sum_{k=1}^{j-1}\frac{j-k}{j} a_{k,\mp\alpha_q}(\lambda) l_{j-k,\mp\alpha_q}(\lambda)
\end{align*}
and using the inductive hypothesis for $l_{j-k,\mp\alpha_q}(\lambda)$, and collecting similar terms, this gives
\begin{align*}
l_{j,\pm\alpha_q}(\lambda) -& \sum_{k=1}^{j-1} U_k(l_2\sqrt{-\lambda})(-1)^{j-k}W_{\mp\alpha_q,j-k}(l_1\sqrt{-\lambda})\\
=&  l_j(l_2^2 \lambda) + (-1)^j  l_j^{\mp}(l_1^2\lambda) - \sum_{k=1}^{j-1}\frac{j-k}{j} (-1)^k W_{\mp\alpha_q,k}(l_1\sqrt{-\lambda})l_{j-k}(l_2^2 \lambda) \\
&- \sum_{k=1}^{j-1}\frac{j-k}{j} (U_k(l_2\sqrt{-\lambda}))(-1)^{j-k}  l_{j-k}^{\mp}(l_1^2\lambda)\\
&- \sum_{k=1}^{j-1}\frac{j-k}{j} \sum_{h=1}^{k-1} U_h(l_2\sqrt{-\lambda})(-1)^{k-h}W_{\mp\alpha_q,k-h}(l_1\sqrt{-\lambda})l_{j-k}(l_2^2 \lambda) \\
&- \sum_{k=1}1^{j-1}\frac{j-k}{j} \sum_{h=1}^{k-1} U_h(l_2\sqrt{-\lambda})(-1)^{k-h}W_{\mp\alpha_q,k-h}(l_1\sqrt{-\lambda})) (-1)^{j-k}  l_{j-k}^{\mp}(l_1^2\lambda).
\end{align*}

Rearranging the summation's indices, this reads
\begin{align*}
l_{j,\pm\alpha_q}(\lambda) -& \sum_{k=1}^{j-1} U_k(l_2\sqrt{-\lambda})(-1)^{j-k}W_{\mp\alpha_q,j-k}(l_1\sqrt{-\lambda})\\
&=  l_j(l_2^2 \lambda) + (-1)^j  l_j^{\mp}(l_1^2\lambda) - \sum_{k=1}^{j-1} (-1)^k W_{\mp\alpha_q,k}(l_1\sqrt{-\lambda})U_{j-k}(l_2 \sqrt{-\lambda}) \\
&+\sum_{k=1}^{j-1}(-1)^{j-k} W_{\mp\alpha_q,j-k}(l_1\sqrt{-\lambda})\sum_{h=1}^{k-1}\frac{k-h}{j} U_{h}(l_2\sqrt{-\lambda}) l_{k-h}(l_2^2\lambda)\\
&+\sum_{k=1}^{j-1}(-1)^k U_{j-k}(l_2\sqrt{-\lambda})\sum_{h=1}^{k-1}\frac{h}{j} W_{\mp\alpha_q,k-h}(l_1\sqrt{-\lambda}) l_{h}^{\mp}(l_1^2\lambda)\\
&- \sum_{k=1}^{j-1}\frac{j-k}{j} l_{j-k}(l_2^2 \lambda) \sum_{h=1}^{k-1} U_h(l_2\sqrt{-\lambda})(-1)^{k-h}W_{\mp\alpha_q,k-h}(l_1\sqrt{-\lambda}) \\
&- \sum_{k=1}^{j-1}\frac{j-k}{j}(-1)^{j-k}  l_{j-k}^{\mp}(l_1^2\lambda) \sum_{h=1}^{k-1} U_h(l_2\sqrt{-\lambda})(-1)^{k-h}W_{\mp\alpha_q,k-h}(l_1\sqrt{-\lambda}) 
\end{align*}

Reordering the first two double sums as
\begin{align*}
\sum_{k=1}^{j-1}(-1)^{j-k} &W_{\mp\alpha_q,j-k}(l_1\sqrt{-\lambda})\sum_{h=1}^{k-1}\frac{k-h}{j} U_{h}(l_2\sqrt{-\lambda}) l_{k-h}(l_2^2\lambda)\\
&= \sum_{k=1}^{j-1}\frac{j-k}{j} l_{j-k}(l_2^2 \lambda) \sum_{h=1}^{k-1} U_h(l_2\sqrt{-\lambda})(-1)^{k-h}W_{\mp\alpha_q,k-h}(l_1\sqrt{-\lambda}),\\
\sum_{k=1}^{j-1}(-1)^k &U_{j-k}(l_2\sqrt{-\lambda})\sum_{h=1}^{k-1}\frac{k-h}{j} W_{\mp\alpha_q,h}(l_1\sqrt{-\lambda}) l_{k-h}^{\mp}(l_1^2\lambda)\\
& = \sum_{k=1}^{j-1}\frac{j-k}{j}(-1)^{j-k}  l_{j-k}^{\mp}(l_1^2\lambda) \sum_{h=1}^{k-1} U_h(l_2\sqrt{-\lambda})(-1)^{k-h}W_{\mp\alpha_q,k-h}(l_1\sqrt{-\lambda}),
\end{align*}
the result follows.
\end{proof}

We are now in the position of proving equation (\ref{sos}). Proceeding as in the proof of Propositions \ref{pro2} and \ref{pro3}, the singular part of the torsion is given by some residua of the zeta function associated to the sequence $U$ and some residua of the functions $\Phi$. Since the sequence $U$ is the same for the conical frustum and for the cone, and the range of the indices are the same, we only need to compare the functions $\Phi$ in the two cases. The functions $\Psi$ are defined in equation(\ref{fi1}), we introduce the linear operation
\[
\Phi_{\sigma_h}(s)=\mathcal{T}(\phi_{\sigma_h}(\lambda))(s)=\int_0^\infty t^{s-1}\frac{1}{2\pi i}\int_{\Lambda_{\theta,c}}\frac{\e^{-\lambda t}}{-\lambda} \phi_{\sigma_h}(\lambda) d\lambda dt.
\]

Let use the notation $\phi^{\rm cone}$ and $\phi^{\rm frust}$. We have
\begin{align*}
\phi^{\rm cone}_{2j-1,q}=-2l_{2j-1}(\lambda)+l^+_{2j-1}(\lambda)+l^-_{2j-1}(\lambda),\\
\phi^{\rm frust}_{j,q}=-l_{j,+}(l_2^2\lambda)-l_{j,-}(l_1^2\lambda)+\hat l_{j,+}(l_2^2\lambda)+\hat l_{j,-}(l_1^2\lambda).
\end{align*}

Note that all the functions appearing in the definition of the functions $\phi(\lambda)$ are polynomial in $w=\frac{1}{\sqrt{1-\lambda}}$. Applying the formula in equation (\ref{e}), we have that
\[
\T (l_{j+}(l_2^2\lambda))(s)=l_2^{2s} \T (l_{j+}(l_2^2\lambda))(s),
\]
and similarly for the other. Using Lemma \ref{ora}, and odd indices, we obtain for example
\begin{align*}
\Phi_{2j-1,q}^{\rm frust}(s)= (l_1^{2s}+l_2^{2s}) \Phi_{2j-1,q}^{\rm cone}(s)
\end{align*}

Since by Corollaries \ref{c33} and \ref{c44} all the residua $\Ru$ of the function $\Phi^{\rm cone}_{2j-1}(s)$ at $s=0$ vanish, equation \ref{sos} follows.

\subsection{Conclusion}
As recalled  in Section \ref{cm}, if $\b W=\b_1 W\sqcup \b_2 W$, is the union of two disjoint components, and since the boundary term is local, 
\[
\log T_{{\rm rel} ~\b_1, {\rm abs}~ \b_2}((W,g);\rho)=\log \tau(((W,\b_1 W),g);\rho)+A_{\rm BM, rel}(\b_1 W)+A_{\rm BM, abs}(\b_2 W).
\]

Applying this formula to the conical frustum we have
\[
\log T_{{\rm rel} ~\b_1, {\rm abs}~ \b_2}(\tr)=\log \tau(\tr,\b_1 \tr)+A_{\rm BM, rel}(\b_1 )+A_{\rm BM, abs}(\b_2 ).
\]

Let $X$ be a manifold of dimension $2p$ with boundary $\b X=\b_2 \tr$, and assume there is an isometry of a collar neighborhood of the boundary of $X$ onto a collar neighborhood of $\b_2 \tr$. Let $Z$ be the manifold obtained by glueing smoothly $X$ to $\tr$ along the boundary $\b_2 \tr$. Applying duality of analytic torsion \cite{Luc} Proposition 2.10 to $Z$, and since the anomaly boundary term is local, it follows that $A_{\rm BM, rel}(\b_1 \tr)=-A_{\rm BM, abs}(\b_1 \tr)$. Since it follows by the definition that $A_{\rm BM, abs}(\b_1 \tr)=-A_{\rm BM, abs}(\b_2 \tr)$, we obtain
\[
\log T_{{\rm rel} ~\b_1, {\rm abs}~ \b_2}(\tr)=\log \tau(\tr,\b_1 \tr)+2A_{\rm BM, abs}(\b_2 \tr).
\]

Considering the exact sequence of chain complex associated to the pair $(\tr,\b_1 \tr)$,  it is not difficult to see (see for example \cite{Mil} Section 3) that  the Reidemeister torsion of the pair vanishes, and hence 
\[
\log T_{{\rm rel} ~\b_1, {\rm abs}~ \b_2}(\tr)= 2A_{\rm BM, abs}(\b_2 \tr).
\]

Since the anomaly boundary term is local $A_{\rm BM, abs}(\b_2 \tr)=A_{\rm BM,abs}(\b C_l W)$, and hence
\[
\log T_{{\rm rel} ~\b_1, {\rm abs}~ \b_2}(\tr)= 2A_{\rm BM,abs}(\b C_l W).
\]

This formulas also follows using the formulas for the variation of the torsion with mixed boundary conditions given in the new paper of Br\"{u}ning and Ma \cite{BM2}. We thanks the authors for making available to us this part of the results of their still unpublished paper. Since by  the calculations of the previous subsections
\[
\log T_{{\rm rel} ~\b_1, {\rm abs}~ \b_2}(\tr)=\log T_{{\rm rel} ~\b_1, {\rm abs}~ \b_2, {\rm sing}}(\tr) =2\log T_{\rm abs, sing}(C_l W )=2S(\b C_l W),
\]
this completes the proof of Theorem \ref{t03}.

\newpage

\vskip 20pt
\centerline{\bf Appendix}
\vskip 10pt

The next two formulas follow from the definition of the Euler Gamma function. Here $j$ is any positive integer.
\beq
\label{residuo Gamma}
\begin{aligned}
\Rz_{s=0} \frac{\Gamma\left(s+\frac{2j+1}{2}\right)}{\Gamma\left(\frac{2j+1}{2}\right)s} &= -\gamma -2\log 2
+2\sum^{j}_{k=1}\frac{1}{2k-1},\\
\Ru_{s=0} \frac{\Gamma(s+\frac{2j+1}{2})}{\Gamma(\frac{2j-1}{2})s}& =1,
\end{aligned}
\eeq

The next formula is proved in \cite{Spr3} Section 4.2. Let
$\Lambda_{\theta,c}=\{\lambda\in\C~|~|\arg(\lambda-c)|=\theta\}$,
$0<\theta<\pi$, $0<c<1$, $a$ real, then
\beq\label{e}
\int_0^\infty t^{s-1} \frac{1}{2\pi
i}\int_{\Lambda_{\theta,c}}\frac{\e^{-\lambda
t}}{-\lambda}\frac{1}{(1-\lambda)^a}d\lambda
dt=\frac{\Gamma(s+a)}{\Gamma(a)s}.
\eeq


\begin{thebibliography}{99}

\bibitem{BGV} N. Berline, E. Getzler, M. Verge, {\em Heat Kernels and Dirac Operators}, Grundlehren Math. Wiss. 298, Springer-Verlag, 1992.

\bibitem{BZ} J.-M. Bismut and W. Zhang, {\em An extension of a theorem by Cheeger and M\"{u}ller}, Ast\'{e}risque 205 (1992) 1-235.

\bibitem{BM} J. Br\"uning and Xiaonan Ma, {\em An anomaly formula for Ray-Singer metrics on manifolds with boundary}, GAFA 16 (2006) 767-873.

\bibitem{BM2} J. Br\"uning and Xiaonan Ma, {\em On the gluing formula for the analytic torsion}, to appear.

\bibitem{BS1} J. Br\"uning and R. Seeley, {\em Regular Singular Asymptotics}, Adv. in Math. 58 (1985), 133-148.

\bibitem{BS2} J. Br\"uning and R. Seeley, {\em The resolvent expansion for
second order regular singular operators}, J. of Funct. An. 73 (1988) 369-415.


\bibitem{Che1} J. Cheeger, {\em Analytic torsion and the heat equation}, Ann. Math. 109 (1979) 259-322.

\bibitem{Che0} J. Cheeger, {\em On the spectral geometry of spaces with conical singularities}, Proc. Nat. Acad. Sci. 76 (1979) 2103-2106.

\bibitem{Che2} J. Cheeger, {\em Spectral geometry of singular Riemannian spaces}, J. Diff. Geom. 18 (1983) 575-657.

\bibitem{Che3} J. Cheeger, {\em On the Hodge theory of Riemannian pseudomanifolds}, Proc. Sympos. Pure Math. 36 (1980) 91-146.




\bibitem{Dar} A. Dar, {\it Intersection R-torsion and the analytic torsion for pseudomanifolds}, Math. Z. 154 (1987) 155-210.




\bibitem{deR} G. de Rham, {\em Complexes \`{a} automorphismes et hom\'{e}omorphie diff\'{e}rentiables}, Ann. Inst. Fourier, Grenoble 2, 51-67.

\bibitem{Fra} W. Franz, \emph{\"Uber die Torsion einer \"Uberdeckung}, J. Reine Angew. Math. 173 (1935) 245-254.

\bibitem{Gil} P.B. Gilkey, {\em Invariance theory, the heat equation, and the Atiyah-Singer index theorem}, Studies in Advanced Mathematics, 1995.

\bibitem{GZ} I.S. Gradshteyn and I.M. Ryzhik, {\em Table of integrals, Series and Products}, Academic Press, 2007.


\bibitem{GR} L. Graham, D.R. Knuth, O. Patashnik, {\em Concrete Mathetamics: a fundation for computer science},
Addison Wesley,  1994.

\bibitem{Har} L. Hartmann, PhD Thesis University of S\~ao Paulo, 2009.

\bibitem{HS} L. Hartmann and M. Spreafico, to appear in J. Math. Pure Ap. (2009) doi:10.1016/j.matpur.2009.11.001.

\bibitem{HMS} L. Hartmann, T. de Melo and M. Spreafico, {\em Reidemeister torsion and analytic torsion of discs}, (2008)
arXiv:0811.3196v1.




\bibitem{IT} A. Ikeda and Y. Taniguchi, {\it Spectra and eigenforms of the Laplacian on $S\sp{n}$ and $P\sp{n}(\C)$},
Osaka J. Math. 15 (1978) 515-546.


\bibitem{Kra} V.A. Krechmar, {\it A problem book in algebra}, Mir, 1974.


\bibitem{Luc} W. L\"uck,  \emph{Analytic and topological torsion for manifolds with boundary and symmetry},  J. Differential Geom.  37   (1993) 263-322.



\bibitem{MS} T. de Melo and M. Spreafico {\em Reidemeister torsion and analytic torsion of spheres},  J. Homotopy Relat. Struct. 4 (2009) 181-185.

\bibitem{Mil} J. Milnor, {\em Whitehead torsion}, Bull. AMS 72 (1966) 358-426.

\bibitem{Mor} S. Morita, {\em Geometry of Differential Forms}, Translations of Mathematical Monographs AMS 201,
2001.

\bibitem{Mul} W. M\"{u}ller, {\em Analytic torsion and R-torsion of Riemannian manifolds}, Adv. Math. 28 (1978) 233-305.


\bibitem{Nag} M. Nagase, {\em De Rham-Hodge theory on a manifold with cone-like singularities}, Kodai Math. J., 1 (1982) 38-64.

\bibitem{Olv} F.W.J. Olver, {\em Asymptotics and special functions}, AKP, 1997.

\bibitem{RS} D.B. Ray and I.M. Singer, {\em R-torsion and the Laplacian
on Riemannian manifolds}, Adv. Math. 7 (1971) 145-210.

\bibitem{Ray} D.B. Ray, {\em Reidemeister torsion and the Laplacian on lens spaces}, Adv. Math. 4 (1970) 109-126.

\bibitem{Rei} K. Reidemeister, {\em Homotopieringe und Linser\"aume}, Hamburger Abhandl. 11 (1935) 102-109.

\bibitem{Rel} F. Rellich, {\em Die zul\"{a}ssigen Randbedingungen bei den singul\"{a}ren Eigenwert-problem der mathematischen Physik}, Math. Z. 49 (1943/44) 702-723.

\bibitem{Spr0} M. Spreafico, {\em Zeta function and regularized determinant on projective spaces}, Rocky Mount. Jour. Math. 33 (2003) 1499-1512.

\bibitem{Spr1} M. Spreafico, {\em On the non homogeneous quadratic Bessel zeta function}, Mathematika 51 (2004) 123-130.

\bibitem{Spr3} M. Spreafico, {\em Zeta function and regularized determinant on a disc and on a cone}, J. Geo. Phys. 54 (2005) 355-371.

\bibitem{Spr4} M. Spreafico, {\em Zeta invariants for sequences of spectral type, special functions and the Lerch formula}, Proc. Roy. Soc. Edinburgh 136A (2006) 863-887.

\bibitem{Spr5} M. Spreafico, {\em Zeta invariants for Dirichlet series}, Pacific. J. Math. 224 (2006) 180-199.

\bibitem{Spr6} M. Spreafico, {\it Determinant for the Laplacian on forms and a torsion type invariant on a cone over the circle}, Far East J. Math. Sc. 29 (2008) 353-368.

\bibitem{Spr9} M. Spreafico, {\it Zeta invariants for double sequences of spectral type},  arXiv:math/0607816.

\bibitem{Ver} B. Vertman, \emph{Analytic Torsion of a Bounded Generalized Cone}, Comm. Math. Phys. 290 (2009) 813-860.

\bibitem{Wat} G.N. Watson, {\em A treatise on the theory of Bessel
functions}, Cambridge University Press, 1922.

\bibitem{WY} L. Weng and Y. You, {\em Analytic torsions of spheres}, Int. J. Math. 7 (1996) 109-125.




%
%
%
%
%
%
%


\end{thebibliography}
\end{document}